\documentclass{article}
\usepackage[latin1]{inputenc}
\usepackage{amsmath}
\usepackage{amsthm,amssymb,amsfonts}
\usepackage{graphicx}
\usepackage{verbatim}
\usepackage{color,cite}
\usepackage{latexsym}
\usepackage{lscape}
\usepackage[T1]{fontenc}
\usepackage[all,cmtip]{xy}
\usepackage{diagmac}
\theoremstyle{plain}
\newtheorem{theorem}{Theorem}[section]
\newtheorem{lemma}[theorem]{Lemma}
\newtheorem{proposition}[theorem]{Proposition}
\newtheorem{corollary}[theorem]{Corollary}
\newtheorem*{conjecture}{Conjecture}
\newtheorem*{definition}{Definition}

\newtheorem*{question}{Question}
\newtheorem*{problem}{Problem}
\newtheorem*{remark}{Remark}
\newtheorem*{condthm}{Conditional Theorem}
\newtheorem*{thm:Dg1FPImpCP}{Theorem \ref{thm:Dg1FPImpCP}}

\title{Teichm\"uller Discs with Completely Degenerate Kontsevich-Zorich Spectrum}
\author{David Aulicino\thanks{Partially supported by National Science Foundation grant DMS - 0800673.}}
\date{}

\begin{document}

\newcommand{\splin}{$\text{SL}_2(\mathbb{R})$}
\newcommand{\spolin}{$\text{SO}_2(\mathbb{R})$}
\newcommand{\TeichDisk}{$\Phi(\text{SL}_2(\mathbb{R})(X,\omega))$}
\newcommand{\RankOne}{$\mathcal{D}_g (1)$}

\maketitle

\tableofcontents

\newpage

\begin{abstract}
We reduce a question of Eskin-Kontsevich-Zorich and Forni-Matheus-Zorich, which asks for a classification of all \splin -invariant ergodic probability measures with completely degenerate Kontsevich-Zorich spectrum, to a conjecture of M\"oller's.  Let \RankOne ~be the subset of the moduli space of Abelian differentials $\mathcal{M}_g$ whose elements have period matrix derivative of rank one.  There is an \splin -invariant ergodic probability measure $\nu$ with completely degenerate Kontsevich-Zorich spectrum, i.e. $\lambda_1 = 1 > \lambda_2 = \cdots = \lambda_g = 0$, if and only if $\nu$ has support contained in \RankOne .  We approach this problem by studying Teichm\"uller discs contained in \RankOne .  We show that if $(X,\omega)$ generates a Teichm\"uller disc in \RankOne , then $(X,\omega)$ is completely periodic.  Furthermore, we show that there are no Teichm\"uller discs in \RankOne , for $g = 2$ and $g \geq 13$, and the two known examples of Teichm\"uller discs in \RankOne , for $g = 3, 4$, are the only two such discs in those genera.  Finally, we prove that if there are no genus five Veech surfaces generating Teichm\"uller discs in $\mathcal{D}_5(1)$, then there are no Teichm\"uller discs in \RankOne , for $g \geq 5$.
\end{abstract}

\section{Introduction}

In \cite{KZHodge}, Kontsevich and Zorich introduced the Kontsevich-Zorich cocycle as a cocycle on the Hodge bundle over the moduli space of Riemann surfaces, denoted $G_t^{KZ}$, which is a continuous time version of the Rauzy-Veech-Zorich cocycle.  They showed that this cocycle has a spectrum of $2g$ Lyapunov exponents with the property
$$1 = \lambda_1 \geq \lambda_2 \geq \cdots \geq \lambda_g \geq -\lambda_g \geq \cdots \geq -\lambda_2 \geq -\lambda_1 = -1.$$
These exponents have strong implications about the dynamics of flows on Riemann surfaces, interval exchange transformations, rational billiards, and related systems.  They also describe how generic trajectories of an Abelian differential distribute over a surface \cite{Zorich1Form}.  Furthermore, Zorich \cite{Zorich1Form} proved that they fully describe the non-trivial exponents of the Teichm\"uller geodesic flow, denoted $G_t$.  Veech \cite{VeechGt} proved $\lambda_2 < 1$, which implies that $G_t$ is non-uniformly hyperbolic.  Since then, the study of the Lyapunov spectrum of the Kontsevich-Zorich cocycle has become of widespread interest.  Forni \cite{ForniDev} proved the first part of the Kontsevich-Zorich conjecture \cite{KZHodge}: $\lambda_g > 0$ for the canonical \splin -invariant ergodic measure in the moduli space of holomorphic quadratic differentials.  His result implies $G_t^{KZ}$ is also non-uniformly hyperbolic.  Avila and Viana \cite{AvilaVianaSimp} then used independent techniques to show that the spectrum is simple for the canonical measures on the strata of Abelian differentials, i.e. $\lambda_k > \lambda_{k+1}$, for all $k$.

Throughout this paper, the spectrum of Lyapunov exponents of the Kontsevich - Zorich cocycle will be referred to as the Kontsevich-Zorich spectrum (KZ-spectrum).  Veech asked to what extent the KZ-spectrum could be degenerate.  Forni \cite{ForniHand} found an example of an \splin -invariant measure supported on the Teichm\"uller disc of a genus three surface with completely degenerate KZ-spectrum, i.e. $\lambda_1 = 1 > \lambda_2 = \lambda_3 = 0$.  In the literature, the genus three surface generating Forni's example, denoted here by $(M_3, \omega_{M_3})$, is known as the Eierlegende Wollmilchsau for its numerous remarkable properties \cite{HerrlichSchmithusenEier}.  Forni and Matheus \cite{ForniMatheus} then found an example generated by a genus four surface, denoted here by $(M_4, \omega_{M_4})$, with $\lambda_1 = 1 > \lambda_2 = \lambda_3 = \lambda_4 = 0$.  Both surfaces are Veech surfaces and in particular, square tiled cyclic covers.  They will be defined and depicted in Section \ref{PctdVeechSurfSect}.  By relating Teichm\"uller and Shimura curves, M\"oller \cite{MollerShimuraTeich} proved that these two examples are the only examples of Veech surfaces generating Teichm\"uller discs supporting a measure with completely degenerate KZ-spectrum except for possible examples in certain strata of Abelian differentials in genus five.  In a paper of Forni, Matheus and Zorich \cite{ForniMatheusZorichSqTiled}, they proved that the two examples are the only square-tiled cyclic cover surfaces generating Teichm\"uller discs supporting a measure with completely degenerate KZ-spectrum.  In the recent work of \cite{EskinKontsevichZorich2}, it was shown that there are no regular \splin -invariant suborbifolds with completely degenerate Kontsevich-Zorich spectrum for $g \geq 7$.  It was recently announced by Eskin and Mirzakhani \cite{EskinMirzakhaniInvariantMeas}, that the closure of every Teichm\"uller disc is an \splin -invariant suborbifold.  However, it is an open conjecture that every \splin -invariant suborbifold is regular \cite{EskinKontsevichZorich2}[Section 1.5].  Hence, the result of \cite{EskinKontsevichZorich2}[Corollary 5] does not yet imply some of the results of this paper.

Both \cite{EskinKontsevichZorich2} and \cite{ForniMatheusZorichSqTiled} asked if the two known examples generate the only Teichm\"uller discs whose closures support an \splin -invariant ergodic probability measure with completely degenerate Kontsevich-Zorich spectrum.  In this paper we give a nearly complete answer to this question by reducing the entire problem to a conjecture of M\"oller that claims there are no Veech surfaces in genus five that generate a Teichm\"uller disc with this property.  Let \RankOne ~denote the subset of the moduli space of Abelian differentials, where the derivative of the period matrix has rank one.  We address a potentially stronger problem and ask for a classification of all Teichm\"uller discs in \RankOne .

\begin{theorem}
\label{RankOnePartialClass}
There are no Teichm\"uller discs in \RankOne , for $g = 2$ and $g \geq 13$.  The surface $(M_3, \omega_{M_3})$ generates the only Teichm\"uller disc in $\mathcal{D}_3(1)$ and $(M_4, \omega_{M_4})$ generates the only Teichm\"uller disc in $\mathcal{D}_4(1)$.  Furthermore, if there are no Teichm\"uller curves in $\mathcal{D}_5(1)$, then there are no Teichm\"uller discs in \RankOne , for $g \geq 5$.
\end{theorem}

The main techniques used in this paper include degenerating surfaces under the Deligne-Mumford compactification of the moduli space of Riemann surfaces, and an analysis of the derivative of the period matrix under such deformations.  This concept has already been used successfully in \cite{ForniDev}.  Several other authors have also used this concept in other guises such as the second fundamental form of the Hodge bundle \cite{ForniMatheusZorichLyapSpectHodge} and the Kodaira-Spencer map in the work of M\"oller and his coauthors \cite{MollerShimuraTeich, BainbridgeMoller, ChenMollerNonVarySums}.

To prove this theorem we show first that any surface generating a Teichm\"uller disc in \RankOne ~is completely periodic, cf. Theorem \ref{RankOneImpCP}.  Then we show that degenerating surfaces in the closure of a Teichm\"uller disc in \RankOne ~must have a very specific configuration, cf. Lemma \ref{TDDerPerRank1Lem}.  Proving the results requires some technical lemmas demonstrating convergence of the derivative of the period matrix, cf. Section \ref{DerPerMatSubSect}, and a variation of a theorem of Masur \cite{MasurClosedTraj}[Theorem 2] to a more general setting, cf. Lemma \ref{MasurThmwithParSlits}.  These results quickly yield some applications, cf. Proposition \ref{Genus2Cor}.

Next we show that the closure of every Teichm\"uller disc in \RankOne ~must contain a (possibly degenerate) surface that is a Veech surface, cf. Theorem \ref{CPImpVeech}.  This leads to an analysis of punctures on a Veech surface with the goal of excluding more and more configurations of the punctures until the remainder of the results follow.  Theorem \ref{RankOnePartialClass} summarizes Proposition \ref{Genus2Cor}, Theorem \ref{Genus3Classification}, Corollary \ref{NoTeichDiskSuffHighGenus}, Theorem \ref{Gen4Class}, and the Conditional Theorem in Section \ref{RegPtsMainThm}.

\

\noindent \textbf{Acknowledgments:}  This work was done in partial fulfillment of the requirements for the Ph.D. at the University of Maryland -- College Park.  The author would like to express his sincere gratitude to his advisor Giovanni Forni, for his patience and the generosity of his time throughout the entirety of the research process.  The author would also like to thank Scott Wolpert for his insight into the geometry of the moduli space, especially concerning the content of Section \ref{PctsAtRegPtsSect}.  The author is grateful to Martin M\"oller for expressing interest in his work at an early stage.  Finally, the author would like to thank Matt Bainbridge and Kasra Rafi for their helpful discussions.

\section{Preliminaries}

\subsection{The Moduli Space of Riemann Surfaces}

Let $X$ be a Riemann surface of genus $g$ with $n$ punctures (i.e. marked points).  Let $R(X)$ denote the \emph{Teichm\"uller space of $X$} or simply $R_{g,n}$ when $X$ is understood.  The surface $X$ admits a \emph{pants decomposition}, $X = \mathcal{P}_1 \cup \cdots \cup \mathcal{P}_{3g-3+n}$, into $3g-3+n$ pairs of pants, where each pair of pants is homeomorphic to the sphere with a total of three punctures and disjoint boundary curves.  The Fenchel-Nielsen coordinates for Teichm\"uller space describe surfaces in terms of the lengths and twists of curves in a pants decomposition of $X$.  A point in Teichm\"uller space is given by $(\ell_1, \ldots, \ell_{3g-3+n}, \theta_1, \ldots, \theta_{3g-3+n}) \in \mathbb{R}_+^{3g-3+n} \times \mathbb{R}^{3g-3+n}$.

Let $\text{Diff}^+(X)$ be the group of orientation preserving diffeomorphisms on $X$.  Let $\text{Diff}_0^+(X)$ denote the normal subgroup of $\text{Diff}^+(X)$ whose elements are isotopic to the identity.  Then the \emph{mapping class group} is the quotient
$$\Gamma(X) = \text{Diff}^+(X)/\text{Diff}_0^+(X).$$
The \emph{moduli space of genus $g$ surfaces with $n$ punctures} is defined to be
$$\mathcal{R}_{g,n} = R(X)/ \Gamma(X).$$

Deligne and Mumford \cite{DeligneMumford} introduced a compactification of the moduli space denoted $\overline{\mathcal{R}_{g,n}}$ of Riemann surfaces within the more general setting of compactifying the space of stable curves.  Every neighborhood of a point on a \emph{Riemann surface with nodes} is either conformally equivalent to the unit complex disc, or to the set $\{(x,y) \in \mathbb{C}^2 \vert xy = 0\}$.  The point mapped to $(0,0)$ with the latter property is called a \emph{node}.  We regard this as the contraction or pinching of a simple closed curve on a surface to a point.  Removing a node results in two punctures on either side of the node.  This may or may not disconnect the surface.  After removing all nodes, each of the connected components of the punctured degenerate surface is called a \emph{part}.  A \emph{pair of punctures}, denoted $(p,p')$, will specifically refer to the punctures created by removing a node.  We will assume this deconstruction throughout and say that pinching a curve results in a pair of punctures unless we say otherwise.  Theorem B.1 in Appendix B of \cite{ImaTan} describes the compactification of the moduli space in terms of the Fenchel-Nielsen coordinates (or equivalently, a choice of pants decomposition) for Teichm\"uller space.  By \cite{ImaTan}[Theorem B.1], the boundary of the moduli space $\overline{\mathcal{R}_{g,n}}$ under the Deligne-Mumford compactification is given by letting one or more of the lengths $\ell_i$ in the Fenchel-Nielsen coordinates be zero.

\subsection{Abelian and Quadratic Differentials}

\subsubsection{Abelian Differentials}

Let $K$ be the cotangent bundle over $X$.  A section $\omega$ of $K$ is a complex 1-form called an \emph{Abelian differential}.  An Abelian differential $\omega$ on $X$ is given in local coordinates by $\omega = \phi(z)\,dz$, where $\phi(z)$ is a holomorphic function on the punctured surface possibly having poles of finite order at the punctures.  Furthermore, $\omega$ obeys the change of coordinates formula
$$\phi(\sigma(z))\,d\sigma(z) = \phi(\sigma(z)) \sigma'(z) \,dz.$$
The zeros and poles of $\omega$ are called \emph{singularities} and all other points are called \emph{regular}.  The Chern formula relates the total number of zeros and poles counting multiplicity, by 
$$\sharp (\text{zeros}) - \sharp (\text{poles}) = 2g-2.$$

An Abelian differential $\omega$ determines an orientable horizontal and vertical foliation of a surface given by $\{\Im(\omega) = 0\}$ and $\{\Re(\omega) = 0\}$, respectively.  Equivalently, the foliations can be defined by a pullback of the horizontal and vertical lines in the complex plane under the local coordinate chart on the surface.  The Abelian differential $\omega$ determines a flat structure on the surface away from the singularities.  A maximal connected subset of a foliation is called a \emph{leaf}.  If a leaf is compact and it does not pass through a singularity of $\omega$, then it is called a \emph{closed regular trajectory}.  A closed connected subset $\sigma$ of a leaf with endpoints at zeros of $\omega$ whose interior consists entirely of regular points of $\omega$ is called a \emph{saddle connection}.  Given a closed regular trajectory $\gamma$, the closure of the maximal set of parallel closed regular trajectories homotopic to $\gamma$ form a \emph{cylinder}.  By definition, the boundaries of a cylinder consist of a union of saddle connections.  We say that two cylinders are \emph{homologous} (resp. \emph{parallel}) if their core curves are homologous (resp. parallel).  If every leaf of a foliation is compact, the foliation is \emph{periodic}.

\begin{lemma}
\label{HomCylImpParLem}
If $C_1$ and $C_2$ are homologous cylinders on a surface $(X,\omega)$, then $C_1$ and $C_2$ are parallel.
\end{lemma}

\begin{proof}
Let $\gamma_1$ and $\gamma_2$ be the core curves of $C_1$ and $C_2$, respectively.  Without loss of generality, assume that $\gamma_1$ is a closed curve of the vertical foliation on $X$ by $\omega$.  Then by the definition of homologous
$$\int_{\gamma_1} \omega = \int_{\gamma_2} \omega.$$
Thus
$$\int_{\gamma_1} \omega = \int_{\gamma_1} \Re(\omega) + i\Im(\omega) = i\int_{\gamma_1} \Im(\omega).$$
The last equality follows because $\gamma_1$ lies exactly in the vertical foliation so it has no horizontal holonomy.  However, this implies
$$\int_{\gamma_2} \omega = i\int_{\gamma_1} \Im(\omega),$$
which implies
$$\int_{\gamma_2} \Re(\omega) = 0.$$
Therefore, $\gamma_2$ has no horizontal holonomy either, so it must be parallel to $\gamma_1$.
\end{proof}

Call $\phi(z)$ or $\omega$ \emph{holomorphic} if it can be continued holomorphically across all punctures of $X$.  When $\phi(z)$ is holomorphic it naturally determines a flat metric on the surface.  The length of a curve $\gamma$ in this metric is given by 
$$\int_{\gamma} |\phi(z)dz|.$$
Furthermore, there is an area form given by
$$A(\omega) = \frac{i}{2}\int_{X} \omega \wedge \bar \omega.$$
In the case of meromorphic differentials, the metric is still defined on compact subsets away from the punctures at which the differential has a pole though the area form is infinite.

Let $T_{g,n}$ be the Teichm\"uller space of Riemann surfaces carrying Abelian differentials.  Define \emph{the moduli space of Abelian differentials on Riemann surfaces of genus $g$ with $n$ punctures} by $\mathcal{M}_{g,n} = T_{g,n}/\Gamma(X)$.  Define $\mathcal{M}_g := \mathcal{M}_{g,0}$ and $\mathcal{M}_{g,n}^{(1)} := \{(X,\omega) \in \mathcal{M}_g | A(\omega) = 1\}$.

Given a holomorphic differential $\omega$ on $X$, the sum of the orders of the zeros of $\omega$ is $2g-2$.  This determines a stratification of the moduli space of holomorphic differentials by the multiplicities of the zeros of the Abelian differential.  Denote the strata by $\mathcal{H}(\kappa)$, where $\kappa$ is a vector corresponding to a partition of $2g-2$.  In the case of meromorphic differentials, we list the orders of the poles in the vector $\kappa$ so that the sum of the components of the vector remains $2g-2$.

The moduli space of Abelian differentials can be expanded so that limits of convergent sequences of Abelian differentials lying on degenerating surfaces exist on nodal surfaces \cite{HarrisMorrison}.  An Abelian differential $\omega$ on a nodal Riemann surface is holomorphic everywhere except possibly at the punctures arising from removing the nodes, where $\omega$ is meromorphic with at most simple poles.  At each pair of punctures $(p,p')$, $\omega$ satisfies
$$\text{Res}_p(\omega) = -\text{Res}_{p'}(\omega).$$
Let $\overline{\mathcal{M}_g}$ denote the moduli space of meromorphic Abelian differentials over the compactified base space $\overline{\mathcal{R}_g}$.

There is a natural action by $\mathbb{R}^*$ on the bundle of Abelian differentials.  Let $r \in \mathbb{R}^*$ and $(X,\omega) \in \overline{\mathcal{M}_g}$, then
$$r \cdot (X,\omega) := (X, r\omega).$$
For the remainder of the paper, we abuse notation and assume that the moduli space $\mathcal{M}_g$ is always quotiented by $\mathbb{R}^*$ unless we say otherwise.  Furthermore, it will often be useful to choose a representative differential of the coset $(X,\omega)[\mathbb{R}^*]$.  For instance, if $\omega$ is holomorphic and nonzero, we may choose the representative so that its area form is one and if $\omega$ is not holomorphic, we may choose a representative such that the modulus of the largest residue is one.  This will be called \emph{area normalization} or \emph{residue normalization}, respectively.

The advantage of this projectivized moduli space of Abelian differentials is that it guarantees that for every sequence of Abelian differentials converging to an Abelian differential on a degenerate surface that there is at least one part of the degenerate surface on which the limiting Abelian differential is not identically zero.  Without the projectivization, no such guarantee can be made.  Let $\{(X_n, \omega_n)\}_{n=0}^{\infty}$ be a sequence of surfaces carrying holomorphic Abelian differentials converging to a degenerate surface $(X', \omega')$ in $\overline{\mathcal{M}_g}$.  Since $X_n$ has finite genus, there are finitely many pinching curves.  We can assume that $\omega_n$ is band bounded \cite{WolpertAnnuli}[Definition 1] on the annulus around each pinching curve.  If we multiply $\omega_n$ by $r_n$ so that the constant $M$ in the definition of band bounded is uniformly bounded away from zero and infinity for all $n$, then Lemma \ref{NonZeroPart} follows from \cite{WolpertAnnuli}[Lemma 2].

\begin{lemma}
\label{NonZeroPart}
Given a sequence $\{(X_n,\omega_n)\}_{n=0}^{\infty}$ such that the sequence $\{X_n\}_{n=0}^{\infty}$ converges to a degenerate surface $X'$, there exists an Abelian differential $\omega'$ on $X'$ such that $\omega'$ is the limit of the sequence $\{\omega_n\}_{n=0}^{\infty}$ in $\overline{\mathcal{M}_g}/\mathbb{R}^*$ and $\omega'$ is not identically zero on every part of $X'$.
\end{lemma}

\subsubsection{Quadratic Differentials}

Let $K$ be the cotangent bundle over $X$.  The sections of the bundle $K \otimes_{\mathbb{C}} K$ are complex 2-forms called \emph{quadratic differentials}.  A quadratic differential is given in local coordinates by $q = \phi(z)\,dz^2$ and obeys the change of coordinates formula
$$\phi(\sigma(z))\,d\sigma(z)^2 = \phi(\sigma(z)) (\sigma'(z))^2 \,dz^2.$$
Singularities and regular points are defined as before and in this case the Chern formula reads
$$\sharp (\text{zeros}) - \sharp (\text{poles}) = 4g-4.$$

A quadratic differential determines a horizontal and vertical foliation of a surface given by $\{\Im(\sqrt{\phi(z)}) = 0\}$ and $\{\Re(\sqrt{\phi(z)}) = 0\}$, respectively.  These foliations are not necessarily orientable.  If they are, $q$ is called an \emph{orientable quadratic differential}.  If a quadratic differential is holomorphic everywhere except for at most a finite set of simple poles, then it is called an \emph{integrable quadratic differential}.  Denote the \emph{Teichm\"uller space of integrable quadratic differentials} by $Q_{g,n}$ and the corresponding \emph{moduli space of integrable quadratic differentials} by $\mathcal{Q}_{g,n} := Q_{g,n}/\Gamma_{g,n}$.

There is a natural way of associating all quadratic differentials to Abelian differentials.  If $q$ is non-orientable, then there is a connected double covering $\pi: \hat X \rightarrow X$ defined as follows.  For each chart $U$ of $X$, let $q = \phi_U(z)\,dz^2$ and define two charts $V^{\pm}$ of $\hat X$ each of which maps homeomorphically to $U$ under $\pi$ and $V^{\pm}$ carry the local differentials $\pm\sqrt{\phi_U(z)}\,dz$.  This lift is compatible across charts and defines a quadratic differential $\omega^*$ with the property $\hat q = h^2$, where $h$ is an Abelian differential.  This lifting procedure is called the \emph{orientating double cover construction}, and it can be used to translate the terms defined for Abelian differentials above (metrics, etc.) to non-orientable quadratic differentials.

As above, the bundle of quadratic differentials can be extended to the boundary of the moduli space of Riemann surfaces as defined by the Deligne-Mumford compactification.  By admitting quadratic differentials with at most double poles, limits of sequences of integrable quadratic differentials on non-degenerate surfaces exist on degenerate surfaces.  Define the residue of a quadratic differential $q$ to be the coefficient of the term $1/z^2$ in its Taylor expansion.  Given a quadratic differential $q$ on a degenerate surface $X$ with a pair of punctures $(p,p')$, the residues of $q$ obey the relation
$$\text{Res}_p(q) = \text{Res}_{p'}(q).$$
Let $\overline{\mathcal{Q}_{g,n}}$ denote the moduli space of regular quadratic differentials on the compactified base space of Riemann surfaces $\overline{\mathcal{R}_{g,n}}$.

\subsection{The \splin ~Action}

We define the \splin ~action on quadratic differentials.  It is clear that this definition applies to Abelian differentials as well.  Let $q$ be an integrable quadratic differential.  Let $h$ (resp. $v$) denote the horizontal (resp. vertical) foliation of $q$.  The action by $A \in$ \splin ~on an integrable quadratic differential $q$ is defined by
$$\left[ \begin{array}{cc}
1 & i \end{array} \right] \left[ \begin{array}{cc}
a & b \\
c & d \end{array} \right] \left[ \begin{array}{c}
h\\
v \end{array} \right]$$
and denoted by $A \cdot (X,q)$.  The action is well-defined on and between charts of $X$.  Thus it defines an action by $A$ globally on $(X,q)$.  It was stated in \cite{BainbridgeMoller}[Section 11] that the action is also well-defined on meromorphic Abelian differentials with at most simple poles.  Furthermore, \cite{BainbridgeMoller}[Proposition 11.1] says that the action of $\text{GL}_2^+(\mathbb{R})$ extends continuously to the boundary of $\overline{\mathcal{M}_g}$.  We point out to the reader that the action by $\text{GL}_2^+(\mathbb{R})$ on $\overline{\mathcal{M}_g}$ without the action by $\mathbb{R}^*$ is the same as considering the action of \splin ~on $\overline{\mathcal{M}_g}/\mathbb{R}^*$ because the action by $\mathbb{R}$ commutes with everything.

\begin{definition}
Given a surface $(X,q) \in \mathcal{Q}_{g,n}$, the \emph{Teichm\"uller disc} of $(X,q)$ is the orbit of $(X,q)$ in $\mathcal{Q}_{g,n}$ under the action by \splin .
\end{definition}

The \emph{Teichm\"uller geodesic flow}, denoted $G_t$, on the bundle of quadratic differentials is the action by diagonal matrices:
$$G_t = \left[ \begin{array}{cc}
e^t & 0 \\
0 & e^{-t} \end{array} \right].$$
We note for the convenience of the reader that the residue of the simple pole of an Abelian differential differs from the holonomy vector by a factor of $2\pi i$.

\begin{lemma}
\label{GtResidue}
Let $\omega$ be an Abelian differential on a surface $X$ with residue $c = a + ib$ at $p \in X$.  Let $c_{G_t}$ denote the residue at $p$ after acting by $G_t$ on $(X,\omega)$.  Then
$$c_{G_t} = ae^{-t} +ibe^t.$$
\end{lemma}

\begin{proof}
Without loss of generality, let $p = 0$ in local coordinates about $p$.  By \cite{Strebel}[Theorem 6.3], it suffices to look at how the differential $c\,dz/z$ changes under the action by $G_t$.  To do this, convert to polar coordinates and integrate the differential around the curve $\gamma$ defined by $r = 1$.  Let $c = a + ib$.  Then
$$\frac{c\,dz}{z} = (a+ib)\left(\frac{dr}{r} + i\,d\theta\right) = \frac{a\,dr}{r}-b\,d\theta + i\left(b\frac{dr}{r}+a\,d\theta\right).$$
Furthermore, $dr = 0$ because $r = 1$.  So this simplifies to $(-b + ia)\,d\theta$ and acting by $G_t$ we get $(-be^t + iae^{-t})\,d\theta$.  Therefore,
$$c_{G_t} = \frac{1}{2\pi i}\int_0^{2\pi}(-be^t + iae^{-t})\,d\theta = ae^{-t} +ibe^t.$$
\end{proof}

\begin{definition}
A number $c \in \mathbb{C}$ is \emph{$\varepsilon$-nearly imaginary} if $|\arg(c) \pm \pi/2| < \varepsilon$.
\end{definition}

\begin{lemma}
\label{MostlyImagRes}
Let $(X', \omega')$ be a degenerate surface carrying an Abelian differential with simple poles and residues $\{c_1, \ldots, c_m\}$.  Given $\varepsilon > 0$, there exists $A \in$ \splin ~such that if $c_j'$ is a residue of $A \cdot (X',\omega')$, for $1 \leq j \leq m$, then $c_j'$ is $\varepsilon$-nearly imaginary.
\end{lemma}

\begin{proof}
It is possible that $\omega'$ has some real residues.  If so, multiply $\omega'$ by a complex unit $\zeta$ so that $\zeta\omega'$ has no real residues.  Given a residue $\zeta c_j$ of $\zeta \omega'$, after acting on $\zeta c_j$ by $G_t$, the real part of the resulting residue is $e^{-t}\Re(\zeta c_j)$ by Lemma \ref{GtResidue}.  Hence, there exists $T$ such that $|e^{-T}\Re(\zeta c_j)| < \varepsilon$.
\end{proof}

\begin{lemma}
\label{CylModImpPinch}
Let $\{(X_n,\omega_n)\}_{n=0}^{\infty}$ be a sequence of surfaces containing cylinders $C_n \subset X_n$ with core curves $\gamma_n$.  Let $w_n$ and $h_n$ denote the flat length with respect to $\omega_n$ of the circumference and height of $C_n$, respectively.  If the ratio $h_n/w_n$ tends to infinity with $n$, then the hyperbolic length of $\gamma_n$ converges to zero.
\end{lemma}

\begin{proof}
The modulus of the cylinder $C_n$ is exactly the quotient $h_n/w_n$.  By \cite{LenzhenRafi}[Lemma 3],
$$\text{Ext}_x(\gamma_n) \leq \frac{1}{\text{Mod}_x(\gamma_n)}.$$
By \cite{MaskitExtHyp}[Corollary 2], $\text{Ext}_x(\gamma_n)$ goes to zero with the hyperbolic length of $\gamma_n$.
\end{proof}

\begin{corollary}
\label{GtCurvePinch}
Let $(X,\omega)$ admit a cylinder with core curve $\gamma$ such that $\gamma$ lies in the vertical foliation of $X$ by $\omega$.  Then for all divergent sequences of times $\{t_n\}_{n=1}^{\infty}$ for which the limit
$$\lim_{n\rightarrow \infty} G_{t_n} \cdot (X, \omega) = (X',\omega'),$$
exists, $\gamma$ degenerates to a node of $X'$.
\end{corollary}

\begin{proof}
Let $C \subset X$ denote the cylinder with core curve $\gamma$ and let $w$ and $h$ denote the circumference and height of $C$, respectively.  After time $t_n$, the circumference and height are given by $e^{-t_n}w$ and $e^{t_n}h$.  Since
$$\lim_{n \rightarrow \infty} \frac{e^{t_n}h}{e^{-t_n}w} = \infty,$$
$\gamma$ pinches as $n$ tends to infinity, by Lemma \ref{CylModImpPinch}.
\end{proof}

\begin{lemma}
\label{NonHoloPart}
Let $D$ be a Teichm\"uller disc in $\overline{\mathcal{M}_g}/\mathbb{R}^*$.  Given a sequence \newline $\{(X_n,\omega_n)\}_{n=0}^{\infty}$ in $D$ converging to a degenerate surface $(X',\omega')$, there exists a degenerate surface $(X'', \omega'')$ in the closure of $D$ such that $\omega''$ is not holomorphic on every part of $X''$.  Furthermore, $X''$ is reached from $X'$ by pinching additional curves of $X'$.
\end{lemma}

\begin{proof}
By Lemma \ref{NonZeroPart}, we assume that there is a part $S \subset X'$ such that $\omega'$ is not identically zero on $S$.  If $\omega'$ has simple poles on $X'$, then we are done, so assume otherwise.  By \cite{MasurClosedTraj}[Theorem 2], there is a cylinder $C_1$ on $S$.  Degenerate $S$ under the Teichm\"uller geodesic flow by pinching the core curve of $C_1$.  All punctures of $X'$ are obviously preserved under the \splin ~action.  The new limit $\omega'_1$ carries an Abelian differential which is not identically zero everywhere by Lemma \ref{NonZeroPart}.  If $\omega'_1$ is holomorphic on every part we can repeat the argument.  Since the genus is finite, the repetition of this argument will terminate when we reach a differential that is not holomorphic or when the surface degenerates to a sphere, which does not carry holomorphic differentials.  Since the punctures of $X'$ are preserved under the \splin ~action, $X''$ is reached from $X'$ by pinching additional curves.  Furthermore, it follows from the continuity of the \splin ~action \cite{BainbridgeMoller}[Proposition 11.1] that $X''$ is in the closure of $D$.
\end{proof}

\section{Lyapunov Exponents and the Rank One Locus}

In the first subsection, we give the precise formulation of the problem answered in this paper.  In the second subsection we present all of the technical lemmas related to the derivative of the period matrix that will be used throughout the remainder of this paper.

\subsection{Lyapunov Exponents of the KZ-Cocycle}

Let $X$ be a Riemann surface of genus $g$.  Consider the cocycle defined by the Teichm\"uller geodesic flow as follows
$$G_t \times \text{Id}: T_g \times H^1(X,\mathbb{C}) \rightarrow T_g \times H^1(X,\mathbb{C}).$$
The mapping class group preserves the real and imaginary parts of \newline $T_g \times H^1(X,\mathbb{C})$.  The \emph{Kontsevich-Zorich cocycle} is the quotient cocycle
$$G_t^{KZ}: \Re\left((T_g \times H^1(X,\mathbb{C}))/\Gamma_g\right) \rightarrow \Re\left((T_g \times H^1(X,\mathbb{C}))/\Gamma_g\right)$$
restricted to the real part.

Let $\nu$ denote a finite \splin -invariant ergodic measure on $\mathcal{M}_g$.  The cocycle $G_t^{KZ}$ admits a spectrum of $2g$ Lyapunov exponents with respect to $\nu$.  The natural symplectic structure on $H^1(X,\mathbb{C})$ induces a symplectic structure on the entire bundle $\Re\left((T_g \times H^1(X,\mathbb{C}))/\Gamma_g\right)$, which forces a symmetry of the $2g$ Lyapunov exponents.
$$1 = \lambda_1^{\nu} \geq \lambda_2^{\nu} \geq \cdots \geq \lambda_g^{\nu} \geq -\lambda_g^{\nu} \geq \cdots \geq -\lambda_2^{\nu} \geq -\lambda_1^{\nu} = -1.$$
We refer to these $2g$ numbers as the \emph{spectrum of Lyapunov exponents of the Kontsevich-Zorich cocycle} or the \emph{KZ-spectrum} for short.  If $\lambda_k^{\nu} = 0$, for some $k$, then the spectrum is called \emph{degenerate}.  If $\lambda_k^{\nu} = 0$ for all $k > 1$, then the KZ-spectrum is \emph{completely degenerate}.

Kontsevich and Zorich \cite{KZHodge} as well as Forni \cite{ForniDev} gave a formula for the sum of these exponents in terms of the eigenvalues of a Hermitian form.  These eigenvalues were reinterpreted through the second fundamental form of the Hodge bundle \cite{ForniMatheusZorichLyapSpectHodge}.  Let $(X,\omega) \in \mathcal{M}_g$.  Let $L_{\omega}^2(X)$ be the Hilbert space of complex-valued functions on $X$ that are $L^2$ with respect to $\omega$.  Let $\langle \cdot, \cdot \rangle_{\omega}$ be the inner product on $L_{\omega}^2(X)$.  Let $M^{\pm}_{\omega} \subset L_{\omega}^2(X)$ be the subspaces of meromorphic and anti-meromorphic functions, respectively.  Define the orthogonal projections
$$\pi_{\omega}^{\pm}: L_{\omega}^2(X) \rightarrow M^{\pm}_{\omega}.$$
For two meromorphic functions $m_1^+, m_2^+ \in M^+_{\omega}$,
$$H_{\omega}(m_1^+, m_2^+) = \langle\pi_{\omega}^-(m_1^+),\pi_{\omega}^-(m_2^+)\rangle_{\omega}.$$
The eigenvalues of $H_{\omega}(\cdot, \cdot)$ are given by the functionals $\Lambda_k(\omega): \mathcal{M}_g^{(1)} \rightarrow \mathbb{R}$, which are continuous for all $k$ and $\omega$, and obey the inequalities
$$1 \equiv  \Lambda_1(\omega) \geq \Lambda_2(\omega) \geq \cdots \geq \Lambda_g(\omega) \geq 0.$$
In \cite{ForniHand}, Forni introduced a filtration of sets
$$\mathcal{D}_g(1) \subset \mathcal{D}_g(2) \subset \cdots \subset \mathcal{D}_g(g-1),$$
where
$$\mathcal{D}_g(k) = \{ (X,\omega) \in \mathcal{M}_g | \Lambda_{k+1}(\omega) = \cdots = \Lambda_g(\omega) = 0 \},$$
and $\mathcal{D}_g(k)$ is called the \emph{rank $k$ locus}.  The set $\mathcal{D}_g(g-1) = \mathcal{D}_g$ is the \emph{determinant locus} introduced in \cite{ForniDev}.

Let $\nu$ be an \splin -invariant measure on a connected component $\mathcal{C}_{\kappa}$ of the stratum $\mathcal{H}(\kappa) \subset \mathcal{M}_g$ of Abelian differentials.  Corollary 5.3 of \cite{ForniDev} gives the following identity:
$$\lambda_2^{\nu} + \cdots + \lambda_g^{\nu} = \frac{1}{\nu(\mathcal{C}_{\kappa})} \int_{\mathcal{C}_{\kappa}}\Lambda_2(\omega) + \cdots + \Lambda_g(\omega) \, d\nu.$$
In \cite{ForniHand}, Forni notes that this formula can be extended to any \splin -invariant ergodic probability measure, from which the lemma follows.

\begin{lemma}[Forni \cite{ForniHand}, Cor. 7.1]
\label{CompDegImpHRank1}
Let $\nu$ be a finite \splin -invariant ergodic measure on the moduli space $\mathcal{M}_g$.  The KZ-spectrum with respect to $\nu$ is completely degenerate if and only if for almost every $(X,\omega) \in \text{supp}(\nu)$, $H_{\omega}$ has rank one, i.e. $\text{supp}(\nu) \subset $ \RankOne .
\end{lemma}

We introduce the derivative of the period matrix, which will be the focus of this paper.  Let $\{a_1, \ldots, a_g, b_1, \ldots, b_g \}$ be a basis for the first homology group $H_1(X,\mathbb{C})$.  Let $\{\theta_j\}_{j=1}^g$ be a basis of the complex vector space of holomorphic Abelian differentials on $X$ normalized so that
$$\int_{a_i}\theta_j = \delta_{ij},$$
where $\delta_{ij}$ is the Kronecker delta.  Under this choice of basis of Abelian differentials, the \emph{period matrix} $\Pi(X)$ is the symmetric matrix with positive definite imaginary part whose components are given by
$$b_{ij} = \int_{b_i}\theta_j.$$

The space of Beltrami differentials, $B(X)$ is dual to the cotangent space of quadratic differentials.  Every Abelian differential $\omega$ uniquely determines a Beltrami differential
$$\mu_{\omega} = \frac{\bar \omega}{\omega},$$
which is defined everywhere except at the zeros and poles of $\omega$ of which there are only finitely many.  In the Teichm\"uller space $R(X)$ the space $B(X)$ represents the tangent space and $\mu \in B(X)$ a tangent vector at $X$.  In $R(X)$, $\mu$ determines a direction in which we can take a derivative of $\Pi(X)$.  The derivative of the period matrix at $X$ in direction $\mu$ is denoted by $d\Pi(X)/d\mu$.  Let $\omega = h(z)\,dz$ and $\theta_k = f_k(z)\,dz$, for all $k$.  Rauch's formula, \cite{ImaTan}[Proposition A.3], gives a concise formula for the components of the derivative of the period matrix.
$$\frac{d\Pi_{ij}(X)}{d\mu_{\omega}} = \int_X \theta_i \theta_j \,d\mu_{\omega} = \int_X f_if_j \frac{\bar h}{h} dz\wedge d\bar z$$
In the proof of Lemma 4.1 of \cite{ForniDev}, Forni defines a complex bilinear form on holomorphic Abelian differentials $\omega_1$, $\omega_2$ by
$$B_{\omega}(\omega_1, \omega_2) = \left\langle\frac{\omega_1}{\omega}, \frac{\bar \omega_2}{\omega}\right\rangle_{\omega}.$$
It was proven in \cite{ForniDev} that $H_{\omega} = B_{\omega}B_{\omega}^*$ (and a typo in the equation in \cite{ForniDev} was corrected in \cite{ForniMatheusZorichLyapSpectHodge}).  It is possible to choose a basis of Abelian differentials $\{\phi_1, \ldots, \phi_g\}$ on $X$ such that
$$\frac{d\Pi_{ij}(X)}{d\mu_{\omega}} = B_{\omega}(\phi_i, \phi_j).$$
Hence, $H_{\omega}$ has rank one if and only if $d\Pi(X)/d\mu_{\omega}$ has rank one.  For this reason it suffices to regard \RankOne ~as the set where $d\Pi(X)/d\mu_{\omega}$ has rank one for the remainder of this paper.

Since $\nu$ is an \splin -invariant measure, $\text{supp}(\nu)$ must be an \splin -invariant set.  Consider $(X,\omega) \in \text{supp}(\nu)$.  Let $D$ be the Teichm\"uller disc generated by $(X,\omega)$. Then $D \subset \text{supp}(\nu)$, and if the KZ-spectrum with respect to $\nu$ is completely degenerate, then $D \subset $ \RankOne .  This is precisely the problem that we address in this paper.

\begin{problem}
Classify all Teichm\"uller discs $D$ such that $D \subset $ \RankOne .
\end{problem}

\subsection{The Derivative of the Period Matrix}
\label{DerPerMatSubSect}

One of the most important techniques in this paper is the use of estimates for the derivative of the period matrix near the boundary of the moduli space $\mathcal{M}_g$.  In this section we introduce plumbing coordinates for a Riemann surface and express Abelian differentials in terms of those plumbing coordinates using the exposition of \cite{WolpertInfinitDeformations}.  Unfortunately, it will not be possible to guarantee convergence of the derivative of the period matrix in every possible scenario, but it will be possible for all cases relevant to this paper.  Lemma \ref{Forni42Seqs} below is a stronger statement than that of \cite{ForniDev}[Lemma 4.2] because it applies to any sequence satisfying a relatively lax set of assumptions.  These convergence lemmas motivate and justify defining the rank of the derivative of the period matrix for surfaces in the boundary of $\overline{\mathcal{M}_g}$.

Plumbing coordinates have been used extensively from \cite{MasurExt} to \cite{ForniDev}, among others.  They have been used to write explicit formulas for differentials near the boundary of the moduli space.  Wolpert \cite{WolpertAnnuli} reworked the foundations of differentials on families of degenerating surfaces using the language of sheaves, and expressed the differentials on degenerating surfaces in terms of plumbing coordinates.  We copy the language and notation of \cite{ForniDev}[Section 4] and \cite{WolpertAnnuli, WolpertInfinitDeformations}, as appropriate.  Let $X'$ be a degenerate Riemann surface in the boundary of $\overline{\mathcal{R}_g}$.  Let $X'$ have $1 \leq m \leq 3g-3$ pairs of punctures $\{(p_i, p_i')\}$, for $1 \leq i \leq m$.  Let $\tau \in \mathbb{C}^{3g-3-m}$ denote the local coordinates for a neighborhood of $X'$ in the Teichm\"uller space of $X'$.  We denote surfaces in a neighborhood of $X' \in \mathcal{R}_g$ by $X(0,\tau)$.  We refer the reader to \cite{WolpertInfinitDeformations}[Section 3], where the coordinates are specifically chosen to correspond to small deformations of the complex structure on $X'$.  For our purposes, it suffices to know that such a coordinate $\tau$ exists.  Let $(U_i(0,\tau), z_i)$ and $(V_i(0,\tau), w_i)$ be coordinate charts around $p_i$ and $p_i'$, respectively, such that $z_i(p_i) = w_i(p_i') = 0$.  Following \cite{WolpertAnnuli, WolpertInfinitDeformations}, let $c', c''$ be positive constants, $V = \{|z| < c', |w| < c''\}$, $D = \{|t| < c'c''\}$, and $\pi: V \rightarrow D$ be the singular fibration with projection $\pi(z,w) = zw = t$, where $t \in \mathbb{C}$.  Let $t = (t^{(1)}, \ldots, t^{(m)}) \in D^m$.  Let $c < 1$ be a small positive constant.  For $|t^{(i)}| < c^4$ and $1 \leq i \leq m$, remove the discs $\{|z_i| \leq c^2\}$ and $\{|w_i| \leq c^2\}$ from $X'(0,\tau)$ to get an open surface $X_{\tau}^*$.  For each $i$, identify a point $u_0 \in \{u | c^2 < |z_i(u)| < c\} \subset X_{\tau}^*$ to the point $(z_i(u_0), t^{(i)}/z_i(u_0))$ in the fiber of a $k^{\text{th}}$ factor of $\pi: V \rightarrow D$, and identify a point $v_0 \in \{v | c^2 < |w_i(v)| < c\} \subset X_{\tau}^*$ to the point $(t^{(i)}/w_i(v_0), w_i(v_0))$ in the fiber of a $k^{\text{th}}$ factor of $\pi: V \rightarrow D$.  This implies that we can write $X(t,\tau)$ to fully coordinatize a neighborhood of the degenerate surface $X' := X(0, \tau_{\infty}) \in \overline{\mathcal{R}_g}$.

In \cite{MasurExt} and \cite{ForniDev}, the identification of the annuli is made directly so that if we translate their language to Wolpert's, we get
$$(z_i(u_0), t^{(i)}/z_i(u_0)) = (t^{(i)}/w_i(v_0), w_i(v_0))$$
and identify along the curve $|w_i(v_0)| = |z_i(u_0)| = \sqrt{|t^{(i)}|}$.  It suffices to follow this convention throughout this paper.  Following the notation of \cite{WolpertAnnuli}, we define annuli with respect to this identification for fixed $q$.  Let
$$R_z(t^{(i)}) := \{\sqrt{|t^{(i)}|}/c'' < |\zeta_i| < c' \} \subset \{|t^{(i)}|/c'' < |\zeta_i| < c' \}$$
and
$$R_w(t^{(i)}) := \{\sqrt{|t^{(i)}|}/c' < |\zeta_i| < c'' \} \subset \{|t^{(i)}|/c' < |\zeta_i| < c'' \}.$$
Let $c = c' = c''$ and define
$$X^*(t,\tau) =: X_{\tau}^* \cup \bigcup_{i=1}^m R_z(t^{(i)}) \cup R_w(t^{(i)}).$$

Next we consider Abelian differentials on Riemann surfaces.  Let $D_1 \times \cdots \times D_m = D^m$ denote the $m$ copies of $D$ above.  Following \cite{WolpertAnnuli}, every Abelian differential can be expressed in terms of local coordinates on $D_j$.  This is done by considering the coordinate $\zeta_j$ on an annulus and the map $\zeta_j \mapsto (\zeta_j, t^{(j)}/\zeta_j)$ (resp. $\zeta_j \mapsto (t^{(j)}/\zeta_j, \zeta_j)$).  As $t^{(j)}$ tends to zero this yields the convergence of the differential in local coordinates about the degenerating annuli resulting in the map $\zeta_j \mapsto (\zeta_j, 0)$ (resp. $\zeta_j \mapsto (0, \zeta_j)$).  

It follows from a form the Cartan-Serre theorem with parameters or \cite{MasurExt} [Proposition 4.1], that there is a basis of Abelian differentials $\{\theta_1(t,\tau),\ldots, $ \newline $\theta_g(t,\tau)\}$ on $X(t,\tau)$, for all small $t$, such that $\{ \theta_1(0,\tau_{\infty}), \ldots, \theta_g(0,\tau_{\infty}) \}$ spans the space of Abelian differentials on $X'$.  We assume such a fixed basis in a neighborhood of a degenerate surface throughout this paper.  Let 
$$t' = (t^{(1)}, \ldots, t^{(j-1)}, t^{(j+1)}, \ldots, t^{(m)}).$$
In local coordinates on $D_j$, let $\theta_i(t', \tau, \zeta_j, t^{(j)}/\zeta_j) = 2f_i(t', \tau, \zeta_j, t^{(j)}/\zeta_j)\,d\zeta_j/\zeta_j$, where
$$f_i(t', \tau, \zeta_j, t^{(j)}/\zeta_j) = \sum_{k,\ell \geq 0} a_{k\ell}(t', \tau)\zeta_j^k(t^{(j)}/\zeta_j)^{\ell},$$
by \cite{WolpertAnnuli}.

Let $\{(X_n, \omega_n)\}_{n=0}^{\infty}$ be a sequence of surfaces carrying Abelian differentials converging to a degenerate surface $(X', \omega')$.  Without loss of generality, we can ignore the beginning of the sequence so that every element of the sequence can be expressed in terms of the local coordinates established above.  Thus, let $X_n = X(t_n, \tau_n)$ and $X' = X(0,\tau_{\infty})$.  Let 
$$\omega_n = 2A_n(t', \tau, \zeta_j, t^{(j)}/\zeta_j)\,\frac{d\zeta_j}{\zeta_j},$$
be local coordinates on $D_j$.  Contrary to the coefficients $f_i$ in the basis of Abelian differentials, note the dependence of the function $A_n$ on $n$.

\begin{lemma}
\label{Forni42Seqs}
We follow the notation established above.  Let $\{(X(t_n, \tau_n), \omega_n)\}_{n=0}^{\infty}$ be a sequence of surfaces converging to a degenerate surface $(X(0,\tau_{\infty}), \omega')$.  For each $n$, let $\{\theta_1(t_n, \tau_n), \ldots, \theta_g(t_n, \tau_n)\}$ be a basis for the space of Abelian differentials on $X(t_n, \tau_n)$.  Given $i, j$, for all $k$, if one of the following is true:
\begin{itemize}
\item[(1)] Either $f_i(0, \tau_{\infty}, 0,0) = 0$ on $D_k$ or $f_j(0, \tau_{\infty}, 0,0) = 0$ on $D_k$, or
\item[(2)] $A_{\infty}(0, \tau_{\infty}, 0,0) \not= 0$ on $D_k$, then
\end{itemize}
$$\lim_{n \rightarrow \infty} \left( \frac{d\Pi_{ij}(X(t_n,\tau_n))}{d\mu_{\omega_n}} - \int_{X^*(t_n,\tau_{\infty})} \theta_i(0,\tau_{\infty})\theta_j(0,\tau_{\infty})\,d\mu_{\omega'}  \right) = 0.$$
\end{lemma}

\begin{proof}
On compact subsets away from the punctures, the integrand converges to an integrable real analytic function, so the dominated convergence theorem gives us the desired convergence on these compact sets.  Hence, it suffices to prove convergence on each annulus $R_z(t_n^{(k)})$ and $R_w(t_n^{(k)})$.  To get convergence on $R_w(t_n^{(k)})$, it suffices to show convergence on $R_z(t_n^{(k)})$ because they are symmetric up to multiplication by a constant.  Using Rauch's formula, we explicitly write the expression to be estimated as $t_n$ tends to zero in $\mathbb{C}^n$.  That the following integral makes sense and proves the desired convergence follows from \cite{WolpertAnnuli}[Lemma 2].
$$4\int_{R_z(t_n^{(k)})} \left(\frac{f_i(t_n', \tau_n,\zeta_k, t_n^{(k)}/\zeta_k)}{\zeta_k}\frac{f_j(t_n', \tau_n,\zeta_k, t_n^{(k)}/\zeta_k)}{\zeta_k}\frac{\overline{A_n(t_n', \tau_n,\zeta_k, t_n^{(k)}/\zeta_k)/\zeta_k}}{A_n(t_n', \tau_n,\zeta_k, t_n^{(k)}/\zeta_k)/\zeta_k} \right .$$
$$\left . - \frac{f_i(0, \tau_{\infty},\zeta_k, 0)}{\zeta_k}\frac{f_j(0, \tau_{\infty},\zeta_k, 0)}{\zeta_k}\frac{\overline{A_{\infty}(0, \tau_{\infty},\zeta_k, 0)/\zeta_k}}{A_{\infty}(0, \tau_{\infty},\zeta_k, 0)/\zeta_k} \right) d\zeta_k \wedge d \overline{\zeta_k}.$$
Following the proof of \cite{ForniDev}[Lemma 4.2], we split the difference in the integrand into the following three terms:
\[ \begin{array}{cl}
(\text{I})   & 4\int_{R_z(t_n^{(k)})} \left(\frac{f_i(t_n', \tau_n,\zeta_k, t_n^{(k)}/\zeta_k)}{\zeta_k} - \frac{f_i(0, \tau_{\infty},\zeta_k, 0)}{\zeta_k} \right) \\
& \qquad \frac{f_j(t_n', \tau_n,\zeta_k, t_n^{(k)}/\zeta_k)}{\zeta_k} \frac{\overline{A_n(t_n', \tau_n,\zeta_k, t_n^{(k)}/\zeta_k)/\zeta_k}}{A_n(t_n', \tau_n,\zeta_k, t_n^{(k)}/\zeta_k)/\zeta_k} \,d\zeta_k \wedge d \overline{\zeta_k} \\
(\text{II})  & 4\int_{R_z(t_n^{(k)})} \left( \frac{f_j(t_n', \tau_n,\zeta_k, t_n^{(k)}/\zeta_k)}{\zeta_k} - \frac{f_j(0, \tau_{\infty},\zeta_k, 0)}{\zeta_k} \right) \\
& \qquad \frac{f_i(0, \tau_{\infty},\zeta_k, 0)}{\zeta_k}\frac{\overline{A_n(t_n', \tau_n,\zeta_k, t_n^{(k)}/\zeta_k)/\zeta_k}}{A_n(t_n', \tau_n,\zeta_k, t_n^{(k)}/\zeta_k)/\zeta_k} \,d\zeta_k \wedge d \overline{\zeta_k} \\
(\text{III}) & 4\int_{R_z(t_n^{(k)})} \left(\frac{f_i(0, \tau_{\infty},\zeta_k, 0)}{\zeta_k}\frac{f_j(0, \tau_{\infty},\zeta_k, 0)}{\zeta_k} \right) \\ 
& \qquad \left(\frac{\overline{A_n(t_n', \tau_n,\zeta_k, t_n^{(k)}/\zeta_k)/\zeta_k}}{A_n(t_n', \tau_n,\zeta_k, t_n^{(k)}/\zeta_k)/\zeta_k} - \frac{\overline{A_{\infty}(0, \tau_{\infty},\zeta_k, 0)/\zeta_k}}{A_{\infty}(0, \tau_{\infty},\zeta_k, 0)/\zeta_k}\right) \,d\zeta_k \wedge d \overline{\zeta_k} . \\
\end{array}\]

Regardless of whether Case 1) or 2) holds, convergence of the expressions (I) and (II) is guaranteed.  Consider the difference
$$f_*(t_n', \tau_n,\zeta_k, t_n^{(k)}/\zeta_k) - f_*(0, \tau_{\infty},\zeta_k, 0),$$
where $*$ indicates that the choice of subscript $i$ or $j$ does not matter here as long as the subscript is the same on both functions.  By \cite{WolpertAnnuli}, $f_*$ is holomorphic in all variables, hence, there is a constant $C_0 > 0$ such that
$$\frac{2}{|\zeta_k|}\left|f_*(t_n', \tau_n,\zeta_k, t_n^{(k)}/\zeta_k) - f_*(0, \tau_{\infty},\zeta_k, 0)\right| \leq C_0 \frac{|t_n^{(k)}|}{|\zeta_k|^2}$$
and
$$2\left|\frac{f_*(t_n', \tau_n,\zeta_k, t_n^{(k)}/\zeta_k)}{\zeta_k}\right| \leq C_0 \frac{1}{|\zeta_k|}.$$
Using H\"older's inequality, there is a constant $C_1 > 0$ such that the following inequalities hold
\[ \begin{array}{cl}
|(\text{I})| & \leq  4\|(f_i(t_n', \tau_n,\zeta_k, t_n^{(k)}/\zeta_k)-f_i(0, \tau_{\infty},\zeta_k, 0))/\zeta_k\|_{L^2(R_z(t_n^{(k)}))} \\
& \qquad \| f_j(t_n', \tau_n,\zeta_k, t_n^{(k)}/\zeta_k)/\zeta_k \|_{L^2(R_z(t_n^{(k)}))} \\
& \leq \|C_0 \frac{|t_n^{(k)}|}{|\zeta_k|^{2}}\|_{L^2(R_z(t_n^{(k)}))}\| C_0 \frac{1}{|\zeta_k|} \|_{L^2(R_z(t_n^{(k)}))} \\
& \leq C_1 \frac{|t_n^{(k)}|}{\sqrt{|t_n^{(k)}|}} (\log |t_n^{(k)}|)^{1/2} = C_1 \sqrt{|t_n^{(k)}|} (\log |t_n^{(k)}|)^{1/2} \\
\end{array}\]
and
\[ \begin{array}{cl}
|(\text{II})| & \leq  4\|(f_j(t_n', \tau_n,\zeta_k, t_n^{(k)}/\zeta_k)-f_j(0, \tau_{\infty},\zeta_k, 0))/\zeta_k\|_{L^2(R_z(t_n^{(k)}))} \\
& \qquad\| f_i(0, \tau_{\infty},\zeta_k, 0)/\zeta_k \|_{L^2(R_z(t_n^{(k)}))} \\
& \leq \|C_0 \frac{|t_n^{(k)}|}{|\zeta_k|^{2}}\|_{L^2(R_z(t_n^{(k)}))}\| C_0 \frac{1}{|\zeta_k|} \|_{L^2(R_z(t_n^{(k)}))} \leq C_1 \sqrt{|t_n^{(k)}|} (\log |t_n^{(k)}|)^{1/2} . \\
\end{array}\]

The convergence for (III) remains to be shown.  We split this into two cases that are resolved by Lemmas \ref{IIIConvergenceNoTwoSimpPoles} and \ref{IIIConvergenceLimitNonzero}.  Note that in Case 2), it suffices to assume that $f_i(0, \tau_{\infty}, 0, 0) \not= 0$ and $f_j(0, \tau_{\infty}, 0, 0) \not= 0$.  Otherwise, Case 2) is subsumed by Case 1).
\end{proof}

\begin{lemma}
\label{IIIConvergenceNoTwoSimpPoles}
Given $k$, if $f_i(0, \tau_{\infty}, 0,0) = 0$ on $D_k$ or $f_j(0, \tau_{\infty}, 0,0) = 0$ on $D_k$, then (III) converges to zero as $n$ tends to infinity.
\end{lemma}

\begin{proof}
By the assumption that at most one of $f_i$ and $f_j$ has a simple pole, we have
$$|\text{(III)}| \leq 4\left| \int_{R_z(t_n^{(k)})} (f_i(0, \tau_{\infty},\zeta_k, 0)f_j(0, \tau_{\infty},\zeta_k, 0)) \right .$$
$$\left . \left(\frac{\overline{A_n(t_n', \tau_n,\zeta_k, t_n^{(k)}/\zeta_k)/\zeta_k}}{A_n(t_n', \tau_n,\zeta_k, t_n^{(k)}/\zeta_k)/\zeta_k} - \frac{\overline{A_{\infty}(0, \tau_{\infty},\zeta_k, 0)/\zeta_k}}{A_{\infty}(0, \tau_{\infty},\zeta_k, 0)/\zeta_k}\right ) \frac{d\zeta_k \wedge d \overline{\zeta_k}}{\zeta_k}\right|.$$
Since $f_i$ and $f_j$ are holomorphic, they are bounded on $R_z(t_n^{(k)})$.  This implies that there is a constant $C > 0$ such that
$$|\text{(III)}| \leq  C \int_{R_z(t_n^{(k)})} \left|\frac{\overline{A_n(t_n', \tau_n,\zeta_k, t_n^{(k)}/\zeta_k)}}{A_n(t_n', \tau_n,\zeta_k, t_n^{(k)}/\zeta_k)} - \frac{\overline{A_{\infty}(0, \tau_{\infty},\zeta_k, 0)}}{A_{\infty}(0, \tau_{\infty},\zeta_k, 0)}\right | \frac{d\zeta_k \wedge d \overline{\zeta_k}}{|\zeta_k|}.$$
This converges by the dominated convergence theorem because the integrand is bounded by the integrable function $2/|\zeta_k|$ for all $n$.
\end{proof}

\begin{lemma}
\label{IIIConvergenceLimitNonzero}
Given $k$, if $f_i(0, \tau_{\infty}, 0,0) \not= 0$, $f_j(0, \tau_{\infty}, 0,0) \not= 0$, and \newline $A_{\infty}(0, \tau_{\infty}, 0,0) \not= 0$ on $D_k$, then (III) converges to zero as $n$ tends to infinity.
\end{lemma}

\begin{proof}
By assumption, there exists $N$ such that $A_n(0, \tau_{\infty}, 0,0) \not= 0$ for all $n \geq N$.  Since $A_{\infty}(0, \tau_{\infty}, 0,0) \not= 0$, there exists $r > 0$ such that \newline $A_n(t_n', \tau_n,\zeta_k, t_n^{(k)}/\zeta_k) \not= 0$ and $\overline{A_n(t_n', \tau_n,\zeta_k, t_n^{(k)}/\zeta_k)}/A_n(t_n', \tau_n,\zeta_k, t_n^{(k)}/\zeta_k)$ is a real analytic function in the polydisc $\{|t_n^{(k)}| < r, |\zeta_k| < r\} \subset \mathbb{C}^2$.  Therefore, there exists a constant $C_2 > 0$ such that in the annulus $\{\sqrt{|t_n^{(k)}|} < |\zeta_k| < r/2\}$, we have
$$\left|\frac{\overline{A_n(t_n', \tau_n,\zeta_k, t_n^{(k)}/\zeta_k)/\zeta_k}}{A_n(t_n', \tau_n,\zeta_k, t_n^{(k)}/\zeta_k)/\zeta_k} - \frac{\overline{A_{\infty}(0, \tau_{\infty},\zeta_k, 0)/\zeta_k}}{A_{\infty}(0, \tau_{\infty},\zeta_k, 0)/\zeta_k}\right|$$
$$ = \left|\frac{\overline{A_n(t_n', \tau_n,\zeta_k, t_n^{(k)}/\zeta_k)}}{A_n(t_n', \tau_n,\zeta_k, t_n^{(k)}/\zeta_k)} - \frac{\overline{A_{\infty}(0, \tau_{\infty},\zeta_k, 0)}}{A_{\infty}(0, \tau_{\infty},\zeta_k, 0)}\right| \leq C_2 \frac{|t_n^{(k)}|}{|\zeta_k|} \leq C_2\sqrt{|t_n^{(k)}|}.$$
Exactly as in the proof of \cite{ForniDev}[Lemma 4.2], there exists a constant $C_3 > 0$ such that
$$|(\text{III})| \leq -C_3\sqrt{|t_n^{(k)}|} \log|t_n^{(k)}| + 4\left |\int_{|\zeta_k| \geq r/2} \left(\frac{f_i(0, \tau_{\infty},\zeta_k, 0)}{\zeta_k}\frac{f_j(0, \tau_{\infty},\zeta_k, 0)}{\zeta_k} \right) \right.$$
$$ \left . \left(\frac{\overline{A_n(t_n', \tau_n,\zeta_k, t_n^{(k)}/\zeta_k)/\zeta_k}}{A_n(t_n', \tau_n,\zeta_k, t_n^{(k)}/\zeta_k)/\zeta_k} - \frac{\overline{A_{\infty}(0, \tau_{\infty},\zeta_k, 0)/\zeta_k}}{A_{\infty}(0, \tau_{\infty},\zeta_k, 0)/\zeta_k}\right ) d\zeta_k \wedge d \overline{\zeta_k}\right|.$$
Since the domain of integration in the right-hand integral does not depend on $t$, the domain of integration is compact and the integrand is bounded by an integrable function for all $n$.  This proof is completed by applying the dominated convergence theorem to the sequence as $n$ tends to infinity.
\end{proof}

\begin{definition}
Define the extension of the rank $k$ locus to the boundary of $\mathcal{M}_g$ to be the closure of $\mathcal{D}_g(k)$ in $\overline{\mathcal{M}_g}$ and denote it by $\overline{\mathcal{D}_g(k)}$.
\end{definition}

\begin{remark}
Since $\mathcal{D}_g(k)$ is already a closed set in $\mathcal{M}_g$, we would never need to write $\overline{\mathcal{D}_g(k)}$ to mean the closure of $\mathcal{D}_g(k)$ in $\mathcal{M}_g$.
\end{remark}

\begin{lemma}
\label{RankkLocusProp1}
If $(X', \omega') \in \overline{\mathcal{D}_g(k)}$, $\omega'$ is holomorphic on $X'$, and $\omega' \not\equiv 0$ on any part of $X'$, then
$$\text{Rank}\left(\frac{d\Pi(X')}{d\mu_{\omega'}}\right) \leq k.$$
\end{lemma}

\begin{proof}
This is clear for $(X',\omega') \subset \mathcal{D}_g(k)$, so we assume $(X',\omega') \in \overline{\mathcal{D}_g(k)} \cap \partial \overline{\mathcal{M}_g}$.  By definition, $\overline{\mathcal{D}_g(k)}$ is the closure of $\mathcal{D}_g(k)$ in $\overline{\mathcal{M}_g}$, so there exists a sequence $\{(X_n, \omega_n)\}_{n=1}^{\infty}$ in \RankOne ~converging to $(X',\omega')$.  Let $X'$ be a surface of genus $g' < g$.  Let $\{\theta^{(n)}_1, \ldots, \theta^{(n)}_{g'}, \ldots \theta^{(n)}_g\}$ be a basis of Abelian differentials on $X_n$ ordered so that
$$\lim_{n\rightarrow \infty} \theta^{(n)}_m = \theta_m,$$
for $1 \leq m \leq g'$, and the set $\{\theta_1, \ldots, \theta_{g'}\}$ is a basis for the space of holomorphic Abelian differentials on $X'$.  Note that for each $m$, $1 \leq m \leq g'$, $\{\theta^{(n)}_m\}_{n=1}^{\infty}$ is a sequence of holomorphic differentials converging to a holomorphic differential.  Let $A_n = (A_n)_{ij}$ denote the minor of $d\Pi(X_n)/d\mu_{\omega_n}$ defined by
$$A^{(n)}_{ij} = \int_{X_n}\theta^{(n)}_i\theta^{(n)}_j d \mu_{\omega_n},$$
for $1 \leq i,j \leq g'$, and let $A$ denote the derivative of the period matrix of $(X',\omega')$.  Since we restricted our attention to the basis of differentials that are holomorphic on $X'$ and $\omega'$ is holomorphic, $A_n$ converges to $A$ component-wise by Lemma \ref{Forni42Seqs}.  For any sequence of matrices $\{A_n\}_{n=1}^{\infty}$ converging to a matrix $A$ component-wise, there exists an $\varepsilon > 0$ such that if $\| A_n - A \| < \varepsilon$, where $\|A \|$ denotes the sum of the absolute values of the components of $A$, then $\text{Rank}(A_n) \geq \text{Rank}(A)$.  Also, given a matrix $M$ with minor $B$, $\text{Rank}(M) \geq \text{Rank}(B)$.  The lemma follows by letting $M = d\Pi(X_n)/d\mu_{\omega_n}$ and $B = A_n$, so that
$$k \geq \text{Rank}\left(\frac{d\Pi(X_n)}{d\mu_{\omega_n}}\right) \geq \text{Rank}(A_n) \geq \text{Rank}(A) = \text{Rank}\left(\frac{d\Pi(X')}{d\mu_{\omega'}}\right).$$
\end{proof}

\begin{lemma}
\label{BddOutsideDisks}
Let $\{(X(t_n, \tau_n), \omega_n)\}_{n=0}^{\infty}$ be a sequence of surfaces converging to a surface $(X', \omega') \in \overline{\mathcal{M}_g}$.  For all $i, j$ and $n \geq 0$, there exists a constant $C > 0$, such that
$$\left| \int_{X_{\tau_{\infty}}^*} \theta_i(0,\tau_{\infty}) \theta_j(0,\tau_{\infty}) \frac{\bar \omega'}{\omega'}\right| < C.$$
\end{lemma}

\begin{proof}
The differentials $\theta_i(0,\tau_{\infty})$ are holomorphic on the compact set $\overline{X_{\tau_{\infty}}^*}$, for all $i$, by the definition of $X_{\tau_{\infty}}^*$.  Hence, $|\theta_i(0,\tau_{\infty})| < C'$ for some constant $C'$ and all $i$.  This implies
$$\left| \int_{X_{\tau_{\infty}}^*} \theta_i(0,\tau_{\infty}) \theta_j(0,\tau_{\infty}) \frac{\bar \omega'}{\omega'}\right| \leq \int_{X_{\tau_{\infty}}^*} \left|\theta_i(0,\tau_{\infty}) \theta_j(0,\tau_{\infty})\right| \leq C'^2 = C.$$
\end{proof}

\begin{lemma}
\label{ZeroInSmallDisks}
Let $D_{\varepsilon} = \{z \big| |\varepsilon| \leq |z| \leq 1 \} \subset \mathbb{C}$.  For all $N \geq 0$ and $\varepsilon > 0$,
$$\int_{D_\varepsilon} \frac{z^N}{\bar z} dz \wedge d \bar z = 0.$$
\end{lemma}

\begin{proof}
Convert to polar coordinates by letting $z = re^{i\theta}$.  For all $\varepsilon > 0$
$$\int_{D_{\varepsilon}} z^N/\bar z \,dz \wedge d \bar z = -2i\int_0^{2\pi}\int_{\varepsilon}^1 \frac{r^Ne^{iN\theta}}{re^{-i\theta}} r \,dr d\theta = -2i\int_0^{2\pi}\int_{\varepsilon}^1 r^Ne^{i(N+1)\theta} \,dr d\theta.$$
This expression integrates to zero, for all $N \geq 0$.
\end{proof}

\begin{lemma}
\label{ZeroInSmallDisks2}
Let $D_{\varepsilon} = \{z \big| |\varepsilon| \leq |z| \leq 1 \} \subset \mathbb{C}$.  For all $N \in \mathbb{Z}$, $K \geq 0$ and $\varepsilon > 0$, there exists $C > 0$ such that
$$\left|\int_{D_\varepsilon} z^N\bar z^K \,dz \wedge d \bar z\right| < C.$$
\end{lemma}

\begin{proof}
Convert to polar coordinates by letting $z = re^{i\theta}$.  Then for all $\varepsilon > 0$
$$\int_{D_\varepsilon} z^N\bar z^K \,dz \wedge d \bar z = -2i\int_0^{2\pi}\int_{\varepsilon}^1 r^Ne^{iN\theta}r^Ke^{-iK\theta} r\,dr d\theta$$
$$= -2i\int_0^{2\pi}\int_{\varepsilon}^1 r^{N+1+K}e^{i(N-K)\theta} \,dr d\theta$$
If $N-K \not= 0$, this expression integrates to zero.  Otherwise, this equals
$$2\left|\int_0^{2\pi}\int_{\varepsilon}^1 r^{2K+1} \,dr d\theta\right| < \frac{2\pi}{K+1} + O(\varepsilon) < C,$$
for some $C > 0$.
\end{proof}

We state the following two results for the annulus $R_z(t_n^{(k)})$ and remark that the same results hold for $R_w(t_n^{(k)})$.

\begin{lemma}
\label{BddInDisks}
We follow the notation established above.  Let $\{(X(t_n, \tau_n), \omega_n)\}_{n=0}^{\infty}$ be a sequence of surfaces converging to a degenerate surface $(X(0,\tau_{\infty}), \omega')$.  For each $n$, let $\{\theta_1(t_n, \tau_n), \ldots, \theta_g(t_n, \tau_n)\}$ be a basis for the space of Abelian differentials on $X(t_n, \tau_n)$.  Given $i, j, k$, if either $f_i(0, \tau_{\infty}, 0,0) = 0$ on $D_k$ or $f_j(0, \tau_{\infty}, 0,0) = 0$ on $D_k$, then there exists $C > 0$ such that for all $n \geq 0$
$$\left| \int_{R_z(t_n^{(k)})} \theta_i(0, \tau_{\infty}) \theta_j(0, \tau_{\infty}) d \mu_{\omega'} \right| < C.$$
In particular, if $f_i(0, \tau_{\infty}, 0,0) = 0$, $f_j(0, \tau_{\infty}, 0,0) = 0$ on $D_k$, and \newline $A_{\infty}(0, \tau_{\infty}, 0,0) \not= 0$ on $D_k$, then
$$\lim_{n \rightarrow \infty} \int_{R_z(t_n^{(k)})} \theta_i(0, \tau_{\infty}) \theta_j(0, \tau_{\infty})\,d \mu_{\omega'} = 0.$$
\end{lemma}

\begin{proof}
There are three cases to consider in the first claim of the lemma.  It suffices to consider the case where exactly one of the differentials $\theta_i(0, \tau_{\infty})$ or $\theta_j(0, \tau_{\infty})$ has a simple pole.  Without loss of generality, assume that $\theta_j(0, \tau_{\infty})$ is holomorphic.  Fix a choice of coordinates $\zeta_k$ in $R_z(t_n^{(k)})$ so that by \cite{Strebel}[Theorem 6.3], there exists $K \geq -1$ and $c \in \mathbb{C}$ such that $\omega' = c\zeta_k^K\,d\zeta_k$.  Let $\theta_i(0, \tau_{\infty}) = (c_i/\zeta_k + h_i(\zeta_k))\,d\zeta_k$ and $\theta_j(0, \tau_{\infty}) = h_j(\zeta_k)\,d\zeta_k$, where $h_i$ and $h_j$ are holomorphic in $\zeta_k$.  This yields
$$\left| \int_{R_z(t_n^{(k)})} \theta_i(0, \tau_{\infty}) \theta_j(0, \tau_{\infty}) \,d \mu_{\omega'}\right|$$
$$ = \left| \int_{R_z(t_n^{(k)})} \left(c_i/\zeta_k + h_i(\zeta_k)\right)h_j(\zeta_k)\frac{\overline{c\zeta_k^K}}{c\zeta_k^K}\,d\zeta_k \wedge d\overline{\zeta_k}\right|$$
$$\leq \left| \int_{R_z(t_n^{(k)})} h_j(\zeta_k)\frac{c_i\overline{c\zeta_k^K}}{c\zeta_k^{K+1}}\,d\zeta_k \wedge d\overline{\zeta_k}\right| + \left| \int_{R_z(t_n^{(k)})}  h_i(\zeta_k)h_j(\zeta_k)\frac{\overline{c\zeta_k^K}}{c\zeta_k^K}\,d\zeta_k \wedge d\overline{\zeta_k}\right|$$
$$\leq \left| \int_{R_z(t_n^{(k)})} h_j(\zeta_k)\frac{c_i\overline{\zeta_k^K}}{\zeta_k^{K+1}}\,d\zeta_k \wedge d\overline{\zeta_k}\right| + \left|\int_{R_z(t_n^{(k)})} \left| h_i(\zeta_k)h_j(\zeta_k)\right|\,d\zeta_k \wedge d\overline{\zeta_k}\right|.$$
By Lemma \ref{ZeroInSmallDisks} or \ref{ZeroInSmallDisks2}, depending on the value of $K$, the right-hand side of the inequality is bounded.

In the particular case when $K = -1$, we have
$$\left| \int_{R_z(t_n^{(k)})} \theta_i(0, \tau_{\infty}) \theta_j(0, \tau_{\infty}) \,d \mu_{\omega'}\right|$$
$$ = \left| \int_{R_z(t_n^{(k)})} \left(c_i/\zeta_k + h_i(\zeta_k)\right)h_j(\zeta_k)\frac{\overline{c}\zeta_k}{c\overline{\zeta_k}}\,d\zeta_k \wedge d\overline{\zeta_k}\right|$$
$$\leq \left| \int_{R_z(t_n^{(k)})} h_j(\zeta_k)\frac{c_i}{\overline{\zeta_k}}\,d\zeta_k \wedge d\overline{\zeta_k}\right| + \left| \int_{R_z(t_n^{(k)})}  h_i(\zeta_k)h_j(\zeta_k)\frac{\zeta_k}{\overline{\zeta_k}}\,d\zeta_k \wedge d\overline{\zeta_k}\right|.$$
By Lemma \ref{ZeroInSmallDisks}, both terms on the right-hand side of the inequality are zero.
\end{proof}

\begin{lemma}
\label{DPiDivergentTerm}
We follow the notation established above.  Let $\{(X(t_n, \tau_n), \omega_n)\}_{n=0}^{\infty}$ be a sequence of surfaces converging to a degenerate surface $(X(0,\tau_{\infty}), \omega')$.  For each $n$, let $\{\theta_1(t_n, \tau_n), \ldots, \theta_g(t_n, \tau_n)\}$ be a basis for the space of Abelian differentials on $X(t_n, \tau_n)$.  Given $i, j, k$, if $f_i(0, \tau_{\infty}, 0,0) = c_i \not= 0$, $f_j(0, \tau_{\infty}, 0,0) = c_j \not= 0$, and $A_{\infty}(0, \tau_{\infty}, 0,0) = c \not= 0$ on $D_k$, then for sufficiently large $n$,
$$\int_{R_z(t_n^{(k)})} \theta_i(0, \tau_{\infty}) \theta_j(0, \tau_{\infty}) \,d \mu_{\omega'} = c_ic_j\frac{\overline{c}}{c}(1-q)4\pi\sqrt{-1} \log |t_n^{(k)}| + O(1).$$
\end{lemma}

\begin{proof}
We have
$$\int_{R_z(t_n^{(k)})} \theta_i(0, \tau_{\infty}) \theta_j(0, \tau_{\infty}) \,d \mu_{\omega'}$$
$$ = \int_{R_z(t_n^{(k)})} \left(c_i/\zeta_k + h_i(\zeta_k) \right)\left(c_j/\zeta_k + h_j(\zeta_k) \right) \frac{\zeta_k}{\overline{\zeta_k}}\left(\overline{c}/c + H(\zeta_k, \overline{\zeta_k}) \right) \, d\zeta_k \wedge d\overline{\zeta_k},$$
where $h_i$ and $h_j$ are holomorphic, $H$ is analytic in both variables, and $H(0,0) = 0$.  It follows from Lemmas \ref{ZeroInSmallDisks} and \ref{ZeroInSmallDisks2} that every term is bounded uniformly for all $n$ with the exception of
$$c_ic_j\frac{\overline{c}}{c}\int_{R_z(t_n^{(k)})} \frac{1}{|\zeta_k|^2} \, d\zeta_k \wedge d\overline{\zeta_k} = -2c_ic_j\frac{\overline{c}}{c}\sqrt{-1}\int_0^{2\pi}\int_{|t_n^{(k)}|^{1-q}/c''}^{c'} \frac{1}{r^2} r\, dr \,d\theta$$
$$ = -4\pi c_ic_j\frac{\overline{c}}{c}\sqrt{-1}\left( \log(c') - \log(|t_n^{(k)}|^{1-q}/c'') \right)$$
$$ = c_ic_j\frac{\overline{c}}{c}(1-q)4\pi\sqrt{-1} \log |t_n^{(k)}| + O(1).$$
\end{proof}

\section{Surgery on Abelian Differentials}

We introduce a surgery on holomorphic Abelian differentials that associates them to integrable quadratic differentials.  Define $\mathcal{Q}_{g,n}^{(s)}$ to be the moduli space of integrable quadratic differentials with marked line segments of finite length, called \emph{slits}, on the surfaces in $\mathcal{Q}_{g,n}$.  The existence of such a moduli space can be seen by considering the moduli space of bordered Riemann surfaces with the additional information that there is a marked point on each of the boundary curves of the surface.  The idea is that this information can be used to glue opposite sides of the boundary curves relative to the marked point to get slits.  The surgery associates elements of $\mathcal{M}_g$ to $\cup_i \mathcal{Q}_{g_i',n_i'}^{(s)}$, where $g_i' \leq g$ and $n_i' \geq 0$, for all $i$.

\begin{definition}
Let $(X, \omega) \in \mathcal{M}_g$.  Let $S = \{\gamma_1, \ldots, \gamma_n\}$ be a set of pairwise non-homotopic closed regular trajectories of the vertical foliation of $X$ by $\omega$, with $n \geq 0$.  Let $C_i$ be the cylinder defined by the closure of the maximal set of closed regular leaves homotopic to $\gamma_i$.  Let $X^* = X \setminus \cup_i C_i$.  Choose antipodes with respect to the flat metric induced by $\omega$ in each of the $2n$ holes of $\overline{X^*}$ so that the antipodes lie at regular points of $(X,\omega)$ and for each hole identify the two semicircles.  This identification results in marked line segments called \emph{slits}.  We call this procedure the \emph{cylinder surgery} and let $\tilde X$ denote the (possibly disconnected) surface with slits resulting from performing the cylinder surgery.
\end{definition}

The surgery is well-defined up to a choice of antipodes.  In this paper we will only be concerned with the relationship between the \splin ~action and the cylinder surgery.  Since the foliation in which each slit is a subset of a leaf is invariant under the choice of antipodes, the choice of antipodes will not matter to us.  It is possible that the vertical foliation of $\omega$ is periodic in which case the cylinder surgery results in the empty set.  We exclude this case, whenever we perform the cylinder surgery because the surgery does not provide us with any useful information in this case.  Now we show that $\omega$ naturally induces a quadratic differential $\tilde q$ on $\tilde X$.  If $S = \emptyset$, then $\tilde X = X$ and let $\tilde q = \omega$.

\begin{lemma}
\label{TildeqExists}
If $\tilde X \not= \emptyset$, then $\omega$ induces a non-zero integrable quadratic differential on $\tilde X$ denoted by $\tilde q$.
\end{lemma}

\begin{proof}
Assume that $S \not= \emptyset$.  There is a natural inclusion $i: X^* \rightarrow X$ by the identity.  Let $\mathcal{F}$ denote the vertical foliation of $\omega$ on $X$.  This naturally pulls back to a foliation on $X^*$ under $i$.  Since the boundary of every hole of $X^*$ is a union of saddle connections by definition of the cylinder surgery, gluing opposite sides of the slits of $X^*$ results in a foliation of $\tilde X$, denoted $\mathcal{\tilde F}$, that is identical to $i^*(\mathcal{F})$ away from the slits.  In a sufficiently small neighborhood of any point $p$ in a slit which is not an endpoint or a zero of $\omega$, the foliation  looks like the vertical (or horizontal) foliation of $dz$.  By definition of the cylinder surgery, the ends of the slits locally look like the vertical foliation of $dz^2/z$.  

We define a double cover $\pi: \hat X \rightarrow \tilde X$, which the reader will recognize as the classical orientating double cover construction for a quadratic differential.  Let $\Sigma$ denote the union of the set of zeros of $\omega$ and the set of antipodes chosen in the cylinder surgery.  Let $(U_i, \phi_i)$ be an atlas for $\tilde X \setminus \Sigma$.  For each $U_i$ define $g_i^{\pm}(z) = \pm \sqrt{\phi_i(z)}$ on the open sets $V_i^{\pm}$ which are each a copy of $U_i$.  The charts $\{V_i^{\pm}\}$ can be glued together in a compatible way after filling in the holes of $\Sigma$.  This defines a surface $\hat X$ with a foliation $\mathcal{\hat F}$.  The reader will easily see that the foliation about the endpoints of the slits of $\tilde X$, at which the foliation induced on $\tilde X$ locally has the foliation of the simple pole of a quadratic differential, lifts to the foliation about a regular point on $\hat X$.

By \cite{HubbardMasur}[Main Theorem], the foliation $\mathcal{\hat F}$ induces a quadratic differential $\hat q$ on $\hat X$.  (Hubbard and Masur \cite{HubbardMasur} state their Main Theorem in terms of a horizontal foliation, but it can be stated for a vertical foliation as well simply by considering $\sqrt{-1}\,\hat q$.)  This allows us to view the construction above as the orientating double cover construction, which implies that $\hat q$ defines a quadratic differential $\pi_*(\hat q) = \tilde q$ on $\tilde X$ by pushforward.  It is obvious that $\tilde q$ is not the zero differential.
\end{proof}

Assume that $\tilde X \not= \emptyset$.  The cylinder surgery defines two maps: an injection $i: \overline{X^*} \hookrightarrow X$, which extends to a map on the Abelian differentials, and a ``gluing map'' $g: (\overline{X^*}, \omega^*) \rightarrow (\tilde X, \tilde q)$, where $g \mid_{X^*} = \text{id}$ and $g$ maps $\partial C_j$ to the slits of $\tilde X$ as prescribed by the cylinder surgery for all $j$.  We abuse notation and write $i: (\overline{X^*}, \omega^*) \hookrightarrow (X,\omega)$, where $\omega^*$ is the restriction of $\omega$ to $\overline{X^*}$.  This allows us to define a ``cylinder surgery map'' $P$ such that the diagram commutes.  Furthermore, given $i$ and $g$, $P$ can be inverted, so the cylinder surgery can be regarded as a set of maps $\{i, g\}$ associated to $(X,\omega)$.

$$\ctdiagram{
\ctv 0,50:{(\overline{X^*}, \omega^*)}
\ctv 100,50:{(X,\omega)}
\ctv 100,0:{(\tilde X, \tilde q)}
\ctet 0,50,100,50:{i}
\ctel 0,50,100,0:{g}
\ctdash\cter 100,50,100,0:{P}
}$$

\begin{lemma}
\label{poleSurgBij}
If $\tilde X \not= \emptyset$, then
$$G_t \cdot P(X,\omega) = P \circ G_t \cdot (X,\omega).$$
\end{lemma}

\begin{proof}
To prove this, we construct the following diagram and prove it commutes.  

$$\ctdiagram{
\ctv 0,150:{(\overline{X^*}, \omega^*)}
\ctv 150,150:{(X,\omega)}
\ctv 75,100:{(\tilde X, \tilde q)}
\ctv 0,50:{(\overline{X^*}_t, \omega^*_t)}
\ctv 150,50:{(X_t,\omega_t)}
\ctv 75,0:{(\tilde X_t, \tilde q_t)}
\ctet 0,150,150,150:{i}
\ctel 0,150,75,100:{g}
\cter 150,150,75,100:{P}
\ctet 0,50,150,50:{i_t}
\ctel 0,50,75,0:{g_t}
\cter 150,50,75,0:{P_t}
\cter 0,150,0,50:{\hat f_t}
\cter 150,150,150,50:{f_t}
\cter 75,100,75,0:{\tilde f_t}
}$$

The cylinder surgery denoted by $P$ is completely determined without any ambiguity by the maps $i$ and $g$ defined above.  The action of $G_t$ is well-defined on all three surfaces in the upper commutative triangle and induces quasiconformal maps on the surfaces which we denote by $\hat f_t$, $\tilde f_t$, and $f_t$ as indicated in the large diagram.

The map $g_t$ is well-defined because it is induced by the map $g$, which dictates a choice of antipodes.  In charts away from the holes of $\overline{X^*}$, or slits of $\tilde X$, $g = g_t = \text{Id}$ and the square trivially commutes.  Since there exists a quasiconformal map which is well-defined across the slit, namely $\tilde f_t$, the map $\hat f_t$ commutes with $g$ and $g_t$ by construction.

The map $i_t$ is well-defined because compact leaves of the horizontal and vertical foliations of $(X,\omega)$ are preserved under the action by $G_t$.  By quasiconformal continuation \cite{LehtoVirtanenQuasiMap}, the map $i_t \circ \hat f_t$ can be continued to a quasiconformal map $f'_t: (X,\omega) \rightarrow (X_t, \omega_t)$, such that $f'_t\mid_{i(\overline{X^*})} = f_t$.  Hence $i_t \circ \hat f_t = f_t \circ i$.  Since the cylinder surgery is completely determined by $i$ and $g$, the map $P_t$ is induced by $i_t$ and $g_t$, and the diagram commutes.

\end{proof}

\begin{proposition}
\label{MasurThmwithParSlits}
If $\tilde X \not= \emptyset$, then there exists $\theta \in \mathbb{R}$ such that $(\tilde X, e^{i\theta}\tilde q)$ has a closed regular trajectory that does not pass through the slits of $\tilde X$.
\end{proposition}

\begin{proof}
Since the slits lie in the vertical foliation of $\tilde q$ by definition, the lengths of the slits tend to zero in this direction as $t$ tends to infinity.  Let $\{t_n\}_{n=1}^{\infty}$ be a divergent sequence of times such that
$$\lim_{n \rightarrow \infty} G_{t_n} \cdot (\tilde X, e^{i\theta}\tilde q) = (\tilde X', \tilde q')$$
has a limit in $\overline{\mathcal{Q}_{g',n'}}$.  There are two cases to consider.  Either $\tilde q'$ has a double pole or it does not.  If $\tilde q'$ does have a double pole, then there is a cylinder with height tending to infinity with $t_n$.  Hence, there is a closed regular trajectory on this cylinder and it cannot cross a slit since the length of the slits are going to zero while the height of the cylinder is tending to infinity.

On the other hand, if $\tilde q'$ is integrable, then we claim it is holomorphic.  Since the cylinders are maximal sets, every slit will contain at least one zero of $\tilde q$.  Then $\tilde q'$ must be holomorphic because all of the simple poles of $\tilde q$ lie at the ends of slits by definition and every simple pole of $\tilde q$ has been contracted to a zero of $\tilde q$.  Since $\tilde q'$ is holomorphic, we can apply \cite{MasurClosedTraj}[Theorem 2] to find a dense set of closed trajectories on $(\tilde X',\tilde q')$.  Each of these trajectories correspond to a cylinder.  Let $C$ denote one such cylinder.  Then $C$ also represents a cylinder on $G_{t_n} \cdot (\tilde X, e^{i\theta}\tilde q)$ for large $n$.  Choose $N > 0$ so that $C$ has height $h$ and the total length of the slits is $\varepsilon > 0$ and $h >> \varepsilon$.  There is a regular trajectory corresponding to a waist curve $\gamma$ of $C$ at time $t_N$ such that $\gamma$ does not intersect any of the slits.
\end{proof}

\section{Complete Periodicity and the Connectivity Graph in \RankOne}

The key results of this section are Theorem \ref{RankOneImpCP} and Lemma \ref{TDDerPerRank1Lem}.  They form the foundation on which the remainder of this paper rests.  The former result proves that every surface generating a Teichm\"uller disc in the rank one locus must be completely periodic, while the latter result describes the configuration of the parts of a degenerate surface in the closure of a Teichm\"uller disc contained in the rank one locus.  We begin by recalling some basic definitions from graph theory.

Let $G$ be a graph consisting of a vertex set $V(G)$ and an edge set $E(G)$.  A \emph{path} is a graph with vertex set $\{v_1, \ldots, v_n\}$ such that there is an edge from $v_i$ to $v_{i+1}$, for all $1 \leq i \leq n-1$.  A \emph{cycle} is a path with an additional edge connecting $v_1$ to $v_n$.  Consider the set of all cycles contained in $G$.  This set forms a finite dimensional vector space over the field $\mathbb{F}_2$ called the \emph{cycle space of $G$}.  Denote the dimension of the cycle space by $\dim^C(G)$.  All the graphs in the discussion below may be multigraphs, i.e. we permit multiple edges between the same pair of vertices and there may be edges from a vertex to itself.

\begin{definition}
Let $G(X')$ be a multigraph associated to $X'$ or simply $G$ when the surface is understood.  There is a bijection sending $V(G)$ to the parts of $X'$ by $v_i \mapsto S_i$.  For all $i, j$ and all pairs of punctures $(p,p')$ from parts $S_i$ to $S_j$ of $X'$, with $i$ not necessarily distinct from $j$, there is a unique edge of $G$ from $v_i$ to $v_j$ representing $(p,p')$.  The graph $G$ is called the \emph{connectivity graph}.  Let $G^P(X', \omega')$ be the subgraph of $G(X')$ such that $V(G^P) = V(G)$ and the edges of $G^P$ correspond to the pairs of punctures at which $\omega'$ has simple poles.
\end{definition}

\begin{remark}
We will be using Lemma \ref{Forni42Seqs} implicitly throughout this section.  It is extremely important to note that nowhere in these results do we require that every component of the derivative of the period matrix has a limit as we take sequences in $\mathcal{M}_g$ converging to a degenerate surface.  We are very careful to choose minors of the derivative of the period matrix such that the limit exists.  This will suffice to provide the requisite lower bounds on the rank of the derivative of the period matrix near the boundary of the moduli space.
\end{remark}

Throughout this section, it will be advantageous to choose a basis of Abelian differentials with very specific properties depending on the surface to which a sequence of Abelian differentials is converging.  Most importantly, the choice of basis we make in the following lemma will facilitate the application of the convergence lemmas from Section \ref{DerPerMatSubSect}.

\begin{lemma}
\label{AbDiffBasis}
Given a degenerate surface $X(0,\tau_{\infty})$ in the boundary of $\overline{\mathcal{R}_g}$, there exists a set of Abelian differentials $\{\theta_1(0,\tau_{\infty}), \ldots, \theta_g(0,\tau_{\infty})\}$ on $X(0,\tau_{\infty})$ such that for all $t = (t_1, \ldots, t_m)$, with $t_j \not= 0$ for all $j$, $\{\theta_1(t,\tau), \ldots, \theta_g(t,\tau)\}$ is a basis for the space of holomorphic Abelian differentials on $X(t,\tau)$.  Moreover, this set can be constructed so that $\{\theta_1(0,\tau_{\infty}), \ldots, \theta_g(0,\tau_{\infty})\}$ has the following properties:

\begin{itemize}
\item[(1)] For some $1 \leq g_1 \leq g$, $\theta_i(0,\tau_{\infty})$ is holomorphic if and only if $1 \leq i \leq g_1$.
\item[(2)] For all $(p_i, p_i')$ such that $(p_i, p_i') \in S$ for some part $S \subset X'$, $\theta_i(0,\tau_{\infty})$ has simple poles at $(p_i, p_i')$, $\theta_i(0,\tau_{\infty})$ is holomorphic across all other punctures of $S$, and $\theta_i(0,\tau_{\infty}) \equiv 0$ on $X' \setminus S$.
\item[(3)] For each cycle $C_i \in G(X')$ consisting of more than one edge, $\theta_i(0,\tau_{\infty})$ has poles at the pairs of punctures corresponding to the edges of $C_i$ and $\theta_i \equiv 0$ for all $S \subset X'$ such that $S$ does not correspond to a vertex of $C_i$.
\item[(4)] For any puncture $p \in X'$ and for all $i,j$, if $\text{Res}_p(\theta_i) \not= 0$ and $\text{Res}_p(\theta_j) \not= 0$, then $\text{Res}_p(\theta_i) = \text{Res}_p(\theta_j) = \pm 1$.
\end{itemize}
\end{lemma}

\begin{proof}
The first claim follows from the Cartan-Serre theorem or \cite{MasurExt}[Proposition 4.1].  We proceed by explicitly constructing a basis of Abelian differentials on $X'$ with the desired properties.  The first $g_1$ differentials can be taken as a union of the bases of holomorphic differentials on each part such that if $\theta_i$ is an element of the basis of Abelian differentials on a part $S \subset X'$, then define $\theta_i \equiv 0$ on $X' \setminus S$.

Let the parts of $X'$ be given by $S_1 \sqcup \cdots \sqcup S_n$.  By \cite{FarkasKra}[Theorem II.5.1 b.], given two punctures $(p,p')$ on a connected Riemann surface $S$, there exists a meromorphic Abelian differential on $S$ which is holomorphic everywhere on $S$ and across all punctures of $S$ except $p$ and $p'$, where it can be expressed as $dz/z$ and $-dw/w$, in terms of local coordinates $z$ and $w$, respectively.  Hence, for each part $S_j$ carrying a pair of punctures $(p,p')$ we can take a basis element to be a differential which has simple poles only at those two punctures and is zero on every other part.  Let the basis of Abelian differentials on $X'$ consist of $g_2$ such differentials with exactly two simple poles, where $0 \leq g_2 \leq g$.

Finally, let $G_1$ be the subgraph of $G(X',\omega')$ such that $G_1$ has no edges from a vertex to itself.  We claim $\dim^C(G_1) = g - g_1 - g_2$.  This follows because each basis differential on $X'$ corresponds to a closed horizontal homology curve on a surface near $X'$ in the interior of the moduli space $\mathcal{R}_g$.  The only horizontal homology curves that have not been accounted for in the description above are those that split over several parts.  Define the remaining basis differentials as follows.  For each $j$, with $0 \leq j \leq g - g_1 - g_2$, let $C_j$ be an element of the cycle basis of $G$.  Define $\theta_j$ to be zero on every part which does not correspond to a vertex of $C_j$.  Each vertex $v$ of $C_j$ corresponds to a part $S$ of $X'$ such that $S$ has two punctures $p_1$ and $p_2$ corresponding to edges of $C_j$ incident to $v$.  The punctures $p_1$ and $p_2$ are not paired.  By \cite{FarkasKra}[Theorem II.5.1 b.], there is a meromorphic differential holomorphic everywhere on $S$ and across all punctures of $S$ except for $p_1$ and $p_2$ at which it has simple poles with residues $1$ and $-1$, respectively.  Define the differential $\theta_j$ to have two poles on each part corresponding to a vertex in the cycle $C_j$.  The only restriction is given by the rule that if the residue of the simple pole at $p_1$ is $\pm 1$, then the residue of the simple pole at $p_1'$ is $\mp 1$.  This construction completes the proof that such a basis exists.

By construction, the residues of each differential at every pole are $\pm 1$.  In order to satisfy the final property, it may be necessary to multiply some of the differentials by $-1$ so that the residues at each puncture are equal.
\end{proof}

\begin{lemma}
\label{PairPolesMakeRank}
Let $\{(X_n, \omega_n)\}_{n=0}^{\infty}$ be a sequence of surfaces in a Teichm\"uller disc $D$ converging to a degenerate surface $(X', \omega')$.  Let $S \subset X'$ be a part of $X'$.  If $\omega'$ has $k_1$ pairs of poles on $S$, then
$$\sup_n \text{Rank}\left(\frac{d\Pi(X_n)}{d\mu_{\omega_n}}\right) \geq k_1.$$
\end{lemma}

\begin{proof}
We show that a single pair of poles on $X'$ corresponds to a divergent diagonal term of $d\Pi(X_n)/d\mu_{\omega_n}$ as $n$ tends to infinity, while the off-diagonal terms in the row and column of that unbounded diagonal term are bounded for all $n$.  Let $b_{ij}^{(n)}$ be the $ij$ component of $d\Pi(X_n)/d\mu_{\omega_n}$.  Let $(p_i,p_i')$ be a pair of punctures on $S$ such that $\omega'$ has a pair of poles at $(p_i,p_i')$, for $1 \leq i \leq k_1$.  As in Lemma \ref{AbDiffBasis}, let $\theta_i$ have a pair of poles with residue $\pm 1$ at $(p_i,p_i')$ and let $\theta_i$ be holomorphic everywhere else on $X'$, for $1 \leq i \leq k_1$.  We consider the $k_1 \times k_1$ minor of $d\Pi(X_n)/d\mu_{\omega_n}$ given by $(b_{ij}^{(n)})$, for $1 \leq i,j \leq k_1$, and show that it has full rank for sufficiently large $n$.  By Lemmas \ref{BddOutsideDisks} and \ref{BddInDisks}, all of the off-diagonal terms $b_{ji}^{(n)} = b_{ij}^{(n)}$ are bounded, for all $n$, because $\theta_i$ and $\theta_j$ do not have any poles at the same pair of punctures for $i \not= j$.  Furthermore, for each $i$, the contribution of the integral in Rauch's formula to the diagonal term $b_{ii}^{(n)}$ is bounded everywhere outside of the discs around $p_i$ and $p_i'$ by Lemmas \ref{BddOutsideDisks} and \ref{BddInDisks}.  By Lemma \ref{DPiDivergentTerm}, the contribution to the integral in Rauch's formula on $R_z(t_n^{(k)})$ diverges with $n$.  Recall that if $\omega'$ has residue $c$ at $p_i$, then it has residue $-c$ at $p_i'$.  Since the quotient $\bar c/c = -\bar c / -c$, the sum of the two divergent terms coming from Lemma \ref{DPiDivergentTerm} do not cancel and $b_{ii}^{(n)}$ diverges to infinity with $n$.
\end{proof}

\begin{lemma}
\label{CycleImpliesRank}
Let $\{(X_n, \omega_n)\}_{n=0}^{\infty}$ be a sequence of surfaces in a Teichm\"uller disc $D$ converging to a degenerate surface $(X', \omega')$.  Let $G'^P$ be the subgraph of $G^P$ formed by removing all edges from each vertex to itself.  Let $k_2 = \min(\dim^C(G'^P), 2)$.  Then
$$\sup_n \text{Rank}\left(\frac{d\Pi(X_n)}{d\mu_{\omega_n}}\right) \geq k_2.$$
\end{lemma}

\begin{proof}
If $\dim^C(G'^P) = 0$, we are done.  If $\dim^C(G'^P) = 1$, then we claim that $d\Pi(X_n)/d\mu_{\omega_n}$ is not the zero matrix, for some choice of $n$.  Let $\theta_1$ be the differential with poles along the cycle of $G'^P$.  Let $(p_1, p_1')$ be a pair of poles of $\omega'$ in the cycle.  The claim follows from Lemma \ref{DPiDivergentTerm} by letting $c_1 = \pm 1$, $\lim_{n \rightarrow \infty} c^{(n)} = c_1 = \pm 1$, where $c^{(n)}$ is the residue of $\omega_n$ in local coordinates about $p_1$, and considering the $1,1$ component of $d\Pi(X_n)/d\mu_{\omega_n}$.

Assume $\dim^C(G'^P) \geq 2$.  Let $C \subset G'^P$ be a cycle.  Using Lemma \ref{MostlyImagRes} assume that the residues of $\omega'$ are $\delta$-nearly imaginary.  It can be shown that given $\varepsilon > 0$, there exists $\delta > 0$ such that, for all $c \in \mathbb{C}$ that are $\delta$-nearly imaginary
$$\left| \frac{\bar c}{c} + 1 \right| < \varepsilon.$$
Hence, the coefficients of the unbounded $\log|t_n^{(k)}|$ terms in Lemma \ref{DPiDivergentTerm}, for all $k$, differ from each other by at most $2\varepsilon$.

By Lemma \ref{AbDiffBasis}, there is a basis $\{\theta_1, \ldots, \theta_g\}$ such that for all $1 \leq i \leq g$, $\theta_i$ has residue $\pm 1$ at all of its simple poles.  Without loss of generality, let $\theta_1$ be an element of the basis of Abelian differentials that has pairs of simple poles corresponding to all of the edges of $C$.  Again, let $b_{ij}^{(n)}$ denote the $ij$ component of the derivative of the period matrix on $X_n$ with respect to $\omega_n$.  By Lemma \ref{BddOutsideDisks}, the integral in Rauch's formula for the derivative of the period matrix is bounded outside of all discs around the punctures of $X'$.  However, it is possible that two different elements in the basis of differentials have simple poles at the same pairs of punctures at which $\omega'$ has a simple pole.

Let $C' \subset G'^P$ be a cycle distinct from $C$ (though it may have non-trivial intersection with $C$).  Let $\theta_2$ be the differential with poles at the pairs of punctures corresponding to edges of $C'$.  Every edge of both $C$ and $C'$ corresponds to a pair of poles of $\omega'$.  (Note that Lemma \ref{Forni42Seqs} guarantees that we can apply all of the lemmas of Section \ref{DerPerMatSubSect} to the $2 \times 2$ minor $(b_{ij}^{(n)})$, for $1 \leq i,j \leq 2$, because $\omega'$ has poles at every puncture where $\theta_1$ or $\theta_2$ have poles.)  We claim that for all $n$ sufficiently large, $|b_{11}^{(n)}| > |b_{12}^{(n)}| = |b_{21}^{(n)}|$.  Lemma \ref{DPiDivergentTerm} implies that each of these three terms is a sum of divergent terms.  However, $\sharp(E(C \cap C')) < \sharp(E(C))$ implies that $b_{12}^{(n)}$ is a sum of fewer divergent terms than $b_{11}^{(n)}$, and there is no cancellation between the divergent terms by the $\delta$-nearly imaginary assumption.  For the exact same reason, $|b_{22}^{(n)}| > |b_{12}^{(n)}| = |b_{21}^{(n)}|$.  Thus the diagonal term of each row and column is strictly larger than the off-diagonal terms in its row and column, for $n$ sufficiently large.  This implies that the derivative of the period matrix has a $2 \times 2$ minor of full rank.
\end{proof}

\begin{lemma}
\label{GPisCycle}
Let $D$ be a Teichm\"uller disc contained in \RankOne .  If $(X', \omega')$ is a degenerate surface in the closure of $D$ and $\omega'$ is not holomorphic, then $G^P(X',\omega')$ is the union of a cycle (possibly on one or two vertices) and a finite (possibly empty) set of isolated vertices.
\end{lemma}

\begin{proof}
Since every Abelian differential with a simple pole on a Riemann surface $S$ has at least two simple poles on $S$, no vertex in $G^P$ has degree one.  Using the notation of Lemmas \ref{PairPolesMakeRank} and \ref{CycleImpliesRank}, we must have $k_1 + k_2 \leq 1$.  The case where $k_1 + k_2 = 0$ is excluded by the assumption that $\omega'$ in not holomorphic, so we assume $k_1 + k_2 = 1$.  If $k_1 = 1$, then $G^P$ has a vertex with an edge forming a loop and Lemmas \ref{PairPolesMakeRank} and \ref{CycleImpliesRank} imply that there are no other edges.  If $k_2 = 1$, then $G^P$ contains a cycle $C$.  However, we claim $G^P$ cannot contain any other edges.  There are no additional paths in $G^P$ between any two vertices in $C$ because $k_2 = 1$.  Since $k_1 = 0$ implies there are no vertices from an edge to itself, there are no additional paths emanating from a vertex in $C$ because any such path would have to end in a vertex of degree one.  Hence, $k_2 = 1$ implies $E(G^P) = E(C)$.
\end{proof}

\begin{definition}
Given $(X,\omega)$, let $\mathcal{F}_{\theta}$ denote the vertical foliation of $(X,e^{i\theta}\omega)$.  If, for all $\theta \in \mathbb{R}$, $\mathcal{F}_{\theta}$ has a closed regular trajectory implies that $\mathcal{F}_{\theta}$ is periodic, then $(X,\omega)$ is \emph{completely periodic}.
\end{definition}

\begin{theorem}
\label{RankOneImpCP}
If the Teichm\"uller disc $D$ generated by $(X,\omega)$ is contained in \RankOne , then $(X,\omega)$ is completely periodic.
\end{theorem}

\begin{proof}
%Setup and introduction of surgery

By \cite{MasurClosedTraj}[Theorem 2], there exists a real number $\theta$ such that $(X,e^{i\theta}\omega)$ admits a cylinder in the vertical foliation.  Without loss of generality, let $(X,\omega)$ admit a cylinder $C_1$ in its vertical foliation.  Performing the cylinder surgery on $(X,\omega)$ results in the set $\tilde X$.  By contradiction, assume that $\tilde X \not = \emptyset$, i.e. the vertical foliation of $X$ by $\omega$ is not periodic.  Then the cylinder surgery results in a surface $(\tilde X, \tilde q)$ carrying an integrable quadratic differential with slits.  Recall that $\tilde q$ has simple poles at the ends of the slits.  Since the slits correspond to cylinders in $X$ with parallel core curves, the slits themselves are parallel, in other words, they are leaves of the same foliation.  By acting on $(\tilde X, \tilde q)$ by the Teichm\"uller geodesic flow, the slits, which lie in the vertical foliation, contract at the maximum rate $e^{-t}$ under the area normalization.  Consider the one parameter family of surfaces $G_t \cdot (\tilde X, \tilde q)$ for all $t \geq 0$.  From this family, choose a sequence of times $\{t_n\}_{n=0}^{\infty}$, with $t_0 = 0$, such that
$$\lim_{n \rightarrow \infty} G_{t_n} \cdot (\tilde X, \tilde q) = (\tilde X', \tilde q'),$$
where $\tilde X'$ is either a degenerate Riemann surface or $\tilde q'$ is holomorphic.  If $\tilde q'$ is holomorphic, this implies that the slits contracted to points.  By the definition of the cylinder surgery, $\tilde q$ will always have a zero on every slit.  In other words, every simple pole of $\tilde q$ converged to a zero, resulting in a holomorphic quadratic differential.  Let
$$(\tilde X_n, \tilde q_n) = G_{t_n} \cdot (\tilde X, \tilde q).$$

Equivalently, consider the sequence $\{(X_n, \omega_n)\}_{n=0}^{\infty}$ defined by
$$(X_n, \omega_n) = G_{t_n} \cdot (X,\omega).$$
By Lemma \ref{poleSurgBij}, we can freely pass between $(X_n, \omega_n)$ and $(\tilde X_n, \tilde q_n)$ for all $n$.  Let $(X', \omega')$ denote the degenerate surface to which $\{(X_n, \omega_n)\}_{n=0}^{\infty}$ converges as $n$ tends to infinity.  Let $\{C_1^{(n)}\}_{n=0}^{\infty}$ denote the sequence of cylinders such that $C_1^{(0)} = C_1$ and $C_1^{(n)}$ is the corresponding cylinder in $X_n$ after action by $G_{t_n}$.  By Corollary \ref{GtCurvePinch}, the core curve of $C_1^{(n)}$ pinches as $n$ tends to infinity.  Let $C_1'$ denote the cylinder to which $C_1^{(n)}$ converges as $n$ tends to infinity, if such a cylinder exists.  It is possible that $\omega'$ is holomorphic at the pair of punctures resulting from pinching the core curves of the cylinder $C_1^{(n)}$, in which case the surface $(X', \omega')$ does not have an infinite cylinder at that pair of punctures, i.e. $C_1'$ does not exist.  On the other hand, if $C_1'$ exists, then it is an infinite cylinder.

We claim that there is a sequence of cylinders $\{C_2^{(n)}\}_{n=0}^{\infty}$ such that $C_2^{(n)} \subset \tilde X_n$, for all $n \geq 0$, and the sequence converges to a cylinder $C_2' \subset \tilde X'$ such that $C_2'$ does not intersect the slits of $\tilde X'$ and $C_2'$ has finite circumference.  It suffices to find a cylinder $C_2' \subset \tilde X'$ with the desired properties.  We consider four cases.  In the first case, the slits have contracted to points and $\tilde q'$ is holomorphic.  In this case we can find a cylinder $C_2'$ on $(\tilde X', \tilde q')$ by \cite{MasurClosedTraj}[Theorem 2].  If the slits have contracted to points and $\tilde q'$ has double poles, then we let $C_2'$ be the infinite cylinder corresponding to a pair of double poles.  If the slits have positive length and $\tilde q'$ is integrable, then Proposition \ref{MasurThmwithParSlits}, guarantees that we can find a cylinder $C_2'$ not intersecting the slits.  Finally, if the slits have positive length and $\tilde q'$ has double poles, then, as before, we let $C_2'$ be the infinite cylinder corresponding to a pair of double poles.

We claim that $\tilde q'$, as defined above, can never have simple poles, though it may have double poles.  In other words, the length of every slit on $\tilde X$ will always converge to zero as $n$ tends to infinity regardless of our choice of normalization.  By contradiction, assume that the lengths of the slits on $\tilde q'$ have nonzero length.  Let $w_j^{(n)}$ denote the circumference of $C_j^{(n)}$, and $w_j'$ denote the circumference of $C_j'$, for $j = 1,2$.  In this case $w_j' > 0$ for $j = 1,2$.  Consider the ratio $w_1^{(n)}/w_2^{(n)}$.  By assumption,
$$\lim_{n \rightarrow \infty} w_1^{(n)}/w_2^{(n)} > C > 0.$$
Pass to a subsequence of times $\{t_n\}_{n=0}^{\infty}$ such that there is a constant $C^L$ satisfying $0 < C^L \leq w_1^{(n)}/w_2^{(n)}$, for all $n$.  Recall that under the area normalization, the length of the slit contracts by $e^{-t_n}$ for each $n$.  Therefore, 
$$\lim_{n \rightarrow \infty} e^{t_n}w_1^{(n)} = w_1' < \infty,$$
and $w_1' > 0$ by assumption.  This implies $w_2' < \infty$ because $e^{t_n}$ is the maximal rate of expansion and
$$w_2' \leq \lim_{n \rightarrow \infty} e^{t_n}w_2^{(n)} < \infty.$$
Hence, the core curves of the cylinders $C_1^{(n)}$ and $C_2^{(n)}$ contract for all $n$ at the maximal rate under the area normalization.  This is only possible if the core curves of $C_1^{(n)}$ and $C_2^{(n)}$ are parallel for all $n$.  Otherwise, there would be an $N > 0$ sufficiently large, such that $w_2^{(n)}$ increases exponentially for all $n \geq N$.  However, it was assumed above that $C_1^{(0)}$ and $C_2^{(0)}$ are not parallel because all cylinders parallel to $C_1^{(0)}$ were removed from the surface.  This contradiction implies that the length of every slit must indeed converge to zero.

Since each slit converges to a point, the cylinder $C_1'$ does not exist.  If $C_2'$ has finite height, pinch the core curve of the cylinder $C_2'$ under the Teichm\"uller geodesic flow while normalizing the largest residue.  The new degenerate surface, denoted $(X', \omega')$ by abuse of notation, either has (Case A:) poles resulting from an infinite cylinder $C_2'$, or (Case B:) neither $C_1'$ nor $C_2'$ exist.  By the continuity of the \splin ~action to the boundary of the moduli space \cite{BainbridgeMoller}[Proposition 11.1], there is a sequence $\{(X_n, \omega_n)\}_{n=0}^{\infty}$ in $D$ converging to $(X', \omega')$.  We address Cases A and B in the course of the remainder of the proof.

By Lemma \ref{HomCylImpParLem}, $C_1^{(0)}$ is not homologous to $C_2^{(0)}$ because $C_1^{(0)}$ is not parallel to $C_2^{(0)}$.  Since the \splin ~action preserves homology, $C_1^{(n)}$ is not homologous to $C_2^{(n)}$ for all $n \geq 0$.  The remainder of this proof is dedicated to finding a degenerate surface $(X', \omega')$ in the closure of $D$ such that $G^P(X', \omega')$ contradicts the conclusion of Lemma \ref{GPisCycle}.

Consider the case when $\omega'$ has one or more pairs of simple poles arising from pinching a set of cylinders that are pairwise homologous.  In this case, let $C_2'$ be an infinite cylinder, while $C_1'$ does not exist because the circumferences of the cylinders in the sequence $\{C_1^{(n)}\}_{n=0}^{\infty}$ converge to zero.  Given $\varepsilon' > 0$, we can find a surface $(X^{(n)}, \omega^{(n)}) \in D$, where $n$ depends on $\varepsilon'$, such that $(X^{(n)}, \omega^{(n)})$ has two non-homologous cylinders of equal circumference at most $\sqrt{\varepsilon'}$ and the moduli of the cylinders tend to infinity as $\varepsilon'$ tends to zero.  Choose $\varepsilon < \varepsilon'$ such that the circumference of $C_1^{(N)}$ is equal to $\varepsilon$ for a sufficiently large value of $N$.  Since the sequence $\{C_2^{(n)}\}_{n=0}^{\infty}$ converges to a cylinder of finite nonzero circumference, the circumferences of the cylinders $C_2^{(n)}$, denoted $w_2^{(n)}$ satisfy $0 < w_2^L \leq w_2^{(n)} \leq w_2^U < \infty$, for all $n$.  The core curves of $C_1^{(n)}$ and $C_2^{(n)}$ are not parallel for all $n$, so for each $n$ there exists a matrix $B_n \in $ \splin ~that transforms the core curve of $C_1^{(n)}$ into a leaf of the vertical foliation and transforms the core curve of $C_2^{(n)}$ into a leaf of the horizontal foliation.  For each $N$, consider the one parameter family of matrices, $G_t B_N \in $ \splin .  Action by $G_t B_N$ on $(X_N, \omega_N)$ results in the core curve of $C_1^{(N)}$ expanding at the maximal rate $e^t$, while the core curve of $C_2^{(N)}$ contracts at the maximal rate $e^{-t}$.  At time $t$, the circumference of $C_1^{(N)}$ is given by $e^t \varepsilon$, and the circumference of $C_2^{(N)}$ is given by $e^{-t} w_2^{(N)}$.  Let $T_N$ be the time satisfying the equation $e^{T_N} \varepsilon = e^{-T_N} w_2^{(N)}$.  At time $T_N$, the circumference of each cylinder is given by $\sqrt{w_2^{(N)} \varepsilon}$.  Define a sequence by
$$(X^{(N)}, \omega^{(N)}) := G_{T_N}B_N \cdot (X_N, \omega_N)$$
and consider $C_1^{(N)}, C_2^{(N)}$ to be cylinders in $X^{(N)}$.  We claim the moduli of $C_1^{(N)}$ and $C_2^{(N)}$ diverge to infinity with $N$.  Let $h$ denote the height of a cylinder $C$, $w$ its circumference, $A(C)$ its area, and $\text{Mod}(C)$ its modulus.  By the definition of the modulus,
$$\text{Mod}(C) = \frac{h}{w} = \frac{A(C)}{w^2}.$$
In the case at hand, the areas of the cylinders $C_1^{(N)}$ and $C_2^{(N)}$ are bounded below for all $N$ because \splin ~preserves area.  Both cylinders have circumference $\sqrt{w_2^{(N)} \varepsilon}$, so their core curves pinch because
$$\lim_{\varepsilon' \rightarrow 0} \sqrt{w_2^{(N)} \varepsilon} \leq \lim_{\varepsilon' \rightarrow 0} \sqrt{w_2^U \varepsilon'} = 0.$$
Note that this argument can be applied to Case A above.  Let $(X'^{(2)}, \omega'^{(2)})$ be the limit of the sequence $\{(X^{(N)}, \omega^{(N)})\}_{N=0}^{\infty}$.  As $N$ tends to infinity, the cylinders $C_1^{(N)}$ and $C_2^{(N)}$ degenerate to cylinders of equal circumference.  If that circumference is non-zero, then $\omega'^{(2)}$ has two pairs of simple poles coming from non-homologous cylinders.  By Lemma \ref{GPisCycle}, $G^P(X', \omega')$ has a cycle with the pair of punctures represented by $C_2'$ corresponding to an edge of $G^P$.  Since cylinders with pinched core curves remain pinched under this procedure, $G^P(X'^{(2)}, \omega'^{(2)})$ must contain an edge $e$ corresponding to $C_1'$ in addition to the cycle of $G^P(X', \omega')$.  It is impossible for $e$ and the edges of $G^P(X', \omega')$ to be part of a larger cycle in $G^P(X'^{(2)}, \omega'^{(2)})$ because that would imply that $e$ represents a cylinder whose core curve, a posteriori, must be parallel to the core curves of the cylinders represented by the edges of $G^P(X', \omega')$.  This contradicts Lemma \ref{GPisCycle}.  However, it is still possible that the circumferences of both cylinders converge to zero in which case neither $C_1'$ nor $C_2'$ exist and $\omega'^{(2)}$ is holomorphic at both pairs of punctures.  We address this possibility.

By Lemma \ref{NonHoloPart}, we can assume without loss of generality, that $\omega'^{(2)}$ has a pair of simple poles.  We proceed by induction, where each step of the induction is to perform the argument of the preceding paragraph until we reach a contradiction.  The first step is already done.  We present the $j^{\text{th}}$ step of the procedure.  Let $\{(X_n, \omega_n)\}_{n=0}^{\infty}$ denote the sequence of surfaces converging to a degenerate surface $(X'^{(j)}, \omega'^{(j)})$ such that $(X_n, \omega_n)$ has $j$ pairwise non-homologous cylinders all of whose circumferences converge to zero while another sequence of cylinders $\{C_{j+1}^{(n)}\}_{n=0}^{\infty}$ converges to a pair of poles of $\omega'^{(j)}$.  Let $\{C_k^{(n)}\}_{n=0}^{\infty}$, for $1 \leq k \leq j$, denote the $j$ distinct sequences of cylinders whose circumferences converge to zero as $n$ tends to infinity.  Without loss of generality, let $\{C_1^{(n)}\}_{n=0}^{\infty}$ be a sequence of cylinders such that for infinitely many values of $n$ and all $k \not = 1$, the circumference of $C_k^{(n)}$ is less than or equal to the circumference of $C_1^{(n)}$.  This may require the sequences to be renamed.  We pass to a subsequence such that this holds for all $n$.  Recall that $\varepsilon' > 0$ was fixed in the preceding paragraph and an appropriate $\varepsilon > 0$ was chosen.  It will become apparent that the circumference of the cylinder $C_1^{(n)}$ is $w_1^{(n)} \varepsilon^{1/(2^j)}$, where $w_1^{(n)}$ is a constant satisfying $0 < w_1^L \leq w_1^{(n)} \leq w_1^U < \infty$ for all $n$.  Let $w_{j+1}^{(n)}$ denote the circumference of $C_{j+1}^{(n)}$, which also satisfies $0 < w_{j+1}^L \leq w_{j+1}^n \leq w_{j+1}^U < \infty$ for all $n$.  We highlight the differences that arise in the course of repeating the argument of the preceding paragraph.  Solving the equation $e^{T_N} w_1^{(N)} \varepsilon^{1/(2^j)} = e^{-T_N} w_{j+1}^{(N)}$ shows that at time $T_N$ the lengths of the circumferences are $\sqrt{w_{j+1}^{(N)} w_1^{(N)}} \varepsilon^{1/(2^{j+1})}$.  To see that the core curves of all $j+1$ cylinders still pinch as $\varepsilon'$ tends to zero, note that, as before, the areas of all of the cylinders are fixed under the \splin ~action and thus their areas are bounded from below.  Finally,
$$\lim_{\varepsilon' \rightarrow 0} \sqrt{w_{j+1}^{(N)} w_1^{(N)}} \varepsilon^{1/(2^{j+1})} \leq \lim_{\varepsilon' \rightarrow 0} \sqrt{w_{j+1}^U w_1^U} \varepsilon'^{1/(2^{j+1})} = 0.$$
Note that this induction procedure includes Case B that was left unaddressed above.  Let $(X'^{(j+1)}, \omega'^{(j+1)})$ denote the degenerate surface formed by letting $N$ tend to infinity in the sequence $\{G_{T_N} B_N \cdot (X_N, \omega_N)\}_{N=0}^{\infty}$.  As above, the cylinders $C_1^{(N)}$ and $C_{j+1}^{(N)}$ degenerate to cylinders of equal circumference.  If that circumference is non-zero, then $\omega'^{(j+1)}$ has at least two pairs of simple poles coming from non-homologous cylinders, namely $C_1'$ and $C_{j+1}'$.  By Lemma \ref{GPisCycle}, $G^P(X'^{(j)}, \omega'^{(j)})$ has a cycle with the pair of punctures represented by $C_{j+1}'$ corresponding to an edge of $G^P$.  Since cylinders with pinched core curves remain pinched under this procedure, $G^P(X'^{(j+1)}, \omega'^{(j+1)})$ must contain an edge $e$ corresponding to $C_1'$ in addition to the cycle from $G^P(X'^{(j)}, \omega'^{(j)})$.  As before, $e$ and the edges of $G^P(X'^{(j)}, \omega'^{(j)})$ cannot be edges of a larger cycle.  This contradicts Lemma \ref{GPisCycle}.  However, it is still possible that the circumferences of all $j+1$ cylinders converge to zero in which case $\omega'^{(j+1)}$ is holomorphic at $j+1$ pairs of punctures.  In that case, repeat this argument.

This procedure must terminate at worst when $j = g$ because the core curves of the cylinders chosen at each step are pairwise non-homologous, and there are at most $g$ such curves.  Hence, performing this procedure at the $g-1$ iteration guarantees at least two poles from sequences of non-homologous cylinders and results in a contradiction.  This contradiction demonstrates that $\tilde X$ must in fact be the empty set.  In other words, the surface is filled by cylinders, and the vertical foliation of $X$ by $\omega$ is periodic.  Since this argument holds for all $\theta \in \mathbb{R}$ such that $(X, e^{i\theta}\omega)$ admits a cylinder in the vertical foliation, $(X,\omega)$ is completely periodic.
\end{proof}

Theorem \ref{RankOneImpCP} is used implicitly in the following corollary to guarantee that it is not a vacuous statement.  Compare this statement with \cite{MollerShimuraTeich}[Lemma 5.3].

\begin{corollary}
\label{RankOneCylConfig}
Let $(X,\omega)$ generate a Teichm\"uller disc $D \subset $ \RankOne .  For each $\theta \in \mathbb{R}$ such that the vertical foliation of $(X, e^{i\theta}\omega)$ is periodic, $(X, e^{i\theta}\omega)$ decomposes into a union of cylinders $C_1, \ldots, C_k$ such that all of the saddle connections on the top of $C_i$ are identified to the saddle connections on the bottom of $C_{i+1}$ and vice versa, for all $i \leq k-1$, and all of the saddle connections on the top of $C_k$ are identified to the saddle connections on the bottom of $C_1$ and vice versa.  Furthermore, the circumference of $C_i$ equals the circumference of $C_j$, for all $i,j$.
\end{corollary}

\begin{proof}
Without loss of generality, assume that the vertical foliation of $(X,\omega)$ is periodic.  Consider a divergent sequence of times $\{t_n\}$ such that the sequence $G_{t_n} \cdot (X, \omega)$ converges to a degenerate surface $(X', \omega')$.  By \cite{MasurThesis}[Theorem 3], the limit of this sequence is given by pinching the core curves of every cylinder in the cylinder decomposition of $(X,\omega)$.  Furthermore, $\omega'$ has a pair of simple poles at all of the pairs of punctures of $X'$.  Hence, $G(X', \omega') = G^P(X', \omega')$.  Since $G(X', \omega')$ is a connected graph, $G(X', \omega')$ must be a cycle by Lemma \ref{GPisCycle}.  This implies that the cylinders must be arranged in exactly the configuration described in the statement of the corollary.  Clearly this argument does not depend on $\theta$, so the result follows.
\end{proof}

\begin{lemma}
\label{RankOneImpNotHoloEverywhere}
Let $(X,\omega)$ generate a Teichm\"uller disc $D \subset$ \RankOne .  If $(X',\omega')$ is a degenerate surface in the closure of $D$ and $\omega'$ is not holomorphic, then $\omega'$ has simple poles on every part of $X'$.
\end{lemma}

\begin{proof}
Let $\{(X_n, \omega_n)\}_{n=0}^{\infty}$ be a sequence in $D$ converging to the degenerate surface $(X', \omega')$ as $n$ tends to infinity.  Since $\omega'$ is not holomorphic, there is a sequence of cylinders $\{C_1^{(n)}\}_{n=0}^{\infty}$, such that $C_1^{(n)} \subset X_n$ and the core curve of $C_1^{(n)}$ pinches to form a pair of simple poles of $\omega'$.  By Theorem \ref{RankOneImpCP}, the foliation in which $(X_n,\omega_n)$ admits the cylinder $C_1^{(n)}$ is periodic.  Therefore, there is a collection of cylinders $\{C_1^{(n)}, \ldots, C_k^{(n)}\}$ that fill $X_n$.  Let $w_i^{(n)}$ denote the circumference of $C_i^{(n)}$.  By Corollary \ref{RankOneCylConfig}, the ratios $w_i^{(n)}/w_1^{(n)} = 1$ for all $i \leq k$ and $n \geq 0$.  Hence, if the core curve of $C_1^{(n)}$ pinches, then the core curve of every cylinder in that foliation pinches.  Since the ratios between the circumferences are constant, every sequence of cylinders converges to an infinite cylinder on $X'$, and $\omega'$ cannot be holomorphic on any part of $X'$.
\end{proof}

\begin{corollary}
\label{RankOneImpNonZero}
Let $(X,\omega)$ generate a Teichm\"uller disc $D \subset$ \RankOne .  If $(X',\omega')$ is a degenerate surface in the closure of $D$ and $\omega' \not \equiv 0$ on a part $S \subset X'$, then $\omega'$ is not identically zero on any part of $X'$.
\end{corollary}

\begin{proof}
By contradiction, assume that there is a part $S$ of $(X', \omega')$ such that $\omega' \equiv 0$ on $S$.  By Lemma \ref{RankOneImpNotHoloEverywhere}, $\omega'$ must be holomorphic on every part of $X'$.  By Lemma \ref{NonHoloPart} and \cite{BainbridgeMoller}[Proposition 11.1], we can find a sequence of surfaces $\{(X_n, \omega_n)\}_{n=0}^{\infty}$ in the closure of $D$ converging to the degenerate surface $(X'', \omega'')$ as $n$ tends to infinity, such that $\omega''$ is not holomorphic on a part of $X''$.  Since the zero differential is obviously fixed by the \splin ~action, this contradicts Lemma \ref{RankOneImpNotHoloEverywhere} and completes the proof.
\end{proof}

\begin{definition}
An edge $e$ of a connectivity graph $G(X')$ is called a \emph{holomorphic edge} with respect to $\omega'$ if $\omega'$ is holomorphic at the pair of punctures corresponding to $e$.
\end{definition}

\begin{lemma}
\label{NoHoloEdges}
Let $(X,\omega)$ generate a Teichm\"uller disc $D \in$ \RankOne ~and let $(X',\omega')$ be a degenerate surface in the closure of $D$.  If $e$ is an edge in the connectivity graph $G(X')$ between two distinct vertices, then $e$ is not a holomorphic edge with respect to $\omega'$.
\end{lemma}

\begin{proof}
By contradiction, assume there is a holomorphic edge $e$ between two distinct vertices.  First, we claim that $\omega'$ cannot be holomorphic on a surface with two or more parts.  By Lemma \ref{NonHoloPart}, we can act by the \splin ~action on $(X', \omega')$ to reach a surface $(X'', \omega'')$ such that $\omega''$ has a pair of simple poles.  By Lemma \ref{RankOneImpNotHoloEverywhere}, $\omega''$ must have simple poles on every part of $X''$.  However, for every pair of punctures $(p,p')$ on $X'$ where $\omega'$ is holomorphic, $\omega''$ must also be holomorphic at the corresponding pair of punctures on $X''$.  This forces $G^P(X'', \omega'')$ to be a disconnected graph where every vertex in $G^P$ has degree at least two.  This contradicts Lemma \ref{GPisCycle}, so we assume that $\omega'$ is not holomorphic on every part of $X'$.

If $\omega'$ is not holomorphic, then Lemmas \ref{GPisCycle} and \ref{RankOneImpNotHoloEverywhere} imply that $e$ is an edge between two vertices of the cycle $G^P(X', \omega')$.  Let $C_1$ be a cylinder corresponding to an edge of $G^P(X', \omega')$.  Let $(X_1,\omega_1)$ be a surface whose vertical foliation contains the core curve of $C_1$.  The vertical foliation of $(X_1, \omega_1)$ is periodic by Theorem \ref{RankOneImpCP}, and \cite{MasurThesis}[Theorem 3] implies that the core curves of all of the cylinders parallel to $C_1$ pinch under $G_t$.  Let $(X'', \omega'')$ be the resulting degenerate surface.  Note that $\omega''$ has simple poles at every pair of punctures on $X''$.  Moreover, since we pinched the core curve of every cylinder parallel to $C_1$, $\omega''$ must have poles at all of the same punctures at which $\omega'$ has poles on $X'$.  However, the edge $e$ is no longer in the graph $G(X'', \omega'')$, which implies that the two vertices it joined are a single vertex in $G(X'', \omega'')$.  This is impossible because it would imply that $\text{dim}_C(G^P) \geq 2$.  Therefore, $G(X')$ has no holomorphic edges with respect to $\omega'$.
\end{proof}

\begin{lemma}
\label{TDDerPerRank1Lem}
If $(X',\omega')$ is a degenerate surface in the closure of a Teichm\"uller disc $D \subset \mathcal{D}_g (1)$, then $(X',\omega')$ has one of the following three configurations:
\begin{itemize}
\item[(1)] $(X',\omega')$ has exactly one part with \emph{at most} two simple poles.
\item[(2)] $(X',\omega')$ has exactly two parts that are joined by exactly two pairs of poles.
\item[(3)] $X' = S_1 \sqcup \cdots \sqcup S_n$ has $n \geq 3$ parts such that $\omega'$ has exactly one pair of poles joining $S_j$ to $S_{j+1}$, for $1 \leq j \leq n-1$, and exactly one pair of poles joining $S_n$ to $S_1$.
\end{itemize}
Furthermore, there are no pairs of punctures joining two distinct parts in the second and third configuration above such that $\omega'$ is holomorphic at those pairs of punctures.
\end{lemma}

\begin{proof}
By Lemma \ref{PairPolesMakeRank}, if $X'$ has one part, then $\omega'$ has at most one pair of poles.  If $X'$ has more than one part, then this lemma follows from Lemmas \ref{GPisCycle}, \ref{RankOneImpNotHoloEverywhere}, Corollary \ref{RankOneImpNonZero}, and Lemma \ref{NoHoloEdges}.
\end{proof}

\begin{remark}
Case (2) describes a cycle on two vertices that is simply a degenerate version of Case (3).  We distinguished it from Case (3) for clarity.
\end{remark}

\section{Applications of Complete Periodicity in \RankOne}
\label{TDThmAppsSec}

The property of complete periodicity imposes very strong restrictions on a surface.  With little effort we prove that there are no Teichm\"uller discs in \RankOne ~in certain strata of Abelian differentials and apply this to genus two.

\begin{lemma}
\label{MoreThan1Cyl}
Given a completely periodic surface $(X,\omega) \in \mathcal{M}_g$, $g \geq 2$, there exists $\theta \in \mathbb{R}$ such that the cylinder decomposition of $(X, e^{i\theta}\omega)$ has at least two cylinders.
\end{lemma}

\begin{proof}
Assume that $(X,\omega)$ is filled by a single cylinder $C$.  We show that there exists a direction such that $(X,\omega)$ is not filled by a single cylinder.  The top and bottom of $C$ consist of a union of saddle connections.  Choose one such saddle connection $\sigma$ on the bottom of $C$ joining zeros $z_1$ to $z_2$, which are not necessarily distinct.  Let $\sigma'$ be the saddle connection on the top of $C$ to which $\sigma$ is identified.  Let $\sigma'$ have endpoints $z_1'$ and $z_2'$ such that $z_i$ is identified to $z_i'$, for $i = 1,2$.  Consider the family of trajectories in $C$ parallel to a trajectory from $z_1$ to $z_1'$.  This determines a cylinder $C' \subset X$ with $z_1$ on its top and $z_2$ on its bottom formed by identifying $\sigma$ to $\sigma'$.  Since $\sigma$ is a proper subset of the top of cylinder $C$, the cylinder $C'$ does not fill $(X,\omega)$.  Furthermore, $(X,\omega)$ is completely periodic, so the complement of $C'$ must contain at least one cylinder.
\end{proof}

\begin{proposition}
\label{H2gm2DerPerRank1}
There are no Teichm\"uller discs contained in $\mathcal{D}_g (1) \cap \mathcal{H}(2g-2)$.
\end{proposition}

\begin{proof}
By contradiction, assume that there is a surface $(X,\omega)$ generating a Teichm\"uller disc in $\mathcal{D}_g (1) \cap \mathcal{H}(2g-2)$.  By Lemma \ref{MoreThan1Cyl}, choose a direction $\theta$ such that $(X,e^{i\theta}\omega)$ decomposes into two or more cylinders.  Under the Teichm\"uller geodesic flow, $(X,e^{i\theta}\omega)$ degenerates to a surface $(X', \omega')$ with two or more parts by Lemma \ref{TDDerPerRank1Lem} and \cite{MasurThesis}[Theorem 3].  Moreover, the zero of order $2g-2$ must lie on exactly one of the parts because \cite{MasurThesis}[Theorem 3] implies that only the core curves of cylinders are pinched.  This implies that there is a part of $X'$ with two simple poles and no zeros, i.e. a twice punctured sphere.  This is not admissible under the Deligne-Mumford compactification, thus we get a contradiction.
\end{proof}

\begin{proposition}
\label{HnmDerPerRank1}
Let $n$ and $m$ be odd numbers such that $n+m = 2g-2$.  There are no Teichm\"uller discs contained in $\mathcal{D}_g (1) \cap \mathcal{H}(n,m)$.
\end{proposition}

\begin{proof}
By contradiction, assume that there is a surface $(X,\omega)$ generating a Teichm\"uller disc in $\mathcal{D}_g (1) \cap \mathcal{H}(n,m)$.  By Lemma \ref{MoreThan1Cyl}, choose a direction $\theta$ such that $(X,e^{i\theta}\omega)$ decomposes into two or more cylinders.  Under the Teichm\"uller geodesic flow, $(X,e^{i\theta}\omega)$ degenerates to a surface $(X', \omega')$ with two or more parts by Lemma \ref{TDDerPerRank1Lem} and \cite{MasurThesis}[Theorem 3].  Moreover, the zeros must lie on one or two of the parts of $X'$ because \cite{MasurThesis}[Theorem 3] implies that only the core curves of cylinders were pinched.  If they lie on the same part, then as before, every other part must be a twice punctured sphere, which is impossible.  However, if they lie on different parts, then there is a part with two simple poles and a zero of order $n$.  Since there does not exist an integer $g' \geq 0$ such that $n-2 = 2g'-2$, the Chern formula cannot be satisfied and we have a contradiction.
\end{proof}

Though Proposition \ref{Genus2Cor} is well-known, we provide an original proof that there are no Teichm\"uller discs contained in $\mathcal{D}_2(1)$.  The best possible result for the Lyapunov exponents of genus two surfaces was proven by Bainbridge \cite{BainbridgeThesis}, who used McMullen's \cite{McMullenGenus2} classification of \splin -invariant ergodic measures in genus two to calculate the Lyapunov exponents of the Kontsevich-Zorich cocycle explicitly.  Bainbridge found $\lambda_2 = 1/2$, for all \splin -invariant ergodic measures with support in $\mathcal{H}(1,1)$, and $\lambda_2 = 1/3$, for all \splin -invariant ergodic measures with support in $\mathcal{H}(2)$.

\begin{proposition}
\label{Genus2Cor}
There are no Teichm\"uller discs contained in $\mathcal{D}_2(1)$.
\end{proposition}

\begin{proof}
This follows from Propositions \ref{H2gm2DerPerRank1} and \ref{HnmDerPerRank1} because $\mathcal{M}_2 = \mathcal{H}(2) \cup \mathcal{H}(1,1)$.
\end{proof}

Note that $\mathcal{D}_2(1)$ is the determinant locus in genus two.  We remark that the author has another proof of Proposition \ref{Genus2Cor} using more direct methods than those in this paper and more elementary than those of \cite{BainbridgeThesis}.

\section{Convergence to Veech Surfaces}

The goal of this section is to prove Theorem \ref{CPImpVeech}, which will serve as the first step toward bridging the gap between the problem of classifying all Teichm\"uller discs in \RankOne ~and M\"oller's \cite{MollerShimuraTeich} nearly complete classification of Teichm\"uller \emph{curves} in \RankOne .

\begin{lemma}
\label{TDAlsoInDg1}
Given a surface $(X,\omega)$ generating a Teichm\"uller disc $D_1 \subset \mathcal{D}_g (1)$, let $\{(X_n,\omega_n)\}_{n=1}^{\infty}$ be a sequence of surfaces in $D_1$ converging to $(X',\omega') \in \overline{\mathcal{M}_g}$, where $(X',\omega') \not \in D_1$ and $\omega'$ is holomorphic.  If $D_2$ is the Teichm\"uller disc generated by $(X',\omega')$, then $D_2 \subset \overline{\mathcal{D}_g (1)}$.  Furthermore, $D_2 \subset \overline{D_1}$.
\end{lemma}

\begin{proof}
We recall that the \splin ~action on $\overline{\mathcal{M}_g}$ is continuous by \cite{BainbridgeMoller}[Proposition 11.1].  Since $\overline{\mathcal{D}_g (1)}$ is closed, the closure of $D_1$ in $\overline{\mathcal{M}_g }$ is also contained in $\overline{\mathcal{D}_g (1)}$.  Furthermore, every point in $D_2$ is the limit of a sequence of points in $D_1$.  This can be seen by taking a sufficiently small neighborhood of $(X',\omega')$, which contains points in $D_1$ by assumption.  By the continuity of the \splin ~action on $\overline{\mathcal{M}_g}$, there is an arbitrarily small neighborhood of any point in $D_2$ that also contains points in $D_1$.  Hence, $D_2 \subset \overline{D_1} \subset \overline{\mathcal{D}_g (1)}$.
\end{proof}

\begin{definition}
A surface $(X,\omega)$ is called a \emph{Veech surface} if its group $\text{SL}(X,\omega)$ of affine diffeomorphisms is a lattice in \splin .  The Teichm\"uller disc generated by a Veech surface in the moduli space $\mathcal{M}_g$ is called a \emph{Teichm\"uller curve}.
\end{definition}

The reason for the term Teichm\"uller curve follows from a result of Smillie, which states that the \splin ~orbit of a Veech surface projected into $\mathcal{R}_g$ is closed.  This result was never published by John Smillie.  However, it was communicated to William Veech, who outlined a proof of it in \cite{VeechGeomRealHypCurv}.  Moreover, when projected into $\mathcal{R}_g$, Teichm\"uller curves are algebraic curves.  One striking property of Veech surfaces is the \emph{Veech dichotomy}.  The Veech dichotomy completely describes the dynamics of the trajectory of any point on the surface $X$ \cite{VeechTeichCurvEisen}.  It says that the geodesic flow on $X$ with respect to the flat structure induced by $\omega$ is either periodic or uniquely ergodic.  The following definition was introduced in \cite{CheungHubertMasurTopDich}.

\begin{definition}
A completely periodic surface satisfies \emph{topological dichotomy} if any direction that admits a saddle connection is periodic.
\end{definition}

\begin{lemma}
\label{CPZeroConv}
Given a Teichm\"uller disc $D \subset$ \RankOne ~of a completely periodic surface $(X_0,\omega_0) \in \mathcal{M}_g$, which does not satisfy topological dichotomy, there exists a sequence of surfaces $\{(X_n, \omega_n)\}_{n=0}^{\infty}$ in $D$ converging to a surface $(X',\omega') \in \overline{\mathcal{D}_g(1)}$ such that $X'$ has one part, $\omega'$ is holomorphic, and a saddle connection of $(X_0,\omega_0)$ contracts to a point on $(X', \omega')$.
\end{lemma}

\begin{proof}
By assumption, there exists a saddle connection $\sigma_0$ lying in a nonperiodic foliation of the surface $(X_0,\omega_0)$.  Without loss of generality, let $\sigma_0$ lie in the vertical foliation of $(X_0, \omega_0)$.  Act by the Teichm\"uller geodesic flow $G_t$ on $(X_0,\omega_0)$ so that $\sigma_0$ contracts by $e^{-t}$ as $t$ tends to infinity.  We prove that we can choose a divergent sequences of times $\{t_n\}_{n=0}^{\infty}$ such that the corresponding sequence of surfaces $\{(X_n, \omega_n)\}_{n=0}^{\infty}$, defined by
$$(X_n, \omega_n) = G_{t_n} \cdot (X_0, \omega_0),$$
converges to a degenerate surface $(X', \omega')$, where $\omega'$ is holomorphic.  Let $t_0 = 0$.

Let $\sigma_t$ be the saddle connection on $G_t \cdot (X_0, \omega_0)$ defined by contracting the saddle connection $\sigma$ by $e^{-t}$.  If $\omega'$ is not holomorphic, then, by Lemma \ref{RankOneCylConfig}, for all $\varepsilon > 0$, there exists an $N$ and $\theta_N$, such that the vertical foliation of $(X_N, e^{i\theta_N}\omega_N)$ determines a decomposition of $(X_N, e^{i\theta_N}\omega_N)$ into a union of cylinders $C_1, \ldots, C_p$, with waist lengths $\varepsilon$ and heights $h_1, \ldots, h_p$, respectively, such that $\sum_k h_k = 1/\varepsilon$.  This follows from the assumption that the area of every surface in the sequence is one.  This sequence of surfaces defines a sequence of closed curves $\{\gamma_{n,t_n}\}_{n=0}^{\infty}$ whose lengths tend to zero as $n$ tends to infinity, where $\gamma_{n,t_n}$ is the waist curve of a cylinder on $(X_n, e^{i\theta_n}\omega_n)$.  Furthermore, for each $n$, the curve $\gamma_{n,t_n}$ corresponds to a closed curve $\gamma_{n,t_0}$ on $(X_0,\omega_0)$ with the property that the image of $\gamma_{n,t_0}$ under $G_{t_n}$ is $\gamma_{n,t_n}$.  Note that for all $n$ and $t_n$, no curve $\gamma_{n,t_n}$ is parallel to $\sigma_{t_n}$ because $\sigma_{t_n}$ does not lie in a periodic foliation while $\gamma_{n,t_n}$ always lies in a periodic foliation.

We claim that we can pass to a subsequence such that $\gamma_{n,t_n}$ is transverse to $\gamma_{n+1,t_n}$.  Let $0 \leq \alpha_{n,t} < \pi$ denote the angle between $\gamma_{n,t}$ and $\sigma_{n,t}$.  For all $n$ and $t_n$, $\alpha_{n,t_n} \not= 0$ because $\gamma_{n,t_n}$ is not parallel to $\sigma_{t_n}$.  Fixing $n$ and letting $t$ tend to infinity, $|\alpha_{n,t}|$ tends to $\pi/2$ because $\gamma_{n,t_n}$ has nontrivial length in the maximally expanding direction of $G_t$, so for sufficiently large $t$, $\gamma_{n,t}$ converges to the direction of maximum expansion, which is orthogonal to the direction of minimal expansion in which $\sigma_0$ lies.  We prove that the set $\Gamma = \{\gamma_{n,0}|n \geq 0\}$ is infinite.  If not, the previous comment would imply that given $\delta > 0$, there exists a time $T>0$, such that for all $n$ and $t > T$, 
$$\sup_{n}\left||\alpha_{n,t}| - \pi/2\right| < \delta.$$
This would contradict the fact that the lengths of the curves $\{\gamma_{n,t}\}_{n=0}^{\infty}$ tend to zero.  Hence, the set $\Gamma$ is infinite and we can pass to a subsequence such that $\gamma_{n,t_n}$ is transverse to $\gamma_{n+1,t_n}$.  Equivalently, $\gamma_{n,t_{n+1}}$ is transverse to $\gamma_{n+1,t_{n+1}}$.

Now we can construct a sequence of surfaces corresponding to a divergent sequence of times $\{t_n'\}_{n=0}^{\infty}$ such that the limit is holomorphic and the saddle connection $\sigma_n$ degenerates to a point.  Let $\varepsilon_N > 0$ be the infimum, taken over all cylinder decompositions of $(X_N, \omega_N)$, of the length of the waist curves of the cylinders at time $t_N$.  By passing to a subsequence of times, we can assume $\gamma_{N+1,t_{N+1}}$ has length $\varepsilon_{N+1} < \varepsilon_N$.  However, $\gamma_{N,t_N}$ has length $\varepsilon_N$ and $\gamma_{N+1,t_N}$ is transverse to $\gamma_{N,t_N}$.  For any surface $(X,\omega)$, whose Teichm\"uller disc is contained in \RankOne , let $\gamma$ be the waist curve of a cylinder $C_j$ which is an element of a cylinder decomposition $\mathcal{C}$ of $(X,\omega)$.  It follows from Lemma \ref{RankOneCylConfig} that every closed regular trajectory transverse to $\gamma$ must pass through every cylinder in $\mathcal{C}$ at least once.  Thus, in this case, $\gamma_{N+1,t_N}$ has length at least $1/\varepsilon_N$.  Since $\gamma_{N+1,t_{N+1}}$ has length $\varepsilon_{N+1} < \varepsilon_N$ and $\gamma_{N+1,t_N}$ can be pinched under the Teichm\"uller geodesic flow so that the direction of $\sigma_N$ contracts, then there is a time $t_{N+1}'$ such that $t_N < t_{N+1}' < t_{N+1}$ and $\gamma_{N+1,t_{N+1}'}$ has length one.  Furthermore, if $\gamma_{N+1,t_{N+1}'}$ has length one, then by the assumption that the area of $(X_0, \omega_0)$ is one, the fact that $G_t$ preserves area, and the Teichm\"uller disc of $(X_0,\omega_0)$ is contained in \RankOne , we have that the minimum length of any curve transverse to $\gamma_{N+1,t_{N+1}'}$ is also one.  This implies that there are no short closed curves which are not unions of saddle connections.  This defines a divergent sequence of times $\{t_n'\}_{n=0}^{\infty}$ such that the corresponding sequence of surfaces $\{(X_n, \omega_n)\}_{n=0}^{\infty}$ converges to a degenerate surface $(X',\omega')$, where $\omega'$ is holomorphic and $\sigma_0$ contracts to a point on $X'$.  Finally, by Lemma \ref{TDDerPerRank1Lem}, the only admissible boundary points of a Teichm\"uller disc contained in \RankOne , which carry holomorphic Abelian differentials, must have exactly one part.
\end{proof}

The following definition was introduced by Vorobets \cite{VorobetsPlaneStrctsBilliards}.  In \cite{SmillieWeissCharLattice}[Theorem 1.3, Parts (i) and (ii)], Smillie and Weiss prove that a surface is uniformly completely periodic if and only if it is a Veech surface.

\begin{definition}
\label{UnifCP}
Let $\mathcal{S}_{\theta}$ denote the set of saddle connections of the vertical foliation of $(X,e^{i\theta}\omega)$.  A surface is called \emph{uniformly completely periodic} if it satisfies topological dichotomy and there exists a real number $s > 0$ such that for all $\theta$, where $\mathcal{S}_{\theta} \not= \emptyset$, the ratio of the length of the longest saddle connection in $\mathcal{S}_{\theta}$ to the shortest saddle connection in $\mathcal{S}_{\theta}$ is bounded by $s$.
\end{definition}

\begin{lemma}
\label{CPPIZeroConv}
Given a Teichm\"uller disc $D \subset$ \RankOne ~of a surface satisfying topological dichotomy $(X_0,\omega_0) \in \mathcal{M}_g$ that is not uniformly completely periodic, there exists a sequence of surfaces $\{(X_n, \omega_n)\}_{n=0}^{\infty}$ in $D$ converging to a surface $(X',\omega') \in \overline{\mathcal{D}_g(1)}$ such that $X'$ has one part, $\omega'$ is holomorphic, and a saddle connection of $(X_0,\omega_0)$ contracts to a point on $(X', \omega')$.
\end{lemma}

\begin{proof}
Since the surface $(X_0, \omega_0)$ is not uniformly completely periodic, given a divergent sequence of positive real numbers $\{s_j\}_{j=0}^{\infty}$, there exists a corresponding sequence of angles $\{\theta_j\}_{j=0}^{\infty}$ such that the ratio of the longest saddle connection to the shortest saddle connection on $(X_0, e^{i\theta_j}\omega_0)$ is greater than $s_j$, for all $j$.  We show that there exists a sequence of times $\{t_n\}_{n=0}^{\infty}$ such that the sequence of surfaces $\{G_{t_n} \cdot (X_0, e^{i\theta_n}\omega_0)\}_{n=0}^{\infty}$ converges to a surface $(X',\omega')$, where $\omega'$ is holomorphic.  Moreover, there is a sequence of saddle connections on $G_{t_n} \cdot (X_0, e^{i\theta_n}\omega_0)$ converging to a point as $n$ tends to infinity.

Pass to a subsequence of $\{\theta_j\}_{j=0}^{\infty}$ defined as follows.  Since there is a finite number of zeros, there is a finite number of pairs of zeros.  Choose a pair of zeros $z_1$ and $z_2$ that occur infinitely often in the sequence $\{(X_0, e^{i\theta_n}\omega_0)\}_{n=0}^{\infty}$ as the pairs of zeros which are joined by the shortest saddle connection.  Pass to a subsequence $\{(X_0, e^{i\theta_n}\omega_0)\}_{n=0}^{\infty}$ such that a saddle connection between $z_1$ and $z_2$ represents the shortest saddle connection on $(X_1, e^{i\theta_n}\omega_1)$.  By Lemma \ref{RankOneCylConfig}, all of the cylinders in the cylinder decomposition of a surface in \RankOne ~have equal circumference and we can assume that $(X_1,\omega_1)$ has unit area and cylinders of unit circumference.  For each angle $\theta_j$, denote by $w_j > 1$ the length of the circumference of the cylinders in that direction.  Then define the times $t_j$ by $e^{-t_j}w_j = 1$, for all $j$.  Then $G_{t_n} \cdot (X_0,e^{i\theta_n} \omega_0) = (X_n, \omega_n)$ is the action on the surface such that the waist curves of the cylinders of circumference $w_j$ contract at the maximal rate.  Furthermore, since the length of each saddle connection is bounded above by the circumference of the cylinders, the length of the shortest saddle connection on $(X_n, \omega_n)$ is bounded above by $1/s_n$.  Note that $\lim_{n \rightarrow \infty} 1/s_n = 0$.  The Teichm\"uller geodesic flow preserves area, so the surface $(X_n, \omega_n)$ also has unit area for all $n$.  This implies that the sum of the heights of the cylinders is equal to one, as well.  It follows from Lemma \ref{RankOneCylConfig} that any curve transverse to a horizontal curve has length at least one because any such curve must travel the heights of every cylinder in the cylinder decomposition.  Since the minimum length of a curve transverse to a vertical curve is the waist curve of a cylinder which has length one, there are no closed curves that can pinch that are not unions of saddle connections.

If a closed curve, which is a union of saddle connections, degenerates as $n$ tends to infinity, then the limit is a degenerate surface carrying a holomorphic Abelian differential.  By Lemma \ref{TDDerPerRank1Lem}, the only such degenerate surfaces in the boundary of \RankOne ~have one part.
\end{proof}

\begin{theorem}
\label{CPImpVeech}
If the Teichm\"uller disc $D$ of $(X,\omega)$ is contained in \RankOne , then there is a Veech surface $(X',\omega') \in \overline{\mathcal{M}_g}$ such that the Teichm\"uller disc $D'$ generated by $(X',\omega')$ is contained in $\overline{\mathcal{D}_g(1)}$.  Furthermore, every surface in $D'$ is the limit of a sequence of surfaces in $D$.
\end{theorem}

\begin{proof}
If $(X,\omega)$ is a Veech surface, let $(X,\omega) = (X',\omega')$.  Otherwise, assume that $(X,\omega) = (X_{0,1},\omega_{0,1})$ is not a Veech surface and let $D_1$ be its Teichm\"uller disc.  Since $(X,\omega)$ is not a Veech surface, but its Teichm\"uller disc is contained in $\overline{\mathcal{D}_g(1)}$, $(X,\omega)$ is completely periodic by Theorem \ref{RankOneImpCP}.  Furthermore, $(X,\omega)$ is not uniformly completely periodic by \cite{SmillieWeissCharLattice}[Theorem 1.3, Parts (i) and (ii)].  By Lemmas \ref{CPZeroConv} and \ref{CPPIZeroConv}, there exists a sequence $\{(X_{n,1},\omega_{n,1})\}_{n=1}^{\infty}$ converging to a surface $(X_{0,2},\omega_{0,2}) \in \overline{\mathcal{M}_g}$ with one part carrying a holomorphic Abelian differential with $(X_{0,2},\omega_{0,2}) \in$ \RankOne ~and a saddle connection on $\omega_{0,1}$ degenerates to a point on $X_{0,2}$.  A degenerate saddle connection implies either two or more zeros of $\omega_{0,1}$ converge to a single zero of $\omega_{0,2}$ or a closed curve of $X_{0,1}$ converges to a pair of punctures on $X_{0,2}$.  Then $(X_{0,2},\omega_{0,2})$ has Teichm\"uller disc $D_2$ and by Lemma \ref{TDAlsoInDg1}, $D_2 \subset \overline{\mathcal{D}_g(1)}$ .  By Theorem \ref{RankOneImpCP}, $(X_{0,2},\omega_{0,2})$ is also completely periodic.  If it is a Veech surface, then we are done.  Otherwise, we proceed by induction using Lemmas \ref{CPZeroConv} and \ref{CPPIZeroConv} to create a sequence of surfaces $\{(X_{0,j},\omega_{0,j})\}_{j=1}^{N}$ in $\overline{\mathcal{M}_g}$ such that each surface in the sequence carries a differential either with fewer distinct zeros or lower genus than the previous surface in the sequence.  Since both the number of zeros as well as the genus are finite, this process will terminate at some step $N$ resulting in a surface $(X_{0,N},\omega_{0,N}) \in \overline{\mathcal{D}_g(1)}$ with Teichm\"uller disc $D_N$.  By Lemma \ref{TDAlsoInDg1}, $D_N \subset \overline{\mathcal{D}_g(1)}$.  The surface $X_{0,N}$ cannot be a sphere because $\omega_{0,N}$ is holomorphic and $\omega_{0,N}$ is nonzero by Lemma \ref{NonZeroPart}.  Hence, there are three possibilities.  Either $\omega_{0,N}$ has a single zero, $X_{0,N}$ is a punctured torus, or $(X_{0,N},\omega_{0,N})$ is a Veech surface.  By Lemma \ref{H2gm2DerPerRank1}, $\omega_{0,N}$ cannot have a single zero and Lemma \ref{NoDisksInD11} says $X_{0,N}$ cannot be a punctured torus.  Thus, the only remaining possibility is that $(X_{0,N},\omega_{0,N})$ is a Veech surface.

Let $D'$ be the Teichm\"uller disc generated by $(X_{0,N}, \omega_{0,N})$.  Lemma \ref{TDAlsoInDg1} implies that every surface in $D'$ is the limit of a sequence of surfaces in $D_1$.
\end{proof}

\section{Punctured Veech Surfaces}
\label{PctdVeechSurfSect}

There are several key results that give a nearly complete picture of Teichm\"uller curves in \RankOne .  We recall all of the results here for the sake of completeness and convenience of the reader.  There are two similarly named, related concepts: a square-tiled covering and a square-tiled cyclic cover.  A \emph{square-tiled covering} is a specific type of Veech surface introduced by Thurston formed by gluing unit squares together to form a genus $g$ surface.  Naturally, such a surface comes with a covering of the unit square, i.e. the torus.  A surface is a square-tiled covering if and only if it has affine group commensurable to $\text{SL}_2(\mathbb{Z})$, by \cite{GutkinJudge}[Theorem 5.9].

We define a \emph{square-tiled cyclic cover} using the exposition of \cite{ForniMatheusZorichSqTiled}.  A square-tiled cyclic cover is a specific type of square-tiled covering.  Let $N > 1$ be an integer and $(a_1, a_2, a_3, a_4) \in \mathbb{Z}^4$ such that they satisfy
$$0 < a_i \leq N; \text{  } \gcd(N, a_1, \ldots, a_4) = 1; \text{  } \sum_{i=1}^4 a_i \equiv 0 (\mod N).$$
Then the algebraic equation
$$w^N = (z-z_1)^{a_1}(z-z_2)^{a_2}(z-z_3)^{a_3}(z-z_4)^{a_4}$$
defines a closed, connected and nonsingular Riemann surface denoted by \newline $M_N(a_1, a_2, a_3, a_4)$.  By construction, $M_N(a_1, a_2, a_3, a_4)$ is a ramified cover over the Riemann sphere $\mathbb{P}^1(\mathbb{C})$ branched over the points $z_1, \ldots, z_4$.  Consider the meromorphic quadratic differential
$$q_0 = \frac{dz^2}{(z-z_1)(z-z_2)(z-z_3)(z-z_4)}$$
on $\mathbb{P}^1(\mathbb{C})$.  It has simple poles at $z_1, \ldots, z_4$ and no other zeros or poles.  Then the canonical projection
$$p: M_N(a_1, a_2, a_3, a_4) \rightarrow \mathbb{P}^1(\mathbb{C})$$
induces a quadratic differential $q = p^*q_0$ by pull-back.  Lemma \ref{VeechRankOneImpSqTil} follows from \cite{MollerShimuraTeich}[Cor. 3.3, Sect. 3.6].

\begin{remark}
The name cyclic cover comes from the fact that the group of deck transformations of a cyclic cover is the cyclic group $\mathbb{Z}/N\mathbb{Z}$.
\end{remark}

\begin{lemma}[M\"oller]
\label{VeechRankOneImpSqTil}
If $(X,\omega)$ is a Veech surface whose Teichm\"uller disc is contained in \RankOne , then $(X,\omega)$ is a square-tiled covering.
\end{lemma}

We recall the two known examples of surfaces that generate Teichm\"uller discs in \RankOne .  The genus three example, denoted here by $(M_3, \omega_{M_3})$, is commonly known as the Eierlegende Wollmilchsau for its numerous remarkable properties \cite{HerrlichSchmithusenEier}.  Forni \cite{ForniHand} discovered that its Kontsevich-Zorich spectrum is indeed completely degenerate.  The surface $(M_3, \omega_{M_3})$ is a square-tiled surface given by the algebraic equation
$$w^4 = (z-z_1)(z-z_2)(z-z_3)(z-z_4).$$
Its differential, given in \cite{MatheusYoccoz}, can be written explicitly as
$$\omega_{M_3} = \frac{dz}{w^2}.$$
It is easy to see that this lies in the principal stratum of genus three, $\mathcal{H}(1,1,1,1)$.  The surface is pictured in Figure \ref{Gen3Ex} and the zeros lie at the corners of the squares and are denoted by $v_1, \ldots, v_4$.  For completeness, note that the stratum $\mathcal{H}(1,1,1,1)$ is connected by \cite{KontsevichZorichConnComps}.

\begin{figure}
 \centering
 \includegraphics[width=120mm]{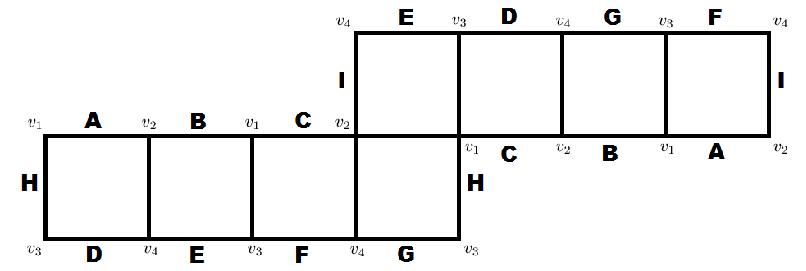}
 \caption{The Eierlegende Wollmilchsau $(M_3, \omega_{M_3})$}
 \label{Gen3Ex}
\end{figure}

\begin{proposition}[Forni]
\label{ForniGen3}
The square-tiled surface $(M_3, \omega_{M_3})$ generates a Teichm\"uller curve in $\mathcal{D}_3(1)$.
\end{proposition}

The genus four example was discovered by Forni and Matheus \cite{ForniMatheus} and we denote it by $(M_4, \omega_{M_4})$.  Recently, Vincent Delecroix and Barak Weiss have proposed to Carlos Matheus that $(M_4, \omega_{M_4})$ be named the Ornithorynque (Platypus in French).  We adopt this terminology here.  The surface $(M_4, \omega_{M_4})$ is a square-tiled surface given by the algebraic equation
$$w^6 = (z-z_1)(z-z_2)(z-z_3)(z-z_4)^3.$$
Its differential, see \cite{MatheusYoccoz}, can be written explicitly as
$$\omega_{M_4} = \frac{z\,dz}{w^2}.$$
It is easy to see that this lies in the stratum $\mathcal{H}(2,2,2)$.  The surface is pictured in Figure \ref{Gen4Ex} and the zeros, denoted by $v_1, v_2, v_3$, lie at the corners of the squares.  For completeness, note that $\mathcal{H}(2,2,2)$ has two connected components by \cite{KontsevichZorichConnComps}, and it was proven in \cite{MatheusYoccoz} and again in \cite{ForniMatheusZorichSqTiled} that $(M_4, \omega_{M_4})$ lies in the connected component $\mathcal{H}^{even}(2,2,2)$ where the spin-structure has even parity.

\begin{proposition}[Forni-Matheus]
\label{FMGen4}
The square-tiled surface $(M_4, \omega_{M_4})$ generates a Teichm\"uller curve in $\mathcal{D}_4(1)$.
\end{proposition}

\begin{figure}
 \centering
 \includegraphics[width=120mm]{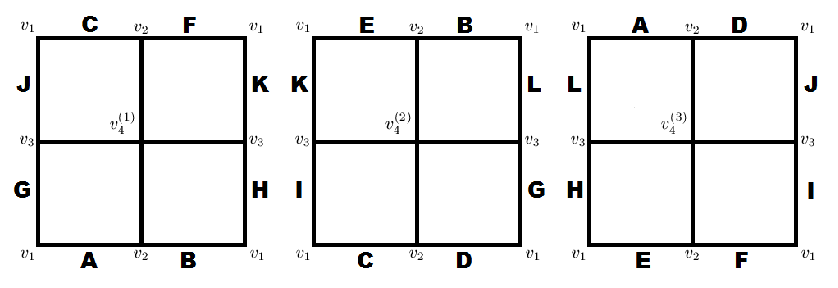}
 \caption{The Ornithorynque $(M_4, \omega_{M_4})$}
 \label{Gen4Ex}
\end{figure}

M\"oller \cite{MollerShimuraTeich} showed that Teichm\"uller curves in \RankOne ~must also be Shimura curves.  This allowed him to give a nearly complete classification of Teichm\"uller curves in \RankOne .

\begin{theorem}[M\"oller]
\label{MollerClass}
Other than possible examples in certain strata of $\mathcal{M}_5$, listed in the table in \cite{MollerShimuraTeich}[Corollary 5.15], and the examples of Propositions \ref{ForniGen3} and \ref{FMGen4}, there are no other Teichm\"uller curves contained in \RankOne , for $g \geq 2$.
\end{theorem}

These results are key to the remainder of the paper.  Theorem \ref{CPImpVeech} implies that for any Teichm\"uller disc in \RankOne ~there is a sequence of surfaces converging to a Veech surface.  This Veech surface may arise from pinching curves to pairs of punctures thereby resulting in a punctured Veech surface.  Moreover, Lemma \ref{VeechRankOneImpSqTil} implies that this punctured Veech surface is, in fact, a punctured square-tiled surface.  The strategy will be to proceed by contradiction and assume that there is such a sequence of surfaces converging to a punctured square-tiled surface.  The theme of the remainder of this paper is captured in the following question.

\begin{question}
Given a sequence of surfaces in a Teichm\"uller disc contained in \RankOne ~converging to a degenerate surface $X'$, which is square-tiled and carries a holomorphic Abelian differential $\omega'$, at which points of $X'$ can the pairs of punctures lie?
\end{question}

\begin{definition}
Let $(X,\omega) \in \mathcal{M}_g$ and $p \in X$.  Let $\Gamma_p(X)$ denote the set of all closed regular trajectories $\gamma$ passing through $p$ with respect to $e^{i\theta}$, for all $\theta \in \mathbb{R}$.  Define the set
$$C_{p}(X) = \bigcap_{\gamma \in \Gamma_p(X)} \gamma.$$
\end{definition}

It should be obvious to the reader that for any compact Riemann surface $X$ and any $p \in X$, $C_{p}(X)$ is a finite set.  Otherwise, it would have an accumulation point on $X$, which is impossible.

\begin{theorem}
\label{PctsAreCCConj}
Let $D$ be a Teichm\"uller disc in \RankOne .  Let $(X', \omega')$ be a degenerate surface in the closure of $D$ such that $\omega'$ is holomorphic and $X'$ has exactly one part.  If $(p, p')$ is a pair of punctures on $(X', \omega')$, then $p' \in C_p(X')$.
\end{theorem}

\begin{proof}
We proceed by contradiction and assume $p' \not \in C_p(X')$.  By definition of $C_p(X')$, there exists a $\theta \in \mathbb{R}$ such that $(X',e^{i\theta}\omega')$ has a closed leaf $\gamma$ passing through $p$ and not through $p'$.  We act on $(X',e^{i\theta}\omega')$ by $G_t$ and claim that we can find a divergent sequence of times $\{t_n\}_{n=1}^{\infty}$ such that $G_{t_n} \cdot (X',e^{i\theta}\omega')$ converges to a degenerate surface which cannot be a boundary point of a Teichm\"uller disc in $\overline{\mathcal{D}_g(1)}$.  Since $(X',\omega')$ is completely periodic, by Theorem \ref{RankOneImpCP}, all of the leaves of the vertical foliation of $(X',e^{i\theta}\omega')$ are closed.  By Corollary \ref{RankOneCylConfig}, all of the leaves have the same length $\ell$.  After time $t$, they have length $e^{-t}\ell$.  Furthermore, since $p$ and $p'$ do not both lie on $\gamma$, the distance between them tends to infinity exponentially with $t$.  Let $(X',\omega')$ degenerate to $(X'',\omega'')$ under the action by $G_t$.  This implies that $(p, p')$ are a pair of holomorphic punctures paired between two distinct parts of $(X'',\omega'')$.  However, Lemma \ref{TDDerPerRank1Lem} says that there cannot be a pair of holomorphic punctures on two distinct parts of a degenerate surface whose Teichm\"uller disc is contained in \RankOne .  This contradiction implies that $p' \in C_p(X')$.
\end{proof}

\begin{lemma}
\label{TorusSepCurve}
Let $\mathbb{T}^2$ denote the torus.  For all $p \in \mathbb{T}^2$, $C_p(\mathbb{T}^2) = \{p\}$.
\end{lemma}

\begin{proof}
Identify $\mathbb{T}^2$ with the unit square $S$.  Consider the horizontal and vertical lines intersecting at $p \in S$.  It is obvious that these two lines have no other intersection point.  Hence, $C_p(\mathbb{T}^2) = \{p\}$.
\end{proof}

\begin{corollary}
\label{PctsLieOverSamePt}
Let $D$ be a Teichm\"uller disc in \RankOne ~such that $(X', \omega')$ is a degenerate surface carrying a holomorphic Abelian differential, and $(X', \omega')$ is a square-tiled surface with covering map $\pi: X \rightarrow \mathbb{T}^2$.  If $(p, p')$ is a pair of punctures on $X'$, then $\pi(p) = \pi(p')$.
\end{corollary}

\begin{proof}
Closed trajectories on $X$ descend to closed trajectories on $\mathbb{T}^2$ under $\pi$.  Hence, it follows from Theorem \ref{PctsAreCCConj} and Lemma \ref{TorusSepCurve} that
$$\pi(p') \in \pi(C_{p}(X)) = C_{\pi(p)}(\mathbb{T}^2) = \{\pi(p)\}.$$
\end{proof}

\begin{remark}
Corollary \ref{PctsLieOverSamePt} is weaker than Theorem \ref{PctsAreCCConj} because $\pi(p) = \pi(p')$ does \emph{not} imply $p' \in C_{p}(X)$.
\end{remark}

\begin{lemma}
\label{NoDisksInD11}
Given a Teichm\"uller disc $D \subset $ \RankOne , for $g \geq 2$, there is no degenerate surface in the closure of $D \subset \overline{\mathcal{D}_g(1)}$ of the form $(S, \omega)$, where $S$ is a punctured torus and $\omega$ is holomorphic.
\end{lemma}

\begin{proof}
By contradiction, assume there is a degenerate surface of the form $(S, \omega)$ in the boundary of $D$.  By the assumption that $S$ arises from pinching curves on a higher genus surface, $S$ has an even, nonzero, number of punctures.  Let $(p,p')$ be a pair of punctures on $S$.  By Lemma \ref{TorusSepCurve}, there are two parallel curves on $S$, $\gamma_1$ and $\gamma_2$ passing through $p$ and $p'$, respectively.  There are two more curves $\gamma_1'$ and $\gamma_2'$ parallel to $\gamma_1$ that do not pass through punctures of $S$ such that $\gamma_1'$ is not homotopic to $\gamma_2'$ (because $S$ is not a torus, but a \emph{punctured} torus).  Pinching the curves $\gamma_1'$ and $\gamma_2'$ degenerates the torus $S$ into a union of two or more spheres $S'$ such that $p$ and $p'$ do not lie on the same sphere.  However, $(p,p')$ represent a pinched curve, so if $S'$ consists of exactly two parts, then there are three pairs of punctures between those two parts and if $S'$ has more than two parts, then there are two parts of $S'$ with two pairs of punctures between them.  This directly contradicts Lemma \ref{TDDerPerRank1Lem} and proves that no surface can degenerate to a punctured torus carrying a holomorphic Abelian differential.
\end{proof}

\begin{theorem}
\label{Genus3Classification}
The Eierlegende Wollmilchsau $(M_3, \omega_{M_3})$ generates the only Teichm\"uller disc in $\mathcal{D}_3(1)$.
\end{theorem}

\begin{proof}
By \cite{MollerShimuraTeich} (restated in Theorem \ref{MollerClass} above), the Eierlegende Wollmilchsau is the only Veech surface that generates a Teichm\"uller disc in $\mathcal{D}_3(1)$.  By contradiction, assume that there is a genus three surface $(X,\omega)$ that generates a Teichm\"uller disc $D \subset \mathcal{D}_3(1)$.  Then $X$ is not a Veech surface, but it is completely periodic by Theorem \ref{RankOneImpCP}.  By Theorem \ref{CPImpVeech}, there is a sequence of surfaces in $D$ converging to a Veech surface $(X',\omega')$ contained in $\overline{\mathcal{D}_3(1)}$.  Since Theorem \ref{CPImpVeech} guarantees that either the Abelian differential $\omega'$ has fewer zeros than $\omega$, which implies $(X',\omega')$ cannot lie in the principal stratum of $\mathcal{M}_3$, or $X'$ has lower genus than $X$.  However, $X'$ cannot have lower genus by Lemma \ref{NoDisksInD11} and Proposition \ref{Genus2Cor} and the fact that the sphere carries no nonzero holomorphic differentials.  Moreover, Theorem \ref{MollerClass} implies that $(X', \omega')$ cannot be a Veech surface because $(X', \omega')$ does lie in the principal stratum.  This contradiction implies that no other Teichm\"uller disc is contained in $\mathcal{D}_3(1)$.
\end{proof}

We complete this section by proving that pairs of punctures cannot lie at the zeros in the known examples of Veech surfaces generating Teichm\"uller discs in \RankOne .

\begin{lemma}
\label{Gen3NoPctsAtZeros}
Let $(M_3',\omega_{M_3}')$ denote $(M_3,\omega_{M_3})$ with punctures.  Let $(X,\omega)$ be a surface generating a Teichm\"uller disc in \RankOne ~such that $(M_3',\omega_{M_3}')$ is a degenerate surface in the closure of $D$.  If $(p,p')$ is a pair of punctures on $(M_3',\omega_{M_3}')$, then
$$\{p,p'\} \cap \{v_1, v_2, v_3, v_4\} = \emptyset.$$
\end{lemma}

\begin{proof}
Let $V = \{v_1, v_2, v_3, v_4\}$ and $\pi: M_3 \rightarrow \mathbb{T}^2$ be a covering map of the torus.  The points in $V$ lie at the corners of each of the squares in the square tile decomposition of $M_3$, so $\pi^{-1}\circ \pi(v_1) = V$.  By contradiction, assume $p \in V$.  By Corollary \ref{PctsLieOverSamePt}, $p \in V$ implies $\{p, p'\} \subset V$.  Hence, there are six pairs of points at which $(p, p')$ can lie.  For each pair of points, listed in the first column of Table \ref{Gen3SlopePfTab}, choose the closed trajectory on $(M_3', \omega_{M_3}')$ in the direction specified in the second column of Table \ref{Gen3SlopePfTab}, according to the orientation of Figure \ref{Gen3Ex}, and pinch the core curves of the cylinders in that direction.  This results in $p$ and $p'$ lying on different parts of a degenerate surface.  This contradicts Lemma \ref{TDDerPerRank1Lem}, which states that there are no holomorphic punctures between parts of a degenerate surface in the boundary of \RankOne .

\begin{table}[ht]
\centering
\begin{tabular}{c|l}
Pair of Points & Direction \\
\hline
$(v_1, v_3), (v_2,v_4), (v_2, v_3), (v_1, v_4)$ & Horizontal \\
$(v_1, v_2), (v_3,v_4)$ & Vertical \\
\end{tabular}
\caption{Pinching Directions for Lemma \ref{Gen3NoPctsAtZeros}}
\label{Gen3SlopePfTab}
\end{table}

\end{proof}

\begin{lemma}
\label{Gen4NoPctsAtZeros}
Let $(M_4',\omega_{M_4}')$ denote $(M_4,\omega_{M_4})$ with punctures.  Let $(X,\omega)$ be a surface generating a Teichm\"uller disc in \RankOne ~such that $(M_4',\omega_{M_4}')$ is a degenerate surface in the closure of $D$.  If $(p,p')$ is a pair of punctures on $(M_4',\omega_{M_4}')$, then
$$\{p,p'\} \cap \{v_1, v_2, v_3\} = \emptyset.$$
\end{lemma}

\begin{proof}
Let $V = \{v_1, v_2, v_3, v_4^{(1)}, v_4^{(2)}, v_4^{(3)}\}$ and $\pi: M_4 \rightarrow \mathbb{T}^2$ be a covering map of the torus.  As above, $\pi^{-1}\circ \pi(v_1) = V$.  By contradiction, assume $p \in V$.  Then Corollary \ref{PctsLieOverSamePt} tells us that $\{p, p'\} \subset V$.  There are $15$ subsets of $V$ of order two.  We exclude the three subsets containing only points of the form $v_4^{(j)}$, for all $j$, because they are not relevant to the statement of this lemma.  We proceed as in the proof of Lemma \ref{Gen3NoPctsAtZeros} using Table \ref{Gen4SlopePf} to specify the direction in which to pinch $M_4'$ given the orientation of Figure \ref{Gen4Ex} to reach a contradiction with Lemma \ref{TDDerPerRank1Lem}.

\begin{table}[ht]
\centering
\begin{tabular}{c|l}
Pair of Points & Direction \\
\hline
$(v_1, v_3), (v_2, v_3), (v_2, v_4^{(j)}), (v_1, v_4^{(j)}), \forall j$ & Horizontal \\
$(v_1, v_2), (v_3, v_4^{(j)}), \forall j$ & Vertical \\
\end{tabular}
\caption{Pinching Directions for Lemma \ref{Gen4NoPctsAtZeros}}
\label{Gen4SlopePf}
\end{table}

\end{proof}

\section{Pairs of Punctures at Regular Points}
\label{PctsAtRegPtsSect}

The purpose of this section is to answer part of the question posed in the previous section.  Given a pair of punctures on the square-tiled surface arising from Theorem \ref{CPImpVeech}, we prove in Theorem \ref{NoPctsAtCCCPts} that both punctures cannot lie at regular points of the differential on the surface.  The consequences of this theorem are Corollary \ref{NoTeichDiskSuffHighGenus} and the Conditional Theorem in Section \ref{RegPtsMainThm}.

We point out to the reader a subtle change in perspective in this section that allows us to prove the result of Theorem \ref{NoPctsAtCCCPts}.  So far we have assumed the standard Deligne-Mumford compactification of the moduli space of Riemann surfaces and discussed limits of surfaces within this context.  In this section we will choose to add a sphere to the degenerate surface so that a sequence of degenerating surfaces results in a degenerate surface and a sphere.  To make this more familiar to those in Teichm\"uller theory, we puncture the surface so that we get a thrice punctured sphere in place of a two punctured sphere, which would not be admissible under the Deligne-Mumford compactification.  However, this puncture is used simply as a tool to appeal to the usual perspective of the Deligne-Mumford compactification found in papers on Teichm\"uller theory.  We will take the differential on the sphere to be the zero differential.

We begin with a few general lemmas that will be applied to the specific case of this section.  The following lemma is similar to the theorems found in \cite{Strebel}[Chapter III, $\S 6$].  However, Lemma \ref{StebelForAnnulus} applies to an annulus, which in general could have non-trivial monodromy.

\begin{lemma}
\label{StebelForAnnulus}
Let $\{(X_n, \omega_n)\}_{n=0}^{\infty}$ be a sequence of surfaces carrying Abelian differentials that converge to $(X_{\infty}, \omega_{\infty})$.  Let $A_n \subset X_n$ be a sequence of annuli converging to a punctured disc $\Delta \subset X_{\infty}$ such that the sequence of Abelian differentials $\{\omega_n\}_n$ are band bounded.  For a fixed $k \geq 0$, assume that there is a local coordinate $z$ on $\Delta$ such that $\omega_{\infty} = z^k\,dz$.  Then there exists a local coordinate $z_n$ for $\omega_n$ on $A_n$ such that $\omega_n = z_n^k\,dz_n$, and $z_n$ converges to $z$ as $n$ tends to infinity.
\end{lemma}

\begin{proof}
As in \cite{Strebel}[Chapter II, $\S 5.1$], there are natural local coordinates about a point $P_n \in A_n$ given by $\Omega_n = \int_{\gamma} \omega_n$, where $\gamma$ is a path from $P_n$ to any point $P \in A_n$.  With respect to the differential $\omega_{\infty}$, we have
$$\Omega_{\infty} = \int_{\gamma} \omega_{\infty} = \int_{\gamma}^P z^k\,dz = \frac{z^{k+1}}{k+1} + C.$$
This implies that for sufficiently large $n$, there is a well-defined $k+1$ root of $\Omega_n$ whose branch cut is given by the well-defined root of $\Omega_{\infty}$.  (This is equivalent to solving the linear, separable ODE for $z_n$:
$$z_n^k\,dz_n = f_n(\zeta)\,d\zeta.)$$

On the other hand, we have $\Omega_n' = \omega_n$ converges to $\Omega_{\infty}' = \omega_{\infty}$ by assumption.  Therefore, $\Omega_n$ converges to $\Omega_{\infty}$ up to addition by a constant.  Choose the local coordinates $z_n$ so that the constant induced on $\Omega_n$ is exactly the sequence of constants that guarantees convergence to $\Omega_{\infty}$.  This simultaneously resolves both claims of the lemma.
\end{proof}

We note that the following lemma will be used in the context of Theorem \ref{CPImpVeech}, which means that we will permit any finite number of pinchings.  Moreover, this lemma will only be applied in the case $k_1 = k_2 = 0$.

\begin{lemma}
\label{SphereBubble}
For $g \geq 2$, let $D \subset \mathcal{M}_g$ be a Teichm\"uller disc.  Let $\{(X_n,\omega_n)\}_{n=0}^{\infty}$ be a sequence of surfaces in $D$ converging to a degenerate surface $(X',\omega')$, where $X'$ has one part and $\omega'$ is holomorphic and not identically zero.  Let $\gamma_n$ be a closed curve on $(X_n,\omega_n)$, which is a union of saddle connections of $\omega_n$.  Let $\gamma_n$ converge to a node as $n$ tends to infinity resulting in a pair of punctures $(p,p')$ on $X'$.  Let $\xi_n \in X_n$ be a zero on $\gamma_n$.  For each $n$, let $(X_n^{(p)}, \omega_n^{(p)}) \in \mathcal{M}_{g,1}$ denote the surface $(X_n,\omega_n)$ punctured at $\xi_n$.  There is a choice of limit such that the sequence $\{(X_n^{(p)}, \omega_n^{(p)})\}_{n=0}^{\infty}$ in $\mathcal{M}_{g,1}$ converges to a degenerate surface $(X'^{(p)}, \omega'^{(p)})$, where $X'^{(p)} = X' \sqcup S'$ and $S'$ is the thrice punctured sphere.  The differential $\omega'^{(p)}$ restricted to $X'$ is equal to $\omega'$, and $\omega'^{(p)}$ restricted to $S'$ is the zero differential.

For $j = 1,2$, there is a differential $\omega_n^{(p)}$ on an annulus $R_w(t_n^{(j)})$ converging to a punctured disc contained in $S'$ such that $\omega_n^{(p)}$ converges to the zero differential on $S'$, and if $w_n^{(j)}$ are local coordinates on the annulus with pinching parameter $t_n^{(j)}$, then
$$\omega_n^{(p)} = -\left(\frac{t_n^{(j)}}{w_n^{(j)}}\right)^{k_j+1}\frac{dw_n^{(j)}}{w_n^{(j)}}.$$
\end{lemma}

\begin{proof}
First we prove that $X'^{(p)} = X' \sqcup S'$.  We decompose the flat region around the pinching curve $\gamma_n$ as in \cite{MinskyHarmMapsLengthEnergyTeichSp}, into expanding annuli $E$ and $G$, and a cylinder $F$ (a flat annulus in the language of \cite{MinskyHarmMapsLengthEnergyTeichSp}).  In the context of this lemma, $F = \emptyset$ because the pinching curve is a union of saddle connections and not a closed regular trajectory.  The curve $\gamma_n$ cannot determine a cylinder, so it must lie at the intersection of two expanding annuli $E_n$ and $G_n$.  The moduli of the cylinders $E_n$ and $G_n$ diverge to infinity by the assumption that $\gamma_n$ pinches as $n$ tends to infinity.  However, since the curve $\gamma_n$ contains a puncture, the regions $E_n$ and $G_n$ must each split into two pinching regions, e.g. $E_n = R_z(t_n^{(1)}) \cup R_w(t_n^{(1)})$.  Let $R_w(t_n^{(1)}) \subset E_n$ and $R_w(t_n^{(2)}) \subset G_n$.  Then $R_w(t_n^{(1)}) \cup R_w(t_n^{(2)})$ converges as $n$ tends to infinity to a sphere or equivalently, a flat pair of pants (cf. the remark at the end of this proof).  The two core pinching curves $\gamma_1$, $\gamma_2$ lying in $E_n$ and $G_n$ are homotopic to $\gamma$ if we ignore the puncture.  Therefore, the surface splits into two parts as $n$ tends to infinity: a surface homeomorphic to $X'$ and a sphere $S'$.  We denote the pair of punctures resulting from pinching $\gamma_i$ by $(p_i, p_i')$, for $i = 1,2$.

Now we prove the claim about $\omega'^{(p)}$.  Let $X^* \subset X' \subset X'^{(p)}$ denote a subsurface of $X'$ determined by removing small discs of radius $R$ around $p_1$ and $p_2$.  On $X^*$, $\omega'^{(p)} = \omega'$, so by analytic continuation in local coordinates about $p_1$ and $p_2$, $\omega'^{(p)} = \omega'$ on all of $X' \subset X'^{(p)}$.

Finally, we give explicit local coordinates for $\omega_n^{(p)}$ on the annuli, denoted $R_w(t_n^{(j)})$, converging to $S'$.  We use the pinching coordinate relation $z_n^{(j)} w_n^{(j)} = t_n^{(j)}$ to induce the differential.  On $X'$, $\omega'^{(p)} = \omega'$.  On $R_z(t_n^{(j)})$, $\omega_n^{(p)}$ can be expressed as $(z_n^{(j)})^{k_j}\,dz_n^{(j)}$, by Lemma \ref{StebelForAnnulus}, where $z_n^{(j)}$ represents local coordinates on the annulus $R_z(t_n^{(j)})$, and $w_n^{(j)}$ is used for local coordinates on $R_w(t_n^{(j)})$.  Identifying $z_n^{(j)}$ with $t_n^{(j)}/w_n^{(j)}$ yields
$$(z_n^{(j)})^{k_j}\,dz_n^{(j)} = \left(\frac{t_n^{(j)}}{w_n^{(j)}}\right)^{k_j}\,d\left(\frac{t_n^{(j)}}{w_n^{(j)}} \right) = \left(\frac{t_n^{(j)}}{w_n^{(j)}}\right)^{k_j} \left(\frac{-t_n^{(j)}\,dw_n^{(j)}}{(w_n^{(j)})^2}\right).$$
\end{proof}

\begin{remark}
We refer the reader to \cite{EskinKontsevichZorich2}[Figure 7.c.], which provides a visualization for exactly this setting.  In their figure the curve $\gamma_n$ and the expanding annuli $E_n$ and $G_n$ are completely clear.
\end{remark}

\begin{lemma}
\label{OnePartHoloNoVanCyl}
For $g \geq 2$, let $D \subset$ \RankOne ~be a Teichm\"uller disc.  Let $\{(X_n,\omega_n)\}_{n=0}^{\infty}$ be a sequence of surfaces in $D$ converging to a degenerate surface $(X',\omega')$, where $X'$ is a punctured surface with exactly one part and $\omega'$ is holomorphic and not identically zero.  If $\gamma_n$ is a closed curve on $(X_n, \omega_n)$ that pinches as $n$ tends to infinity resulting in a pair of punctures on $X'$, then $\gamma_n$ is a union of saddle connections of $(X_n, \omega_n)$, and there is no closed regular trajectory homotopic to $\gamma_n$.
\end{lemma}

\begin{proof}
The idea of this proof is identical to that of Lemma \ref{RankOneImpNotHoloEverywhere}.  By contradiction, assume that $\gamma_n$ is homotopic to a closed regular trajectory on $(X_n, \omega_n)$.  Then $\gamma_n$ determines a cylinder $C_1^{(n)} \subset X_n$.  By Theorem \ref{RankOneImpCP}, the foliation in which $(X_n,\omega_n)$ admits the cylinder $C_1^{(n)}$ is periodic.  Therefore, there is a collection of cylinders $\{C_1^{(n)}, \ldots, C_k^{(n)}\}$ that fill $X_n$.  Since it is assumed that $C_1^{(n)}$ does not exist on $(X', \omega')$ as $n$ tends to infinity, we must have $k > 1$.  Let $w_i^{(n)}$ denote the circumference of $C_i^{(n)}$.  By Corollary \ref{RankOneCylConfig}, the ratios $w_i^{(n)}/w_1^{(n)} = 1$ for all $i \leq k$ and $n \geq 0$.  Hence, if the core curve of $C_1^{(n)}$ pinches, then the core curve of every cylinder in that foliation pinches.  Since the ratios between the circumferences are constant, every sequence of cylinders $\{C_i^{(n)}\}_n$ converges to an infinite cylinder on $X'$, and $\omega'$ cannot be holomorphic at the pair of punctures resulting from pinching $\gamma_n$.  This contradicts the assumption that $\gamma_n$ is homotopic to a closed regular trajectory of $(X_n, \omega_n)$.
\end{proof}

\begin{lemma}
\label{DegVeechUnionPctdDiscs}
Let $(X',\omega')$ denote a degenerate surface in the boundary of the Teichm\"uller disc generated by a Veech surface.  If $S_1$ is a part of $X'$ and $\omega'$ has $m$ poles on $S_1$, then $S_1$ can be written as a union of $m$ punctured discs with an identification scheme for the boundary of the discs.
\end{lemma}

\begin{proof}
Every degenerate surface in the boundary of the Teichm\"uller disc of a Veech surface is given by fixing a periodic direction and degenerating the surface under the Teichm\"uller geodesic flow.  Let $(X,\omega)$ be the surface such that $\lim_{t \rightarrow \infty} G_t \cdot (X,\omega) = (X', \omega')$.  Then $(X,\omega)$ is a union of cylinders, and by \cite{MasurThesis}[Theorem 3], $(X', \omega')$ is given by pinching the core curves of each of these cylinders.  Therefore, $(X', \omega')$ is a union of punctured discs.
\end{proof}

\subsection{Setup of the Punctured Sequence}
\label{PuncSeqSetup}

Let $D$ be a Teichm\"uller disc contained in \RankOne .  By Theorem \ref{CPImpVeech}, if $D$ is not a Teichm\"uller curve, then there exists a sequence of surfaces $\{(X_n,\omega_n)\}_{n=1}^{\infty}$ in $D$ converging to a punctured Veech surface $(X',\omega')$ carrying a holomorphic Abelian differential.  By Lemma \ref{OnePartHoloNoVanCyl}, the assumptions of Lemma \ref{SphereBubble} are satisfied in this case.  We consider the specific case of Lemma \ref{SphereBubble} when $k_1 = k_2 = 0$ for the remainder of Section \ref{PctsAtRegPtsSect}.  In this case, the resulting punctures of $X'$ are regular points of $\omega'$.  This implies that the zeros along the union of saddle connections that form the pinching curve $\gamma_n$ have total order two.  Thus, $\gamma_n$ must be the union of one or two saddle connections.  Pick a zero $\xi_1$ in this union of saddle connections and puncture it as in the setup of Lemma \ref{SphereBubble} to form the sequence $\{(X_n^{(p)}, \omega_n^{(p)})\}_{n=0}^{\infty}$ in $\mathcal{M}_{g,1}$.  By Lemma \ref{SphereBubble}, we can regard the limit of this sequence as a degenerate surface consisting of the punctured Veech surface $(X',\omega')$ and a sphere $S'$.  The Veech surface $(X',\omega')$ can be degenerated, say under the Teichm\"uller geodesic flow, to a surface $(X'', \omega'')$ as described in Lemma \ref{TDDerPerRank1Lem}.  By the continuity of the \splin ~action \cite{BainbridgeMoller}[Proposition 11.1], we can abuse notation and assume the sequence $\{(X_n^{(p)}, \omega_n^{(p)})\}_{n=0}^{\infty}$ converges to a surface, denoted by $(X''^{(p)}, \omega''^{(p)})$, formed by degenerating the underlying Veech surface $(X', \omega') \subset (X'^{(p)}, \omega'^{(p)})$.  If $X'' = S_1 \sqcup \cdots \sqcup S_m$, then $X''^{(p)} = S_1 \sqcup \cdots \sqcup S_m \sqcup S'$ by Lemma \ref{SphereBubble}.  Let $(p_j, p_j')$ denote the pairs of punctures from $X'$ to $S'$, for $j = 1,2$.  Since $p_1$ and $p_2$ lie at regular points, Theorem \ref{PctsAreCCConj} says $p_1$ and $p_2$ always lie on a closed regular trajectory of $(X', \omega')$.  The punctures $p_1$ and $p_2$ must lie on the same part of $X''$ because $X''$ results from pinching a collection of closed curves of $X'$.

\begin{remark}
Note that this does not contradict Lemma \ref{TDDerPerRank1Lem} because that lemma assumes that the Teichm\"uller disc lies in the moduli space $\mathcal{M}_g$, \emph{not} $\mathcal{M}_{g,1}$.
\end{remark}

We consider a basis of Abelian differentials on the elements of $\mathcal{M}_{g,1}$.  It should be clear that a basis of Abelian differentials, such as the one described in Lemma \ref{AbDiffBasis}, still form a linearly independent set of Abelian differentials on a surface in $\overline{\mathcal{M}}_{g,1}$.  In fact, the Riemann-Roch theorem implies that the complex dimension of the bundle of Abelian differentials on $\mathcal{R}_{g,1}$ is $g+1-1 = g$.  Hence, any basis of Abelian differentials on a surface $X \in \mathcal{R}_g$ remains a basis for the space of Abelian differentials after puncturing $X$ once.  By the Cartan-Serre Theorem, this basis extends to the boundary of the moduli space.  Let $\{\theta_1(t_n,\tau_n), \ldots, \theta_g(t_n,\tau_n)\}$ denote the basis of Abelian differentials on $(X_n^{(p)}, \omega_n^{(p)})$.  We choose basis elements and denote them by $\theta_1(t_n,\tau_n)$ and $\theta_2(t_n,\tau_n)$ with very specific properties that will be advantageous to us in the course of our calculations below.

Without loss of generality, let $\theta_1(t_n,\tau_n) \equiv \omega''^{(p)}(t_n, \tau_n)$, for all $n$, and let $\theta_1(0,\tau_{\infty}) \equiv \omega''^{(p)}(0,\tau_{\infty})$.  Let $\theta_2(0,\tau_{\infty}) \equiv 0$ on $S_i$, for $i \geq 2$, and let $\theta_2(0,\tau_{\infty})$ be holomorphic everywhere on $S_1$ and across all punctures except at $(p_j, p_j')$, where it has simple poles for $j = 1,2$.  This implies that $\theta_2(0,\tau_{\infty})$ has simple poles at zero and infinity on $S'$.  Hence, $\theta_2(0,\tau_{\infty})$ can be written explicitly as $\pm dw/w$ on $S'$, for the local coordinate $w$ about $w = 0$ or $w = \infty$ on the Riemann sphere.

\begin{remark}
Puncturing a surface at a zero is similar to what is done in \cite{EskinKontsevichZorich2}, where every zero of a differential is punctured.  Degenerating saddle connections in their setting causes flat pairs of pants to ``bubble'' out of the surface as in \cite{EskinKontsevichZorich2}[Section 7].

In the formulas of this paper, the resulting degenerate surfaces ``seemingly'' carry Abelian differentials with higher order poles.  The moral of Lemma \ref{SphereBubble} is that the limiting differential only appears to have higher order poles because it is taken as a limit of differentials that are scaled at each step to prevent the limit from being identically zero.  In our setting this will provide us with exactly the needed information because the Abelian differential will determine a Beltrami differential and the magnitude of the Abelian differential will be irrelevant in that setting.
\end{remark}

\begin{lemma}
\label{ZerosLessThant}
Assume the notation established thus far in this section.  Let $\zeta_k$ be local coordinates for a neighborhood about the punctures $p_k$, for $k = 1,2$.  Let $t_n^{(k)}$ be the pinching parameter for these annuli.  Let
$$\omega_n^{(p)} = H_n(\zeta_k, t_n^{(k)}/\zeta_k )\,d\zeta_k.$$
Then the zeros of $H_n(\zeta_k, t_n^{(k)}/\zeta_k)$ in $\{|\zeta_k| < c'\} \subset S_1$ lie in the disc $\{|\zeta_k| \leq |t_n^{(k)}|/c''\} \subset S_1$, and the zeros of $H_n(t_n^{(k)}/\zeta_k, \zeta_k)$ in the disc $\{|\zeta_k| < c'' < 1\} \subset S'$ are contained in $\{|\zeta_k| \leq |t_n^{(k)}|/c'\}$.  Furthermore, the simple zeros (or double zero) of $H_n(t_n^{(k)}/\zeta_k, \zeta_k)$ on $S'$ are on the equator of $S'$.
%Then there exists $0 \leq \delta_n << 1$, which tends to zero with $n$, such that the zeros of $H_n(\zeta_k, t_n^{(k)}/\zeta_k)$ in $\{|\zeta_k| < c'\} \subset S_1$ lie in the disc $\{|\zeta_k| \leq |t_n^{(k)}|^{1-\delta_n}/c''\} \subset S_1$, and the zeros of $H_n(t_n^{(k)}/\zeta_k, \zeta_k)$ in the disc $\{|\zeta_k| < c'' < 1\} \subset S'$ are contained in $\{|\zeta_k| \leq |t_n^{(k)}|^{1-\delta_n}/c'\}$.  Furthermore, the simple zeros of $H_n(t_n^{(k)}/\zeta_k, \zeta_k)$ on $S'$ are on the equator of $S'$.
\end{lemma}

\begin{proof}
By assumption, the curve $\gamma$ is a union of one or two saddle connections containing zeros of total order two and therefore, $\gamma$ lies at the intersection of two expanding annuli, $E_n = \{|t_n^{(k)}|/c'' \leq |\zeta_k| < c'\}$ and $G_n$, where $G_n$ is defined similarly, but the constant $c'$ may be different.  By the construction of Section \ref{PuncSeqSetup}, $\gamma$ contains a puncture, thus both $E_n$ and $G_n$ each split into two pinching regions.  The curve $\gamma$ is given by $\{|\zeta_k| = |t_n^{(k)}|/c''\}$ in $\zeta_k$ coordinates on $S_1$, and with respect to the sphere, $\gamma$ is given by $\{|t_n^{(k)}/\zeta_k| = |t_n^{(k)}|/c''\} = \{|\zeta_k| = c''\}$.  After scaling the differential on $S'$ by multiplication by $1/c''$, the zeros along $\gamma$ lie on the equator of the sphere $S'$.

Next we claim that both $E_n$ and $G_n$ contain a large annulus, that does not contain any zeros of $\omega_n^{(p)}$.  It suffices to prove this claim for $E_n$ because the argument for $G_n$ is identical.  As mentioned above, the expanding annulus $E_n$ is the union of two annuli $R_z(t_n^{(k)})$ and $R_w(t_n^{(k)})$.  Let $\zeta_k$ be a local coordinate for $E_n$.  Fix the coordinate $\zeta_k$ such that the limit as $n$ goes to infinity of $\omega_n^{(p)}$ on $R_z(t_n^{(k)})$ is $d\zeta_k$.  By Lemma \ref{SphereBubble}, it follows that for $w = t_n^{(k)}/\zeta_k$, $\omega''^{(p)} = (w-z_1)(w-z_2)\,dw/w^2$, where $z_1$ and $z_2$ are zeros of the differential $\omega''^{(p)}$ on the sphere.

First, we consider $\omega_n^{(p)}$ on the annulus $R_z(t_n^{(k)})$.  In this case we can write
$$\omega_n^{(p)} = (2a_{10}^{(n)} + \varepsilon_n(\zeta_k, t_n^{(k)}/\zeta_k ))\,d\zeta_k,$$
such that given $\varepsilon > 0$ there exists a sufficiently large value of $n$, such that
$$\sup_{\zeta_k \in R_z(t_n^{(k)})} |\varepsilon_n(\zeta_k, t_n^{(k)}/\zeta_k )| < \varepsilon.$$
Since $\lim_{n \rightarrow \infty} 2a_{10}^{(n)} = 1$, the function $2a_{10}^{(n)} + \varepsilon_n(\zeta_k, t_n^{(k)}/\zeta_k )$ cannot have any zeros in the region $R_z(t_n^{(k)})$.

Secondly, we consider $\omega_n^{(p)}$ on the annulus $R_w(t_n^{(k)})$.  Recall that $R_w(t_n^{(1)}) \cup R_w(t_n^{(2)}) \subset S'$ and that the only zeros of $S'$ lie along the curve $\gamma_n$ which lies along the mutual boundaries of $R_w(t_n^{(1)})$ and $R_w(t_n^{(2)})$.  Since $S'$ does not contain any zeros in the complement of $\gamma$, the claim of the lemma follows.  Furthermore, we can scale the differential by a positive real constant if necessary to get that the zeros lie on the equator.
\end{proof}

We point out to the reader that the differential $\omega''^{(p)}$ must be identically zero on the sphere.  This is forced by, among other things, the fact that the Abelian differentials on degenerate surfaces have at most simple poles.  However, if we take a projective limit of $\omega_n^{(p)}$ restricted to the sphere, then we get the nonzero limit with double poles at the north and south poles of the sphere as described above.  Using the formulas of \cite{WolpertInfinitDeformations}, we show that our choice of coordinates such that $\omega''^{(p)} = d\zeta_k$ in neighborhoods about $p_k$ have direct consequences for the differential $\omega''^{(p)}$ on the sphere.  We begin by fixing a choice of coordinates $\zeta_k$ about $p_k$ on $S_1$, for $k = 1,2$ such that $\omega''^{(p)}(\zeta_k,0) = d\zeta_k$.  By the Laurent expansion for Abelian differentials on degenerating annuli given in \cite{WolpertInfinitDeformations}, we have
$$\omega_n^{(p)}(\zeta_k, t_n^{(k)}/\zeta_k) = \frac{2}{\zeta_k} \sum_{i,j \geq 0} a_{ij}^{(n)}\zeta_k^i \left(\frac{t_n^{(k)}}{\zeta_k}\right)^j\,d\zeta_k$$
on the annulus $R_z(t_n^{(k)})$.  The assumption that this converges to $d\zeta_k$ implies that
$$\lim_{n \rightarrow \infty} a_{i0}^{(n)} = 0$$
for all $i \not= 1$, and $\lim_{n \rightarrow \infty} 2a_{10}^{(n)} = 1$.  This has strong implications for the formulas on the annulus $R_w(t_n^{(k)})$ where the expansion is given by
$$\omega_n^{(p)}(t_n^{(k)}/\zeta_k, \zeta_k) = \frac{-2}{\zeta_k} \sum_{i,j \geq 0} a_{ij}^{(n)} \left(\frac{t_n^{(k)}}{\zeta_k}\right)^i\zeta_k^j\,d\zeta_k.$$
We claim that in order to achieve the desired nonzero Abelian differential on $S'$, we have to consider the differentials $\omega_n^{(p)}(t_n^{(k)}/\zeta_k, \zeta_k)/t_n^{(k)}$ on $R_w(t_n^{(k)})$.  In the equation below define $\varepsilon_n(t_n^{(k)}/\zeta_k, \zeta_k)$ to be the function that allows the equality to hold:
$$\frac{-2}{\zeta_k} \sum_{i,j \geq 0} a_{ij}^{(n)} \left(\frac{t_n^{(k)}}{\zeta_k}\right)^i\zeta_k^j$$
$$= \frac{-2}{\zeta_k} \sum_{j \geq 0} a_{0j}^{(n)} \zeta_k^j + \frac{-2}{\zeta_k} \sum_{j \geq 0} a_{1j}^{(n)} \left(\frac{t_n^{(k)}}{\zeta_k}\right)\zeta_k^j + \varepsilon_n(t_n^{(k)}/\zeta_k, \zeta_k)$$
$$= \frac{-2t_n^{(k)}}{\zeta_k^2} \left(\sum_{j \geq 0} \frac{a_{0j}^{(n)}}{t_n^{(k)}} \zeta_k^{j+1} + \sum_{j \geq 0} a_{1j}^{(n)} \zeta_k^j \right)+ \varepsilon_n(t_n^{(k)}/\zeta_k, \zeta_k)$$
\begin{equation}
\label{SphereExpans}
= \frac{-2t_n^{(k)}}{\zeta_k^2} \left(a_{10}^{(n)} + \left(\frac{a_{00}^{(n)}}{t_n^{(k)}} + a_{11}^{(n)}\right)\zeta_k + \left(\frac{a_{01}^{(n)}}{t_n^{(k)}} + a_{12}^{(n)}\right)\zeta_k^2 + \ldots \right) + \varepsilon_n(t_n^{(k)}/\zeta_k, \zeta_k).
\end{equation}
On the other hand, if $\xi_1$ and $\xi_2$ represent the zeros of the differential $\lim_{n \rightarrow \infty} \omega_n^{(p)}/t_n^{(k)}$ on $S'$, then the differential must be given in local coordinates about the puncture $p_k'$ by
$$\frac{-(\zeta_k - \xi_1)(\zeta_k-\xi_2)}{\zeta_k^2}\,d\zeta_k.$$
Combining these two perspectives yields a few results.  Firstly, the function $(\zeta_k^2/t_n^{(k)})\varepsilon_n(t_n^{(k)}/\zeta_k, \zeta_k)$ converges to zero pointwise for every $\zeta_k \in R_w(t_n^{(k)})$ simply by inspection of the series expansion for the function.  Secondly, all of the higher order coefficients of $\zeta_k^j$, for $j \geq 2$ subsumed by the ellipsis in Equation (\ref{SphereExpans}) must converge to zero by the assumption that there are no terms of order higher than two in the limit.  Finally, we must have
$$\lim_{n\rightarrow \infty}\frac{a_{00}^{(n)}}{t_n^{(k)}} + a_{11}^{(n)} < \infty, \qquad \lim_{n\rightarrow \infty}\frac{a_{01}^{(n)}}{t_n^{(k)}} + a_{12}^{(n)} = 1.$$

\subsection{Estimates for the Derivative of the Period Matrix}

\begin{lemma}
\label{b11PctdSurf}
Assuming the notation established in Sections \ref{DerPerMatSubSect} and \ref{PuncSeqSetup},
$$\lim_{n\rightarrow \infty} \left|\int_{X''^{(p)*}(t_n,\tau_{\infty})} \theta_1(0,\tau_{\infty})^2 \,d\mu_{\omega''^{(p)}}\right| = \infty.$$
\end{lemma}

\begin{proof}
This follows from our choice of $\theta_1(0,\tau_{\infty})$, and Lemmas \ref{BddOutsideDisks} and \ref{DPiDivergentTerm}.
\end{proof}

\begin{lemma}
\label{b12PctdSurf}
Assuming the notation established in Sections \ref{DerPerMatSubSect} and \ref{PuncSeqSetup}, there exists a constant $C$ such that, for all $n \geq N$,
$$\left|\int_{X''^{(p)*}(t_n,\tau_{\infty})} \theta_1(0,\tau_{\infty})\theta_2(0,\tau_{\infty}) \frac{\overline{\omega''^{(p)}(0,\tau_{\infty})}}{\omega''^{(p)}(0,\tau_{\infty})}\right| < C.$$
\end{lemma}

\begin{proof}
By definition of $\theta_1(0,\tau_{\infty})$ and $\theta_2(0,\tau_{\infty})$, the only part of $X''^{(p)}$ on which both differentials are not identically zero is $S_1$.  Therefore, 
$$\left|\int_{X''^{(p)*}(t_n,\tau_{\infty})} \theta_1(0,\tau_{\infty})\theta_2(0,\tau_{\infty}) \frac{\overline{\omega''^{(p)}(0,\tau_{\infty})}}{\omega''^{(p)}(0,\tau_{\infty})}\right|$$
$$ = \left|\int_{S_1^*(t_n,\tau_{\infty})} \theta_1(0,\tau_{\infty})\theta_2(0,\tau_{\infty}) \frac{\overline{\omega''^{(p)}(0,\tau_{\infty})}}{\omega''^{(p)}(0,\tau_{\infty})}\right| < C.$$
The last inequality follows from Lemmas \ref{BddOutsideDisks} and \ref{BddInDisks}.
\end{proof}

\begin{lemma}
\label{TruncSeries}
Let $\zeta_k$ be local coordinates for a neighborhood about the punctures $p_k$, for $k = 1,2$.  Let $t_n^{(k)}$ be the pinching parameter for the annulus $R_z(t_n^{(k)})$.  Let
$$\omega_n^{(p)} = \frac{2}{\zeta_k} \sum_{i,j \geq 0}a^{(n)}_{ij}\zeta_k^i(t_n^{(k)}/\zeta_k)^j\,d\zeta_k = H_n(\zeta_k, t_n^{(k)}/\zeta_k)\,d\zeta_k$$
be the differential defined above.  Let
$$h_n(\zeta_k, t_n^{(k)}/\zeta_k) = \frac{2}{\zeta_k}\left(a^{(n)}_{00} + a^{(n)}_{10}\zeta_k + a^{(n)}_{01}(t_n^{(k)}/\zeta_k)\right)$$
and $\theta_2(0, \tau_{\infty}, \zeta_k, 0) = F(\zeta_k)\,d\zeta_k$.  If $\omega_n^{(p)}$ does not have any zeros in the annulus $R_z(t_n^{(k)})$, then
$$\lim_{n \rightarrow \infty} \left| \int_{R_z(t_n^{(k)})} \theta_2(0, \tau_{\infty}, \zeta_k, 0)^2 \frac{\overline{\omega_n^{(p)}}}{\omega_n^{(p)}} - \int_{R_z(t_n^{(k)})} F(\zeta_k)^2 \frac{\overline{h_n(\zeta_k, t_n^{(k)}/\zeta_k)}}{h_n(\zeta_k, t_n^{(k)}/\zeta_k)} \,d\zeta_k \wedge d\bar\zeta_k \right| = 0.$$
\end{lemma}

\begin{proof}
Fix coordinates $\zeta_k$ such that $\omega''^{(p)} = d\zeta_k$ on $R_z(0)$.  Let
$$\varepsilon_n(\zeta_k, t_n^{(k)}/\zeta_k) = H_n(\zeta_k, t_n^{(k)}/\zeta_k) - h_n(\zeta_k, t_n^{(k)}/\zeta_k)$$
$$ = \frac{2}{\zeta_k}\left(a^{(n)}_{20}\zeta_k^2 + a^{(n)}_{11}t_n^{(k)} + a^{(n)}_{02}(t_n^{(k)}/\zeta_k)^2 + \ldots\right)\,d\zeta_k.$$
We want to show that
$$\lim_{t_n^{(k)} \rightarrow 0} \int_{R_z(t_n^{(k)})} F(\zeta_k)^2 \left(\frac{\overline{H_n(\zeta_k, t_n^{(k)}/\zeta_k) }}{H_n(\zeta_k, t_n^{(k)}/\zeta_k)} - \frac{\overline{h_n(\zeta_k, t_n^{(k)}/\zeta_k)}}{h_n(\zeta_k, t_n^{(k)}/\zeta_k)}\right)\,d\zeta_k \wedge d\bar \zeta_k = 0.$$
We perform some elementary algebra to get
$$\frac{\overline{h_n(\zeta_k, t_n^{(k)}/\zeta_k) +\varepsilon_n(\zeta_k, t_n^{(k)}/\zeta_k)}}{h_n(\zeta_k, t_n^{(k)}/\zeta_k) +\varepsilon_n(\zeta_k, t_n^{(k)}/\zeta_k)} - \frac{\overline{h_n(\zeta_k, t_n^{(k)}/\zeta_k)}}{h_n(\zeta_k, t_n^{(k)}/\zeta_k)}$$
$$ = \frac{\overline{\varepsilon_n(\zeta_k, t_n^{(k)}/\zeta_k)}h_n(\zeta_k, t_n^{(k)}/\zeta_k) - \varepsilon_n(\zeta_k, t_n^{(k)}/\zeta_k)\overline{h_n(\zeta_k, t_n^{(k)}/\zeta_k)}}{h_n(\zeta_k, t_n^{(k)}/\zeta_k)(h_n(\zeta_k, t_n^{(k)}/\zeta_k) +\varepsilon_n(\zeta_k, t_n^{(k)}/\zeta_k))}$$
$$= \frac{\varepsilon_n(\zeta_k, t_n^{(k)}/\zeta_k)}{h_n(\zeta_k, t_n^{(k)}/\zeta_k) +\varepsilon_n(\zeta_k, t_n^{(k)}/\zeta_k)}\frac{\overline{\varepsilon_n(\zeta_k, t_n^{(k)}/\zeta_k)}}{\varepsilon_n(\zeta_k, t_n^{(k)}/\zeta_k)}$$
$$\qquad - \frac{\varepsilon_n(\zeta_k, t_n^{(k)}/\zeta_k)}{h_n(\zeta_k, t_n^{(k)}/\zeta_k) +\varepsilon_n(\zeta_k, t_n^{(k)}/\zeta_k)}\frac{\overline{h_n(\zeta_k, t_n^{(k)}/\zeta_k)}}{h_n(\zeta_k, t_n^{(k)}/\zeta_k)}.$$

First, we consider the function 
$$\frac{\varepsilon_n(\zeta_k, t_n^{(k)}/\zeta_k)}{h_n(\zeta_k, t_n^{(k)}/\zeta_k) +\varepsilon_n(\zeta_k, t_n^{(k)}/\zeta_k)} = \frac{\varepsilon_n(\zeta_k, t_n^{(k)}/\zeta_k)}{H_n(\zeta_k, t_n^{(k)}/\zeta_k)}.$$
By Lemma \ref{ZerosLessThant}, all of the zeros of $H_n(\zeta_k, t_n^{(k)}/\zeta_k)$ in $\{|\zeta_k| < c'\}$ lie in the disc $\{|\zeta_k| \leq |t_n^{(k)}|/c''\}$.  Hence,
$$\inf_{\zeta_k \in R_z(t_n^{(k)})} |H_n(\zeta_k, t_n^{(k)}/\zeta_k)| = C_n > 0.$$
Recall that coordinates were fixed so that for all $\zeta_k \in R_z(t_n^{(k)})$, $H_n(\zeta_k, t_n^{(k)}/\zeta_k)$ converges to one.  Hence, for any sequence $\{z_n\}_n$ converging in $R_z(0)$ such that $z_n \in R_z(t_n^{(k)})$ we have
$$|H_n(z_n, t_n^{(k)}/z_n) - 1| = \left| 2a_{10}^{(n)} -1 + \frac{2a_{10}^{(n)}}{z_n} + \frac{2a_{01}^{(n)}t_n^{(k)}}{z_n} + \varepsilon_n(z_n, t_n^{(k)}/z_n) \right|$$
$$\leq | 2a_{10}^{(n)} -1| + \left|\frac{2a_{10}^{(n)}}{\sqrt{|t_n^{(k)}|}}\right| + \left|\frac{2a_{01}^{(n)}t_n^{(k)}}{\sqrt{|t_n^{(k)}|}}\right| + |\varepsilon_n(z_n, t_n^{(k)}/z_n)|$$
Since every term converges to zero, $\inf_n C_n \not= 0$ and there exists a constant $C_1'$, not depending on $n$ or $\zeta_k$, such that
$$1/|H_n(\zeta_k, t_n^{(k)}/\zeta_k)| < C_1'.$$

Now we consider the function $\varepsilon_n(\zeta_k, t_n^{(k)}/\zeta_k)$.  Let
$$\varepsilon_n(\zeta_k, t_n^{(k)}/\zeta_k) = \varepsilon_n(\zeta_k,0) + \sqrt{|t_n^{(k)}|}\tilde \varepsilon_n(\zeta_k, t_n^{(k)}/\zeta_k),$$
where $\tilde \varepsilon_n(\zeta_k, t_n^{(k)}/\zeta_k)$ is defined to be exactly the function that allows the equality to hold, and $\varepsilon_n(\zeta_k,0) = \sum_{i \geq 1} a_{i0}^{(n)}\zeta_k^i$.  Since $\omega''^{(p)} = d\zeta_k$, $\lim_{n \rightarrow \infty} a_{i0}^{(n)} = 0$, for all $i \geq 2$.  Furthermore, $|\tilde \varepsilon_n(\zeta_k, t_n^{(k)}/\zeta_k)| < C_2'$, for all $n$ and $\zeta_k \in R_z(t_n^{(k)})$, by inspection of the series expansion.

Let $F(\zeta_k) = g(\zeta_k)/\zeta_k$, where the function $g(\zeta_k)$ is defined exactly by this equation, and it is holomorphic, thus it is bounded on $R_z(t_n^{(k)})$.  Let $|g(\zeta_k)| < C_1''$, and let $\zeta_k = re^{i\phi}$.  We compute the following and show that it converges to zero:
$$\left| \int_{R_z(t_n^{(k)})} F(\zeta_k)^2 \frac{\varepsilon_n(\zeta_k, t_n^{(k)}/\zeta_k)}{H_n(\zeta_k, t_n^{(k)}/\zeta_k)}\frac{\overline{\varepsilon_n(\zeta_k, t_n^{(k)}/\zeta_k)}}{\varepsilon_n(\zeta_k, t_n^{(k)}/\zeta_k)}\,d\zeta_k \wedge d\bar\zeta_k \right|$$
$$ \leq \int_{R_z(t_n^{(k)})} \left| \left(\frac{g(\zeta_k)}{\zeta_k}\right)^2 \frac{\varepsilon_n(\zeta_k, t_n^{(k)}/\zeta_k)}{H_n(\zeta_k, t_n^{(k)}/\zeta_k)}\right| \,d\zeta_k \wedge d\bar\zeta_k $$
$$ \leq C_1'C_1''^2\int_{R_z(t_n^{(k)})} \left| \frac{1}{\zeta_k^2}\left(\varepsilon_n(\zeta_k,0) + \sqrt{|t_n^{(k)}|}\tilde \varepsilon_n(\zeta_k, t_n^{(k)}/\zeta_k)\right)\right| \,d\zeta_k \wedge d\bar\zeta_k $$
$$ \leq 2\pi C_1' C_1''^2\left(\int_{\sqrt{|t_n^{(k)}|}}^{c'} \left|\frac{\varepsilon_n(\zeta_k,0)}{r^2}\right|r\,dr + \int_{\sqrt{|t_n^{(k)}|}}^{c'} \left|\frac{\sqrt{|t_n^{(k)}|}\tilde \varepsilon_n(\zeta_k, t_n^{(k)}/\zeta_k)}{r^2}\right|r\,dr\right).$$
Consider the following integral and note that $\varepsilon_n(0,0) = 0$ by definition, so $\zeta_k$ can be factored out leaving a holomorphic function $\hat\varepsilon_n(\zeta_k)$, which also converges to zero because every coefficient in its series converges to zero.  Hence,
$$\lim_{n \rightarrow \infty} \int_{\sqrt{|t_n^{(k)}|}}^{c'} \left|\frac{\varepsilon_n(\zeta_k,0)}{r}\right|\,dr = \lim_{n \rightarrow \infty} \int_{\sqrt{|t_n^{(k)}|}}^{c'} \left|\hat\varepsilon_n(\zeta_k)\right|\,dr = 0.$$
Finally, the integral
$$\int_{\sqrt{|t_n^{(k)}|}}^{c'} \left|\frac{\sqrt{|t_n^{(k)}|}\tilde \varepsilon_n(\zeta_k, t_n^{(k)}/\zeta_k)}{r}\right|\,dr$$
$$\leq C_2'\sqrt{|t_n^{(k)}|} \int_{\sqrt{|t_n^{(k)}|}}^{c'}\frac{dr}{r} = C_2'\sqrt{|t_n^{(k)}|}\left(\log(c') - (1/2)\log(|t_n^{(k)}|)\right)$$
converges to zero with $|t_n^{(k)}|$.  To complete the proof of this lemma, note that this argument can be repeated for the term
$$\left| \int_{R_z(t_n^{(k)})} F(\zeta_k)^2 \frac{\varepsilon_n(\zeta_k, t_n^{(k)}/\zeta_k)}{H_n(\zeta_k, t_n^{(k)}/\zeta_k)} \frac{\overline{h_n(\zeta_k, t_n^{(k)}/\zeta_k)}}{h_n(\zeta_k, t_n^{(k)}/\zeta_k)} \,d\zeta_k \wedge d\bar \zeta_k\right|.$$
\end{proof}

We would like to prove Lemma \ref{TruncSeries} for the annulus $R_w(t_n^{(k)})$ as well.  However, this is not trivial because of the different nature of the differential in this region.  In light of the observations made in Section \ref{PuncSeqSetup}, we can prove a lemma analogous to Lemma \ref{TruncSeries} for the region $R_w(t_n^{(k)}) \subset S'$.

\begin{lemma}
\label{TruncSeriesSphere}
Let $\zeta_k$ be local coordinates for a neighborhood about the punctures $p_k'$, for $k = 1,2$.  Let $t_n^{(k)}$ be the pinching parameter for the annulus $R_w(t_n^{(k)})$.  Let
$$\omega_n^{(p)} = \frac{2}{\zeta_k} \sum_{i,j \geq 0}a^{(n)}_{ij}(t_n^{(k)}/\zeta_k)^i\zeta_k^j\,d\zeta_k = H_n(t_n^{(k)}/\zeta_k, \zeta_k)\,d\zeta_k$$
be the differential defined above.  Let
$$h_n(t_n^{(k)}/\zeta_k, \zeta_k) = \frac{2}{\zeta_k}\left(a^{(n)}_{00} + a^{(n)}_{10}(t_n^{(k)}/\zeta_k) + a^{(n)}_{01}\zeta_k + a^{(n)}_{11}t_n^{(k)} + a^{(n)}_{12}t_n^{(k)}\zeta_k\right)$$
and $\theta_2(0, \tau_{\infty}, 0, \zeta_k) = \pm d\zeta_k/\zeta_k$.  If $\omega_n^{(p)}$ does not have any zeros in the annulus $R_w(t_n^{(k)})$, then
$$\lim_{n \rightarrow \infty} \left| \int_{R_w(t_n^{(k)})} \theta_2(0, \tau_{\infty}, 0, \zeta_k)^2 \frac{\overline{\omega_n^{(p)}}}{\omega_n^{(p)}} - \int_{R_w(t_n^{(k)})} \frac{1}{\zeta_k^2} \frac{\overline{h_n(t_n^{(k)}/\zeta_k, \zeta_k)}}{h_n(t_n^{(k)}/\zeta_k, \zeta_k)} \,d\zeta_k \wedge d\bar\zeta_k \right| = 0.$$
\end{lemma}

\begin{proof}
The idea of this proof is naturally to follow the reasoning of the proof of Lemma \ref{TruncSeries} as closely as possible.  We start by fixing coordinates in the same manner.  Let
$$\varepsilon_n(t_n^{(k)}/\zeta_k, \zeta_k) = H_n(t_n^{(k)}/\zeta_k, \zeta_k) - h_n(t_n^{(k)}/\zeta_k, \zeta_k).$$
The same algebraic manipulation used before holds here and we claim that for all $\zeta_k \in R_w(t_n^{(k)})$,
$$\lim_{n \rightarrow \infty} \frac{\varepsilon_n(t_n^{(k)}/\zeta_k, \zeta_k)}{H_n(t_n^{(k)}/\zeta_k, \zeta_k)} = 0.$$
In fact, it will be easier to see that the reciprocal diverges to infinity.  This is equivalent to showing that
$$\left|\frac{h_n(t_n^{(k)}/\zeta_k, \zeta_k)}{\varepsilon_n(t_n^{(k)}/\zeta_k, \zeta_k)}\right|$$
diverges.  After factoring $t_n^{(k)}$ out of the numerator and denominator, we see that the limit of the numerator is nonzero for all $\zeta_k \in R_w(t_n^{(k)})$.  Furthermore, the denominator consists entirely of terms discussed above that all converge to zero.  This shows that the quantity diverges and the lemma follows.
\end{proof}

Since the proof of Lemma \ref{Forni42Seqs} always guarantees convergence if we take a limit of the basis differentials $\theta_i(t_n, \tau_n)$ in Rauch's formula, while fixing $\omega_n$, it suffices to evaluate the derivative of the period matrix for the basis differential $\theta_2(0,\tau_{\infty})$.  This differential is identically zero on $S_j$ for $j \geq 2$, so it suffices to calculate Rauch's formula on the remaining parts, namely, $S_1$ and the sphere $S'$.

\begin{lemma}
\label{Theta2ZeroOnS1}
Let $S_1^*(t_n) := X''^{(p)*}(t_n) \cap S_1$.  Given $\varepsilon$, there exists $N \geq 0$ such that for all $n \geq N$,
$$\left|\int_{S_1^*(t_n)} \theta_2(0, \tau_{\infty})^2 \frac{\overline{\omega_n^{(p)}}}{\omega_n^{(p)}}\right| < \varepsilon.$$
\end{lemma}

\begin{proof}
We begin by noting that $\omega''^{(p)}$ has two simple poles on $S_1$, and by Lemma \ref{DegVeechUnionPctdDiscs}, $S_1$ can be written as the union of two punctured discs $U_1'$ and $U_2'$.  We remove two small disjoint discs, denoted $D_1$ and $D_2$, in a neighborhood of $p_1, p_2 \in \partial U_1' = \partial U_2'$.  Let $U_i = U_i' \setminus (D_1 \cup D_2)$, for $i = 1,2$.  Then $S_1$ is the union of the four sets $\{U_1, U_2, D_1, D_2\}$ and the pairwise intersection of any two of these sets has measure zero.

By Lemma \ref{Forni42Seqs}, the derivative of the period matrix converges on $U_i$, for $i = 1,2$.  Choose local coordinates on $U_i$ so that $\omega''^{(p)}(0,\tau_{\infty},\zeta_k,0) = \pm\,d\zeta_k/\zeta_k$.  Furthermore, $\theta_2(0, \tau_{\infty}, \zeta_k, 0) = h(\zeta_k)\,d\zeta_k$ is holomorphic in $\zeta_k$ because it has no poles on $U_i$.  We get
$$\int_{U_i}\theta_2(0, \tau_{\infty})^2 \frac{\overline{\omega''^{(p)}}}{\omega''^{(p)}} = \int_{U_i} h(\zeta_k)^2\frac{\zeta_k}{\bar \zeta_k}\, d\zeta_k \wedge d\bar \zeta_k.$$
By a change of variables $d\zeta_k \wedge d\bar \zeta_k = (2r/\zeta_k)d\zeta_k \wedge dr$ and $\bar \zeta_k = r^2/\zeta_k$, this integral reduces to
$$\int \frac{2}{r} \int_{\gamma} \zeta_k h(\zeta_k)^2\, d\zeta_k \wedge dr,$$
where $\gamma$ is a small curve around the pole of $\omega''^{(p)}$ in $U_i$.  Since the integrand is holomorphic in the interior of $\gamma$, the middle integral is zero, thus
$$\int_{U_i}\theta_2(0, \tau_{\infty})^2 \frac{\overline{\omega''^{(p)}}}{\omega''^{(p)}} = 0.$$

Next, we consider the region $D_k$, for $k = 1,2$, which we write in the usual way as $R_z(t_n^{(k)})$ to indicate that we have removed a disc of radius $\sqrt{|t_n^{(k)}|}$ from the middle of $D_k$.  We choose the usual coordinates so that $\omega''^{(p)} = d\zeta_k$.  In these coordinates we can write $\theta_2(0, \tau_{\infty}, \zeta_k, 0) = (1/\zeta_k + g(\zeta_k))\,d\zeta_k$, where $g(\zeta_k)$ is holomorphic.  By Lemma \ref{TruncSeries}, we can truncate the series to three terms which implies that given $\varepsilon > 0$, there exists $N$ such that for all $n \geq N$,
$$\left| \int_{R_z(t_n^{(k)})}\theta_2(0, \tau_{\infty})^2 \frac{\overline{\omega_n^{(p)}}}{\omega_n^{(p)}}\right|$$
$$ < \left|\int_{R_z(t_n^{(k)})} \left(\frac{1}{\zeta_k} + g(\zeta_k)\right)^2 \frac{\zeta_k^2}{\overline{\zeta_k}^2}\frac{\overline{a_{00}^{(n)}\zeta_k+a_{10}^{(n)}\zeta_k^2 + a_{01}^{(n)}t_n^{(k)}}}{a_{00}^{(n)}\zeta_k+a_{10}^{(n)}\zeta_k^2 + a_{01}^{(n)}t_n^{(k)}}\, d\zeta_k \wedge d\overline{\zeta_k}\right| + \varepsilon.$$
Let
$$\left(\frac{1}{\zeta_k} + g(\zeta_k)\right)^2 = \frac{1}{\zeta_k^2} + \frac{\tilde g(\zeta_k)}{\zeta_k}.$$
This yields 
$$\left|\int_{R_z(t_n^{(k)})} \left(\frac{1}{\zeta_k^2} + \frac{\tilde g(\zeta_k)}{\zeta_k}\right) \frac{\zeta_k^2}{\overline{\zeta_k}^2}\frac{\overline{a_{00}^{(n)}\zeta_k+a_{10}^{(n)}\zeta_k^2 + a_{01}^{(n)}t_n^{(k)}}}{a_{00}^{(n)}\zeta_k+a_{10}^{(n)}\zeta_k^2 + a_{01}^{(n)}t_n^{(k)}}\, d\zeta_k \wedge d\overline{\zeta_k}\right|$$
$$\leq \left|\int_{R_z(t_n^{(k)})} \frac{1}{\zeta_k^2} \frac{\zeta_k^2}{\overline{\zeta_k}^2}\frac{\overline{a_{00}^{(n)}\zeta_k+a_{10}^{(n)}\zeta_k^2 + a_{01}^{(n)}t_n^{(k)}}}{a_{00}^{(n)}\zeta_k+a_{10}^{(n)}\zeta_k^2 + a_{01}^{(n)}t_n^{(k)}}\, d\zeta_k \wedge d\overline{\zeta_k}\right|$$
$$+ \left|\int_{R_z(t_n^{(k)})} \frac{\tilde g(\zeta_k)}{\zeta_k} \frac{\zeta_k^2}{\overline{\zeta_k}^2}\frac{\overline{a_{00}^{(n)}\zeta_k+a_{10}^{(n)}\zeta_k^2 + a_{01}^{(n)}t_n^{(k)}}}{a_{00}^{(n)}\zeta_k+a_{10}^{(n)}\zeta_k^2 + a_{01}^{(n)}t_n^{(k)}}\, d\zeta_k \wedge d\overline{\zeta_k}\right|.$$
Note that in the second term the integrand is an integrable function so by the dominated convergence theorem, we can let $n$ tend to infinity and use this as an approximation for sufficiently large $n$.  Therefore, we can approximate this second term by 
$$\left|\int_{R_z(t_n^{(k)})} \frac{\tilde g(\zeta_k)}{\zeta_k} \, d\zeta_k \wedge d\overline{\zeta_k}\right| = 0,$$
and the equality follows because the integrand is a meromorphic function.  Therefore, it suffices to concentrate on the first term and absorb the second term into the $\varepsilon$ already in the bound above.

Using the same change of variables as before and letting $\gamma$ denote a curve in $R_z(t_n^{(k)})$ homotopic to a boundary curve of $R_z(t_n^{(k)})$ yields
$$\int_{\sqrt{|t_n^{(k)}|}}^{c'} 2r \int_{\gamma}\frac{\zeta_k^2}{r^4} \frac{1}{\zeta_k}\frac{\overline{a_{00}^{(n)}}(r^2/\zeta_k)+\overline{a_{10}^{(n)}}(r^4/\zeta_k^2) + \overline{a_{01}^{(n)}t_n^{(k)}}}{a_{00}^{(n)}\zeta_k+a_{10}^{(n)}\zeta_k^2 + a_{01}^{(n)}t_n^{(k)}}\, d\zeta_k \wedge dr$$
$$=\int_{\sqrt{|t_n^{(k)}|}}^{c'} \frac{2}{r^3} \int_{\gamma} \frac{1}{\zeta_k}\frac{\overline{a_{00}^{(n)}}r^2\zeta_k+\overline{a_{10}^{(n)}}r^4 + \overline{a_{01}^{(n)}t_n^{(k)}}\zeta_k^2}{a_{00}^{(n)}\zeta_k+a_{10}^{(n)}\zeta_k^2 + a_{01}^{(n)}t_n^{(k)}}\, d\zeta_k \wedge dr.$$
To simplify the computation, we note that the sum of the residues of an Abelian differential on the Riemann sphere is zero and so it suffices to calculate the residue at the pole at infinity.  We perform the usual substitution $\zeta_k \mapsto 1/\zeta_k$ and let $\gamma'$ denote a small curve around the point at infinity.  We consider the integral
$$-\int_{\gamma'}\frac{1}{\zeta_k}\frac{\overline{a_{00}^{(n)}}r^2/\zeta_k+\overline{a_{10}^{(n)}}r^4 + \overline{a_{01}^{(n)}t_n^{(k)}}(1/\zeta_k^2)}{a_{00}^{(n)}(1/\zeta_k)+a_{10}^{(n)}(1/\zeta_k^2) + a_{01}^{(n)}t_n^{(k)}}\, d\zeta_k$$
$$= -\int_{\gamma'}\frac{1}{\zeta_k}\frac{\overline{a_{00}^{(n)}}r^2\zeta_k+\overline{a_{10}^{(n)}}r^4\zeta_k^2 + \overline{a_{01}^{(n)}t_n^{(k)}}}{a_{00}^{(n)}(\zeta_k)+a_{10}^{(n)} + a_{01}^{(n)}t_n^{(k)}\zeta_k^2}\, d\zeta_k = -2\pi\sqrt{-1}\frac{\overline{a_{01}^{(n)}t_n^{(k)}}}{a_{10}^{(n)}}.$$
Returning to the original integral under consideration, we have
$$\int_{\sqrt{|t_n^{(k)}|}}^{c'} \frac{2}{r^3} \left(-2\pi\sqrt{-1}\frac{\overline{a_{01}^{(n)}t_n^{(k)}}}{a_{10}^{(n)}}\right)dr = 2\pi\sqrt{-1}\frac{\overline{a_{01}^{(n)}t_n^{(k)}}}{a_{10}^{(n)}}\left(\frac{1}{c'^2} - \frac{1}{|t_n^{(k)}|}\right).$$
By the observation that the quantities $a_{01}^{(n)}/t_n^{(k)} + a_{12}^{(n)}$ and $a_{12}^{(n)}$ converge to a finite limit as $n$ tends to infinity, we have $\lim_{n \rightarrow \infty} a_{01}^{(n)} = 0$.  Since $a_{10}^{(n)}$ converges to $1/2$ and $a_{01}^{(n)}$ converges to zero, both
$$2\pi\sqrt{-1}\frac{\overline{a_{01}^{(n)}t_n^{(k)}}}{a_{10}^{(n)}c'^2} \quad \text{and} \quad -2\pi\sqrt{-1}\frac{\overline{a_{01}^{(n)}}}{a_{10}^{(n)}}\left(\frac{t_n^{(k)}}{|t_n^{(k)}|}\right)$$
converge to zero as $n$ tends to infinity.  Therefore, the integral converges to zero on all four regions into which $S_1$ was subdivided.
\end{proof}

\begin{lemma}
\label{Theta2NotZeroOnSphere}
Let $S'^*(t_n) := X''^{(p)*}(t_n) \cap S'$.  Given $\varepsilon > 0$, there exists $N \geq 0$ such that for all $n \geq N$,
$$\left|\int_{S'^*(t_n^{(k)})} \theta_2(0, \tau_{\infty})^2 \frac{\overline{\omega_n^{(p)}}}{\omega_n^{(p)}}\right| > 4\pi - \varepsilon.$$
\end{lemma}

\begin{proof}
Since $S'^*(t_n^{(k)})$ is a sphere with holes around the two poles of $\omega''^{(p)}$, it can be written as a union of two charts $U_0(t_n^{(k)}) = \{\zeta_k \in S'^*(t_n^{(k)}) | |\zeta_k| \leq 1\}$ and $U_{\infty}(t_n^{(k)}) = \{\zeta_k \in S'^*(t_n^{(k)}) | |\zeta_k| > 1\}$.  By Lemma \ref{TruncSeriesSphere}, and using the notation of that lemma, we can approximate the integral using the differential $\tilde \omega_n^{(p)} = h_n(t_n^{(k)}/\zeta_k, \zeta_k)\,d\zeta_k$.
If we factor out $1/\zeta_k^2$ from the right hand side, we get a quadratic polynomial as the numerator of the resulting fraction.  Let $\xi_1(t_n^{(k)})$ and $\xi_2(t_n^{(k)})$ denote the two, not necessarily distinct, zeros of this quadratic polynomial, which clearly depend on $n$.  As in Section \ref{PuncSeqSetup}, this truncated differential can be written as
$$\tilde \omega_n^{(p)}(0, \tau_{\infty}, t_n^{(k)}/\zeta_k, \zeta_k) = \frac{(\zeta_k-\xi_1(t_n^{(k)}))(\zeta_k-\xi_2(t_n^{(k)}))}{\zeta_k^2}\,d\zeta_k.$$
Since the zeros in the limit satisfy the equation $|\xi_1(0)| = |\xi_2(0)| = 1$, we have
$$\lim_{n \rightarrow \infty} |\xi_1(t_n^{(k)})| = \lim_{n \rightarrow \infty} |\xi_2(t_n^{(k)})| = 1.$$
Let $1-\delta_n = \min\{1, |\xi_1(t_n^{(k)})|, |\xi_2(t_n^{(k)})|\}$.  Note that $\tilde \omega_n^{(p)}$ is holomorphic in the region $B_n = \{1-\delta_n \leq |\zeta_k| \leq 1+\delta_n\}$ and converges to the holomorphic differential $\omega''^{(p)}$.  Since the region $B_n$ is compact with area vanishing as $n$ tends to infinity, the integral on this region converges to zero and it can be ignored in the remainder of the calculations below.

We compute the derivative of the period matrix on the two charts specified above.  As we noted in Section \ref{PuncSeqSetup}, we must have $\theta_2(0, \tau_{\infty}, 0, \zeta_k) = \pm\,d\zeta_k/\zeta_k$ because the differential on the sphere is determined up to a constant by the fact that there are exactly two simple poles.  This allows us to write the integral explicitly as
$$\int_{U_0(t_n^{(k)})} \theta_2(0, \tau_{\infty}, 0, \zeta_k)^2 \frac{\overline{\tilde \omega_n^{(p)}}}{\tilde \omega_n^{(p)}}$$
$$ = \int_{U_0(t_n^{(k)})} \left(\frac{1}{\zeta_k}\right)^2 \frac{\zeta_k^2\overline{(\zeta_k-\xi_1(t_n^{(k)}))(\zeta_k-\xi_2(t_n^{(k)}))}}{\overline{\zeta_k}^2(\zeta_k-\xi_1(t_n^{(k)}))(\zeta_k-\xi_2(t_n^{(k)}))}\, d\zeta_k \wedge d\bar \zeta_k$$
Letting $\zeta_k = re^{i\theta}$, implies $\bar \zeta_k = r^2/\zeta_k$ and 
$$d\zeta_k \wedge d\bar \zeta_k = \frac{2r}{\zeta_k}\,d\zeta_k \wedge dr.$$
The integral simplifies to
$$\int_{\sqrt{|t_n^{(k)}|}}^{1-\delta_n} \frac{2}{r^3}\int_{|\zeta_k| = r} \frac{(r^2 - \overline{\xi_1(t_n^{(k)})}\zeta_k)(r^2 - \overline{\xi_2(t_n^{(k)})}\zeta_k)}{\zeta_k(\zeta_k-\xi_1(t_n^{(k)}))(\zeta_k-\xi_2(t_n^{(k)}))} \,d\zeta_k \wedge dr.$$
Since the zeros $\xi_1(t_n^{(k)})$ and $\xi_2(t_n^{(k)})$ lie outside the outer boundary of the region of integration, the interior integral is equal to the residue of the integrand at zero, which yields
$$\int_{\sqrt{|t_n^{(k)}|}}^{1-\delta_n} \frac{2}{r^3} \frac{2\pi\sqrt{-1}r^4}{\xi_1(t_n^{(k)})\xi_2(t_n^{(k)})} \,dr = 2\pi\sqrt{-1}\frac{(1 -\delta_n)^2- |t_n^{(k)}|}{\xi_1(t_n^{(k)})\xi_2(t_n^{(k)})}.$$

Next we compute the integral on the chart $U_{\infty}(t_n^{(k)})$.  This can be done by the usual change of coordinates on the sphere $\zeta_k \mapsto 1/\zeta_k$.  After simplifying, we get
$$\int_{U_{\infty}(t_n^{(k)})} \theta_2(0, \tau_{\infty}, 0, \zeta_k)^2 \frac{\overline{\omega_n^{(p)}}}{\omega_n^{(p)}}$$
$$= \int_{U_0(t_n^{(k)})} \overline{\zeta_k}^2 \frac{\overline{(1/\zeta_k-\xi_1(t_n^{(k)}))(1/\zeta_k-\xi_2(t_n^{(k)}))}}{(1/\zeta_k-\xi_1(t_n^{(k)}))(1/\zeta_k-\xi_2(t_n^{(k)}))} \frac{d\zeta_k \wedge d\bar \zeta_k}{|\zeta_k|^4}.$$
We make the same substitution as before so that the differential becomes $d\zeta_k \wedge dr$, and this yields
$$\int_{\sqrt{|t_n^{(k)}|}}^{1-\delta_n} \frac{2}{r^3}\int_{|\zeta_k| = r} \frac{1}{\zeta_k}\frac{(\zeta_k - \overline{\xi_1(t_n^{(k)})}r^2)(\zeta_k - \overline{\xi_2(t_n^{(k)})}r^2)}{(1-\zeta_k\xi_1(t_n^{(k)}))(1-\zeta_k\xi_2(t_n^{(k)}))} \,d\zeta_k \wedge dr.$$
As above, the zeros $\xi_1(t_n^{(k)})$ and $\xi_2(t_n^{(k)})$ lie outside the outer boundary of the region of integration, so the interior integral is equal to the residue of the integrand at zero, which gives us
$$\int_{\sqrt{|t_n^{(k)}|}}^{1-\delta_n} \frac{2}{r^3}2\pi\sqrt{-1}\left(\overline{\xi_1(t_n^{(k)})}r^2\right)\left(\overline{\xi_2(t_n^{(k)})}r^2\right)dr$$
$$ = 2\pi\sqrt{-1}\left(\overline{\xi_1(t_n^{(k)})\xi_2(t_n^{(k)})}\right)((1 -\delta_n)^2 - |t_n^{(k)}|).$$
Therefore,
$$\left|\int_{S'^*(t_n^{(k)})} \theta_2(0, \tau_{\infty}, 0, \zeta_k)^2 \frac{\overline{\omega_n^{(p)}}}{\omega_n^{(p)}}\right|$$
$$ = 2\pi\left|(1 -\delta_n)^2 - |t_n^{(k)}|\right|\left|\frac{1}{\xi_1(t_n^{(k)})\xi_2(t_n^{(k)})} + \overline{\xi_1(t_n^{(k)})\xi_2(t_n^{(k)})}\right|$$
$$= 4\pi\left|(1 -\delta_n)^2 - |t_n^{(k)}|\right|.$$
\end{proof}

\subsection{Main Theorem and Surfaces of High Genus}
\label{RegPtsMainThm}

\begin{theorem}
\label{NoPctsAtCCCPts}
Let $D$ be a Teichm\"uller disc contained in \RankOne .  Let $(X',\omega')$ be a degenerate (punctured) surface carrying a holomorphic Abelian differential in the closure of $D$ and let $(X',\omega')$ be a Veech surface.  If $(p,p')$ are a pair of punctures on $X'$, then $\omega'$ has a zero at $p$ or $p'$.
\end{theorem}

\begin{proof}
By contradiction, assume that both $p$ and $p'$ lie at regular points of $(X', \omega')$.  By Lemma \ref{PctsAreCCConj}, $p' \in C_p(X)$.  We follow the notation established in Section \ref{PuncSeqSetup}.  Since $\omega_k^{(p)} = \omega_n a.e.$, $d\Pi(X_n^{(p)})/d\mu_{\omega_n^{(p)}} = d\Pi(X_n)/d\mu_{\omega_n}$, for all $n$.

We estimate the derivative of the period matrix near $(X''^{(p)}, \omega''^{(p)})$ to achieve a contradiction by finding a $2 \times 2$ minor of full rank.  Let $\theta_1(0, \tau_{\infty})$ and $\theta_2(0, \tau_{\infty})$ be the elements of the basis of Abelian differentials on $X_n^{(p)}$ described above.  Consider the $2 \times 2$ minor of the derivative of the period matrix
\[ \left( \begin{array}{cc}
\int_{X''^{(p)*}(t_n)} \theta_1(0, \tau_{\infty})^2 d\mu_{\omega''^{(p)}} & \int_{X''^{(p)*}(t_n)} \theta_1(0, \tau_{\infty}) \theta_2(0, \tau_{\infty}) d\mu_{\omega''^{(p)}} \\
\int_{X''^{(p)*}(t_n)} \theta_2(0, \tau_{\infty}) \theta_1(0, \tau_{\infty}) d\mu_{\omega''^{(p)}} & \int_{X''^{(p)*}(t_n)} \theta_2(0, \tau_{\infty})^2 d\mu_{\omega''^{(p)}} \end{array} \right) . \]
By Lemma \ref{b11PctdSurf}, the $(1,1)$ component diverges to infinity.  The matrix is symmetric and by Lemma \ref{b12PctdSurf}, the diagonal terms are bounded.  Since the $(2,2)$ component is nonzero by Lemmas \ref{Theta2ZeroOnS1} and \ref{Theta2NotZeroOnSphere}, the determinant of this minor is nonzero, thus it has full rank.  This implies that the derivative of the period matrix must have rank at least two and the contradiction completes the proof of the theorem.
\end{proof}

The main corollary of this theorem is a new proof that for sufficiently large genus, there are no \splin -invariant ergodic measures with completely degenerate Kontsevich-Zorich spectrum.  Though the bound produced here is weaker than that of \cite{EskinKontsevichZorich2}, where it is proven that there are no regular \splin -invariant suborbifolds supporting such a measure for $g \geq 7$, the result below does not rely on the work of \cite{EskinMirzakhaniInvariantMeas} and the conjecture that every \splin -invariant suborbifold is regular \cite{EskinKontsevichZorich2}[Section 1.5].

\begin{corollary}
\label{NoTeichDiskSuffHighGenus}
For $g \geq 13$, there are no Teichm\"uller discs contained in \RankOne .
\end{corollary}

\begin{proof}
By \cite{MollerShimuraTeich} (see Theorem \ref{MollerClass}), there are no Veech surfaces in \RankOne , for $g \geq 6$.  Moreover, M\"oller proved that a Veech surface generating a Teichm\"uller disc in \RankOne ~has at most seven distinct zeros in genus five, \cite{MollerShimuraTeich}[Corollary 5.15].  By Theorem \ref{NoPctsAtCCCPts}, at least one puncture in each pair of punctures must lie at a zero.  In the worst case, there could be a surface in genus $5+7=12$, generating a Teichm\"uller disc in \RankOne , that degenerates to a Veech surface $(X',\omega')$ in genus five with exactly seven pairs of punctures, such that one puncture in each pair of punctures lies at a zero of $\omega'$.
\end{proof}

\begin{lemma}
\label{NoDegToGen3}
Given a surface $(X,\omega)$ generating a Teichm\"uller disc $D$ in \RankOne , for $g \geq 4$, there does not exist a sequence of surfaces in $D$ converging to the surface $(M_3', \omega_{M_3'})$ with punctures.
\end{lemma}

\begin{proof}
Let $(p,p')$ be a pair of punctures on $M_3'$.  Both $p$ and $p'$ cannot lie at regular points by Theorem \ref{NoPctsAtCCCPts} and neither $p$ nor $p'$ can lie at a zero of $\omega_{M_3'}$ by Lemma \ref{Gen3NoPctsAtZeros}.
\end{proof}

\begin{lemma}
\label{NoGen5Deg}
Let $(X,\omega)$ generate a Teichm\"uller disc $D$ contained in \RankOne , for $g \geq 5$.  If $(X', \omega')$ is a Veech surface, which is a limit of a sequence of surfaces in $D$, then $X'$ has genus five.
\end{lemma}

\begin{proof}
By \cite{MollerShimuraTeich} (see Theorem \ref{MollerClass}), the only Veech surfaces generating a Teichm\"uller disc in \RankOne ~have genus three, four, or possibly five.  The differential $\omega'$ is holomorphic by Theorem \ref{CPImpVeech} and not identically zero by Lemma \ref{NonZeroPart}, so $X'$ has positive genus $g'$.  Lemma \ref{NoDisksInD11} implies $g' \not= 1$, Proposition \ref{Genus2Cor} implies $g' \not= 2$, and Lemma \ref{NoDegToGen3} implies $g' \not= 3$.  If $(X', \omega') = (M_4', \omega_{M_4'})$, then there would be a pair of punctures $(p, p')$ on $(M_4', \omega_{M_4'})$.  However, both $p$ and $p'$ cannot lie at regular points by Theorem \ref{NoPctsAtCCCPts} and neither $p$ nor $p'$ can lie at a zero of $\omega_{M_4'}$ by Lemma \ref{Gen4NoPctsAtZeros}.
\end{proof}

Though M\"oller did not prove that there are no Veech surfaces generating Teichm\"uller discs in $\mathcal{D}_5(1)$, he did an extensive computer search that did not find any such surface.  This search strongly convinced him, as it does the author of this paper, that the following conjecture is true.

\begin{conjecture}[M\"oller]
There are no Teichm\"uller curves in $\mathcal{D}_5(1)$.
\end{conjecture}

\begin{condthm}
\label{CondThmGen5}
There are no Teichm\"uller discs contained in \RankOne , for $g \geq 5$.
\end{condthm}

\begin{proof}
Combine Lemma \ref{NoGen5Deg} with M\"oller's conjecture.
\end{proof}

\begin{remark}
In the unlikely case that M\"oller's conjecture is false and there does in fact exist one or more genus five examples, the author strongly believes that the existence of Teichm\"uller discs in genus six and above can still be ruled out by writing down all of the genus five examples and using a proof as in Lemmas \ref{Gen3NoPctsAtZeros} and \ref{Gen4NoPctsAtZeros} to say that no puncture lies at a zero.  Then Theorem \ref{NoPctsAtCCCPts} would conclude the non-existence of such discs for $g \geq 6$.
\end{remark}

\begin{remark}
The author would like to point out that the techniques used in this paper can be used to give simple proofs that many of the strata in the moduli spaces $\mathcal{M}_g$, for $5 \leq g \leq 12$, do not contain Teichm\"uller discs in \RankOne .  However, the author sees little value in listing all such strata when every stratum can be excluded for $g \geq 5$ by proving M\"oller's conjecture.
\end{remark}

\section{The Teichm\"uller Disc in $\mathcal{D}_4(1)$}

The goal of this section is to prove Theorem \ref{Gen4Class}, which says that the Ornithorynque $(M_4, \omega_{M_4})$, discovered by \cite{ForniMatheus}, and depicted in Figure \ref{Gen4Ex}, generates the only Teichm\"uller disc in $\mathcal{D}_4(1)$.  Throughout this section we adopt the standard shorthand for strata, e.g. $\mathcal{H}(1^4, 2) := \mathcal{H}(1,1,1,1,2)$.

\begin{lemma}
\label{Gen4CannotDeg}
If $(X,\omega)$ is not a Veech surface, $(X,\omega)$ generates a Teichm\"uller disc $D$ in $\mathcal{D}_4(1)$, and a sequence of surfaces in $D$ converges to a Veech surface $(X', \omega')$, then $X'$ has genus four.
\end{lemma}

\begin{proof}
The surface $X'$ cannot have positive genus less than four by Lemmas \ref{NoDisksInD11} and \ref{NoDegToGen3}, and Proposition \ref{Genus2Cor}.  Recall that $X'$ cannot be a sphere either because $\omega'$ is holomorphic by Theorem \ref{CPImpVeech} and $\omega'$ is nonzero by Lemma \ref{NonZeroPart}.
\end{proof}

\begin{lemma}
\label{MaxCyls}
If $(X,\omega)$ generates a Teichm\"uller disc in $\mathcal{D}_4(1)$, then $(X,\omega)$ decomposes into at most three cylinders.
\end{lemma}

\begin{proof}
The top of every cylinder must have a positive, even number of zeros counted with multiplicity, and the total order of the zeros of $\omega$ is six.
\end{proof}

Ideally, we would like to use the same exact proof as Theorem \ref{Genus3Classification} to show that the genus four surface $(M_4, \omega_{M_4})$ generates the only Teichm\"uller disc in $\mathcal{D}_4(1)$.  However, this is not possible because $(M_4, \omega_{M_4})$ does not lie in the principal stratum as the genus three example does.  A priori, it is possible for zeros to converge under the conditions of Theorem \ref{CPImpVeech} without reaching a contradiction.  On the other hand, this technique can prove the result in most of the strata of $\mathcal{M}_4$.

\begin{lemma}
\label{NoDisksSomeGen4Strata}
There are no Teichm\"uller discs in $\mathcal{D}_4(1)$ except possibly in the strata $\mathcal{H}(2^3)$, $\mathcal{H}(1^2,2^2)$, $\mathcal{H}(1^4,2)$, and $\mathcal{H}(1^6)$.  Furthermore, $(M_4, \omega_{M_4})$ generates the only Teichm\"uller disc in $\mathcal{H}(2^3) \cap \mathcal{D}_4(1)$.
\end{lemma}

\begin{proof}
By \cite{MollerShimuraTeich} (see Theorem \ref{MollerClass}), $(M_4, \omega_{M_4})$ generates the only Teichm\"uller curve in $\mathcal{D}_4(1)$.  Hence, any other Teichm\"uller disc $D$ must be generated by a surface $(X,\omega)$, which is completely periodic by Theorem \ref{RankOneImpCP}, but not Veech.  By Theorem \ref{CPImpVeech}, there exists a sequence of surfaces in $D$ converging to a Veech surface $(X', \omega')$ in $\overline{\mathcal{D}_4(1)}$.  The surface $X'$ cannot have genus less than four by Lemma \ref{Gen4CannotDeg}.  Moreover, it is impossible to collapse zeros in any strata other than $\mathcal{H}(1^2,2^2)$, $\mathcal{H}(1^4,2)$, and $\mathcal{H}(1^6)$, which are excluded in the statement of the lemma, and converge to the Veech surface in $\mathcal{H}(2^3)$.

Since any other Teichm\"uller disc in $\mathcal{H}(2^3) \cap \mathcal{D}_4(1)$ must be generated by a non-Veech surface $(X,\omega)$, the zeros of $(X,\omega)$ can be collapsed to reach a contradiction.
\end{proof}

Lemma \ref{NoDisksSomeGen4Strata} says that the classification problem is complete in genus four except for three strata.  The remainder of this section is dedicated to addressing those strata.  The strategy is similar to the one used to prove Theorem \ref{CPImpVeech}.

\begin{lemma}
\label{H2211TopDichot}
If $(X,\omega)$ generates a Teichm\"uller disc $D$ in $\mathcal{H}(1^2,2^2) \cap \mathcal{D}_4(1)$, then $(X,\omega)$ satisfies topological dichotomy.
\end{lemma}

\begin{proof}
By Theorem \ref{RankOneImpCP}, $(X,\omega)$ is completely periodic.  We show that every saddle connection between two zeros must lie in a periodic foliation.  First, consider any saddle connection $\sigma$ from a double zero, denoted by $z$, to any other zero, denoted $z'$.  If $\sigma$ does not lie in a periodic foliation, then we can act on it by the Teichm\"uller geodesic flow so that it contracts at the maximal rate and choose a subsequence of times $\{t_n\}_n$ as in the proof of Lemma \ref{CPZeroConv} such that $G_{t_n} \cdot (X, \omega)$ converges to a surface $(X', \omega')$, where $\omega'$ is holomorphic.  The surface $X'$ cannot degenerate to a lower genus surface by Lemma \ref{Gen4CannotDeg}, so $\sigma$ must degenerate to a point resulting in a zero of order strictly greater than two.  However, there are no such Teichm\"uller discs in a stratum with a zero of order strictly greater than two by Lemma \ref{NoDisksSomeGen4Strata}.  This contradiction implies the saddle connection $\sigma$, which does \emph{not} lie in a periodic foliation, can only lie between the two simple zeros denoted by $z_1$ and $z_2$.

Without loss of generality, assume $\sigma$ has length $\varepsilon > 0$ and $(X,\omega)$ has a cylinder decomposition consisting of cylinders with unit circumference.  By Lemma \ref{MaxCyls}, $(X,\omega)$ is a union of one to three cylinders.  Degenerating the cylinders under the Teichm\"uller geodesic flow results in a surface as described in Lemma \ref{TDDerPerRank1Lem}.  This implies that the total order of the zeros on the top (and bottom) of every cylinder in the cylinder decomposition must be even because every part of the degenerate surface has two poles.  In the stratum $\mathcal{H}(1^2,2^2)$, this forces the two simple zeros to lie on the top of the same cylinder in every cylinder decomposition of $(X,\omega)$.  As usual, assume the area of the surface is one and the lengths of the waists of the cylinders are also one, so that the total heights of the cylinders is one.  Since $\sigma$ does not lie in a periodic foliation, it must leave $z_1$ and travel up the entire height of all the cylinders before reaching $z_2$.  However, this implies that $\sigma$ has length at least $1 > \varepsilon$ and this contradiction implies that all saddle connections of $(X,\omega)$ must lie in a periodic foliation.
\end{proof}

\begin{lemma}
\label{NoTDInH2211D41}
There are no Teichm\"uller discs contained in $\mathcal{H}(1^2,2^2) \cap \mathcal{D}_4(1)$.
\end{lemma}

\begin{proof}
We prove this lemma by showing that if $(X,\omega) \in \mathcal{H}(1^2,2^2)$ generates a Teichm\"uller disc in $\mathcal{D}_4(1)$, then $(X,\omega)$ is uniformly completely periodic.  By \cite{SmillieWeissCharLattice}, $(X,\omega)$ is a Veech surface and by \cite{MollerShimuraTeich}, there are no Veech surfaces in $\mathcal{H}(1^2,2^2)$ that generate a Teichm\"uller disc in $\mathcal{D}_4(1)$.  This contradiction will imply the lemma.

By contradiction, assume that there exists a surface $(X,\omega) \in \mathcal{H}(1^2,2^2)$ generating a Teichm\"uller disc in $\mathcal{D}_4(1)$.  By \cite{MollerShimuraTeich}, $(X,\omega)$ is not a Veech surface and by Lemma \ref{H2211TopDichot}, $(X,\omega)$ satisfies topological dichotomy.  As in the proof of the previous lemma, the only two zeros that are permitted to converge in the context of Theorem \ref{CPImpVeech} are the simple zeros $z_1$ and $z_2$.  Without loss of generality, let $\sigma$ be a saddle connection between $z_1$ and $z_2$ of length $\varepsilon > 0$.  Consider a cylinder decomposition of $(X,\omega)$, $C_1, \ldots, C_n$, with $1 \leq n \leq 3$, such that $z_1$ and $z_2$ lie on the bottom of $C_1$.  Let $z_3$ be a double zero on the top of $C_1$.  Consider the saddle connection $\sigma_1$ from $z_1$ to $z_3$.  We can take $\sigma_1$ to have length less than two because the total height of all of the cylinders is one and $\sigma_1$ connects the top and bottom of a single cylinder.  Then $\sigma_1$ lies on the top of a cylinder $C_1'$ in a different cylinder decomposition $\mathcal{C}'$ of $(X,\omega)$ because $(X,\omega)$ satisfies topological dichotomy by Lemma \ref{H2211TopDichot}.  Since the total order of the zeros on the top of every cylinder must be even, $z_2$ must also lie along the top of $C_1'$.  Furthermore, $\mathcal{C}'$ consists of at most two cylinders because the total order of the zeros along the top of one of the cylinders is four.  This implies that the total height of the cylinders in the decomposition $\mathcal{C}'$ is at most $\varepsilon$ because $\sigma$ is transverse to $\sigma_1$ and $\sigma$ must join the top of $C_2$ to the bottom of $C_1$.  The total area of the cylinders is still one, so the waist length of the cylinders in $\mathcal{C}'$ must be at least $1/\varepsilon$.  Act by the Teichm\"uller geodesic flow so that the waist of the cylinders in $\mathcal{C}'$ is reduced to one and the total height of the cylinders is expanded to one.  In the process of the expansion and contraction, the saddle connection $\sigma_1$ of length at most two is contracted to length at most $2\varepsilon$.  Since this argument holds for all $\varepsilon > 0$, $\sigma_1$ can be contracted to a point resulting in a zero of order three.  By Lemma \ref{Gen4CannotDeg}, the surface will not degenerate and we get a surface generating a Teichm\"uller disc in a stratum that does not contain a Teichm\"uller disc.  This contradiction completes the proof.
\end{proof}

\begin{lemma}
\label{H21111TopDichot}
If $(X,\omega)$ generates a Teichm\"uller disc $D$ in $\mathcal{H}(1^4,2) \cap \mathcal{D}_4(1)$, then $(X,\omega)$ satisfies topological dichotomy.
\end{lemma}

\begin{proof}
By Theorem \ref{RankOneImpCP}, $(X,\omega)$ is completely periodic.  We show that every saddle connection between two zeros must lie in a periodic foliation.  First, consider any saddle connection $\sigma$ from the double zero, denoted by $z$, to any other zero, denoted by $z'$.  By contradiction, if $\sigma$ does not lie in a periodic foliation, then we can act on it by the Teichm\"uller geodesic flow so that it contracts at the maximal rate and choose a subsequence of times $\{t_n\}_n$ as in the proof of Lemma \ref{CPZeroConv} such that $G_{t_n} \cdot (X, \omega)$ converges to a surface $(X', \omega')$, where $\omega'$ is holomorphic.  The surface cannot degenerate to a lower genus surface by Lemma \ref{Gen4CannotDeg}, so $\sigma$ must degenerate to a point resulting in a zero of order strictly greater than two.  However, there are no such Teichm\"uller discs in a stratum with a zero of order strictly greater than two by Lemma \ref{NoDisksSomeGen4Strata}.  This contradiction implies a saddle connection $\sigma$, which does \emph{not} lie in a periodic foliation, can only lie between two of the simple zeros.

Let $z_1, \ldots, z_4$ denote the simple zeros of $\omega$ and let $z_5$ denote the double zero.  By contradiction, let $\sigma$ denote the saddle connection that does not lie in a periodic foliation.  In light of the argument above, let $\sigma$ be a saddle connection from $z_1$ to $z_3$ and let it have length $\varepsilon$, while there is a cylinder decomposition $\mathcal{C}$ such that the cylinders have circumference one.  The zeros $z_1$ and $z_5$ cannot lie on the top or bottom of the same cylinder in $\mathcal{C}$ because this would imply $\mathcal{C}$ consists of two cylinders, one of which has height less than $\varepsilon$.  As $\varepsilon$ tends to zero, the resulting sequence of surfaces would have to converge to $(M_4, \omega_{M_4})$ because this generates the only Teichm\"uller disc in $\mathcal{D}_4(1)$ in a lower stratum.  The cylinder decomposition would only consists of one cylinder, which contradicts the cylinder decomposition of $(M_4, \omega_{M_4})$.  With respect to $\mathcal{C}$, and without loss of generality, let $z_1$ and $z_5$ be on the top and bottom of a cylinder.  Consider the shortest saddle connection $\sigma_1$ from $z_1$ to $z_5$, which has length less than two.  Then $\sigma_1$ lies in a periodic foliation by the argument above, and in particular, it is not parallel to $\sigma$.  Since the total order of the zeros on the bottom of a cylinder must be even, the leaf of the periodic foliation containing $\sigma_1$ must contain at least one other simple zero.  We show that this will lead to a contradiction.

Let $C_1$ denote the cylinder with the saddle connection $\sigma_1$ on its bottom.  Then the bottom of $C_1$ must also contain either $z_2$, $z_3$, or $z_4$.  We only consider $z_2$ and $z_3$ here because the argument for $z_4$ will be identical to the argument for $z_2$.  First assume that the bottom of $C_1$ contains the zero $z_3$.  Then the saddle connection $\sigma$ cannot be a subset of the bottom of $C_1$ because it does not lie in a period foliation, so it must traverse the heights of every cylinder in the cylinder decomposition before it reaches $z_3$.  This implies that the total height of the cylinders is less than $\varepsilon$.  By contracting the waist of the cylinders in this direction to unit length, $\sigma_1$ contracts to a saddle connection of length $2\varepsilon$.  Since this argument holds for all $\varepsilon$, we converge to a degenerate surface with a zero of order at least three and reach a contradiction with Lemma \ref{NoDisksSomeGen4Strata}.

Next we assume that the bottom of $C_1$ contains the zeros $z_1$, $z_5$, and $z_2$.  In this case, it is clear that the surface decomposes into at most two cylinders.  Furthermore, $(X,\omega)$ cannot consist of exactly one cylinder because $z_3$ would lie on its top and that would imply that the height of $C_1$ is $\varepsilon$ while its circumference is one, which would contradict that the area of the surface is one.  Though there are two cylinders, this argument shows that one of them, say $C_2$ has height $\varepsilon$ because both cylinders have $z_1$ and $z_3$ on different sides and since the distance between them is $\varepsilon$, the height of the cylinder must be less than $\varepsilon$.  Lemma \ref{NoTDInH2211D41} implies that both $z_1$ and $z_3$ must converge to $z_2$ and $z_4$, simultaneously and respectively, (though we make no claims about the rates at which this happens) because otherwise we would have a contradiction with Lemma \ref{NoTDInH2211D41}.  However, as we consider the sequence of surfaces resulting from letting $\varepsilon$ vary over a sequence decreasing to zero, we get that the cylinder $C_2$ must vanish in the limit so that $(X', \omega')$ lies in $\mathcal{H}(2^3)$ and consists of one cylinder.  However, this directly contradicts the fact that $(M_4, \omega_{M_4})$ decomposes into exactly two cylinders in every direction.  This contradiction that implies that any Teichm\"uller disc satisfying the assumptions of this lemma is generated by a surface satisfying topological dichotomy.
\end{proof}

\begin{lemma}
\label{NoTDInH21111D41}
There are no Teichm\"uller discs contained in $\mathcal{H}(1^4,2) \cap \mathcal{D}_4(1)$.
\end{lemma}

\begin{proof}
As in the proof of Lemma \ref{NoTDInH2211D41}, we show that if $(X,\omega) \in \mathcal{H}(1^4,2)$ generates a Teichm\"uller disc in $\mathcal{D}_4(1)$, then $(X,\omega)$ is uniformly completely periodic.  By \cite{SmillieWeissCharLattice}, $(X,\omega)$ is a Veech surface and by \cite{MollerShimuraTeich}, there are no Veech surfaces in $\mathcal{H}(1^4,2)$ that generate a Teichm\"uller disc in $\mathcal{D}_4(1)$.

By contradiction, assume that there exists a surface $(X,\omega) \in \mathcal{H}(1^4,2)$ generating a Teichm\"uller disc in $\mathcal{D}_4(1)$.  By \cite{MollerShimuraTeich}, $(X,\omega)$ is not a Veech surface and by Lemma \ref{H21111TopDichot}, $(X,\omega)$ satisfies topological dichotomy.  As in the proof of the previous lemma, only simple zeros can converge.  Let $z_1, \ldots, z_4$ denote the simple zeros, and let $z_5$ denote the double zero.  Without loss of generality, let $\sigma$ be a saddle connection between two such zeros $z_1$ and $z_3$ of length $\varepsilon > 0$.  By Lemma \ref{H21111TopDichot}, consider a cylinder decomposition $\mathcal{C}$ of $(X,\omega)$ such that $z_1$ and $z_3$ lie on the bottom of $C_1 \in \mathcal{C}$.  It is possible to choose $C_1$ so that the double zero $z_5$ lies on its top because $z_5$ must lie at the top of some cylinder, and the cylinder will have simple zeros on its bottom of distance $\varepsilon$.  It was noted in the previous proof that there must always be two pairs of simple zeros, say $z_1$, $z_3$ and $z_2$, $z_4$, such that each zero in the pair has distance $\varepsilon$ from the other zero in the pair because the pairs of simple zeros must converge to double zeros simultaneously.  Consider the saddle connection $\sigma'$ from $z_1$ to $z_5$.  We can take $\sigma'$ to have length less than two because the total height of all the cylinders is one and $\sigma'$ connects the top and bottom of a single cylinder.  Then $\sigma'$ lies on the top of a cylinder $C_1'$ in a different cylinder decomposition $\mathcal{C}'$ of $(X,\omega)$.  The top of $C_1'$ must contain exactly one of $z_2$, $z_3$ or $z_4$ because the total order of the zeros on the top of $C_1$ is even.  If it contains $z_3$, then the total height of the cylinders in $\mathcal{C}'$ is less than $\varepsilon$.  We claim that if it contains either $z_2$ or $z_4$, then the total height of the cylinders in $\mathcal{C}'$ is at most $2\varepsilon$.  To see this, note that $z_1$, $z_5$, and say $z_2$, without loss of generality, lie on the top of $C_1'$.  Then $z_3$ and $z_4$ must lie on the bottom of $C_1'$.  Hence, $C_1'$ has height at most $\varepsilon$.  If the height of $C_2'$ is bounded away from zero by a constant $C > 0$, for all $\varepsilon > 0$, then as $\varepsilon$ tends to zero, we get a sequence converging to a surface that must be $(M_4, \omega_{M_4})$, but with a cylinder decomposition consisting of exactly one cylinder.  This contradicts the fact that every cylinder decomposition of $(M_4, \omega_{M_4})$ has two cylinders, so $C_2'$ must have height $\varepsilon'$.

We abuse notation and set $\varepsilon = \max(\varepsilon, \varepsilon')$.  Furthermore, $\mathcal{C}'$ consists of at most two cylinders because the total order of the zeros along the top of one of the cylinders is four.  The total area of the cylinders is one, so the waist length of the cylinders in $\mathcal{C}'$ must be at least $1/(2\varepsilon)$.  Act by the Teichm\"uller geodesic flow so that the circumference of the cylinders in $\mathcal{C}'$ is reduced to one and the total height of the cylinders is expanded to one.  In the process of the expansion and contraction, the saddle connection $\sigma'$ of length at most two is contracted to length at most $4\varepsilon$.  Since this argument holds for all $\varepsilon > 0$, $\sigma'$ can be contracted to a point resulting in a zero of order three.  By Lemma \ref{Gen4CannotDeg}, the surface will not degenerate and we get a surface generating a Teichm\"uller disc in a stratum that does not contain a Teichm\"uller disc.  This shows $(X,\omega)$ must be uniformly completely periodic and yields the desired contradiction.
\end{proof}

Let 
$$H_s = \left[ \begin{array}{cc}
1 & s \\
0 & 1 \end{array} \right]$$
denote the horocycle flow.

\begin{lemma}
\label{PrinStratTopDichot}
If $(X,\omega)$ generates a Teichm\"uller disc $D$ in $\mathcal{H}(1^6) \cap \mathcal{D}_4(1)$, then $(X,\omega)$ satisfies topological dichotomy.
\end{lemma}

\begin{proof}
Let $z_i$, for $1 \leq i \leq 6$, denote the simple zeros of $\omega$.  Assume by contradiction that $(X,\omega)$ does not satisfy topological dichotomy.  Let $\sigma_1$ be a saddle connection from $z_1$ to $z_2$, without loss of generality, that does not lie in a periodic foliation.  If $\sigma_1$ converges to a point, then the other zeros must also converge to each other in pairs because there are no Teichm\"uller discs in $\mathcal{D}_4(1)$ in any lower stratum other than $\mathcal{H}(2^3)$ by Lemmas \ref{NoTDInH2211D41} and \ref{NoTDInH21111D41}.  Setting notation, let $z_3$ and $z_5$ converge to $z_4$ and $z_6$, respectively.  Let $d_{\omega}(\cdot, \cdot)$ denote flat length with respect to $\omega$.  We assume that 
$$\varepsilon = \max_i\{d_{\omega}(z_i, z_{i+1})\},$$
and the length of $\sigma_1$ is at most $\varepsilon$, such that $(X,\omega)$ admits a cylinder decomposition into cylinders of unit circumference.  We define a sequence of surfaces $(X_n, \omega_n)$ converging to $(X', \omega') = (M_4, \omega_{M_4})$ letting $\varepsilon = 1/n$.  For each $(X_n, \omega_n)$ fix a cylinder decomposition $\mathcal{C}_n$ such that the cylinders have unit circumference.  Pass to a subsequence, such that $\mathcal{C}_n$ has the same number of cylinders as $\mathcal{C}_m$, for all $n,m \geq 0$.

First we claim that the cylinder decompositions $\mathcal{C}_n$ do not consist of exactly one cylinder.  Assume by contradiction that it does consist of exactly one cylinder.  Since $\sigma_1$ does not lie in a periodic foliation, $\sigma_1$ must traverse the height of the cylinder.  However, this would imply that the height of the cylinder is at most $1/n$ while the circumference is one, which contradicts the fact that the area of each surface in the sequence is one.

Secondly, we claim that the cylinder decompositions $\mathcal{C}_n$ do not consist of exactly two cylinders.  To see this, we use the same argument as above to see that if there are two cylinders, then one of them must have height at most $1/n$.  As we let $n$ tend to infinity, the surface converges to a surface $(X', \omega')$, which must have a single cylinder because the height of one of the two cylinders in $(X_n, \omega_n)$ converged to zero.  However, $(M_4, \omega_{M_4})$ decomposes into two cylinders in every periodic direction so we have a contradiction that implies that there cannot be two cylinders.

Finally, we assume that for all $n$, $\mathcal{C}_n$ consists of exactly three cylinders, the maximum possible by Lemma \ref{MaxCyls}.  The saddle connection $\sigma_1$ cannot lie in the foliation of the cylinder of $\mathcal{C}_n$ because it does not lie in a periodic foliation.  Therefore, $z_1$ and $z_2$ lie on the top and bottom of a cylinder, say $C_3$.  By the assumption that there are three cylinders, there must be another pair of zeros between the top and bottom of $C_3$.  If not, the total height of the three cylinders would be at most $3/n$, which would contradict the assumption that the surface has area one, for large $n$.  Therefore, we have that the saddle connection $\sigma_3$ of length at most $\varepsilon$ lies on the top of the cylinder $C_1$.  This arrangement of the zeros must hold for all $n$ in the sequence $\{(X_n, \omega_n)\}_n$ because this argument did not depend on the value of $n$.

Now we make an elementary observation.  If we consider the action of $H_s$ on $(X_n, \omega_n)$, then the heights and boundaries of the three cylinders in $\mathcal{C}_n$ are preserved, though the cylinders themselves are twisted (in the sense of Dehn twists).  This implies that the saddle connection $\sigma_3$ is preserved under the action of $H_s$, while the distance between $z_1$ and $z_2$ can be increased to some constant bounded away from zero.  Therefore, for each $n$, there exists a number $s_n$, where $0 < s_n < 1$ such that the sequence $\{H_{s_n} \cdot (X_n, \omega_n)\}_n$ converges to a surface which does not degenerate because the cylinders have circumference one.  However, at least one pair of simple zeros remain simple zeros in the limit, while at least one pair of simple zeros converge to a double zero.  This contradicts either Lemma \ref{NoTDInH2211D41} or \ref{NoTDInH21111D41}, and implies that the arrangement of the cylinders and the zeros described above for the case where $\mathcal{C}_n$ consists of three cylinders cannot occur.  However, since this was the only remaining potentially admissible arrangement of the zeros in such a cylinder decomposition, we have a contradiction which implies $(X,\omega)$ satisfies topological dichotomy.
\end{proof}

\begin{lemma}
\label{NoTDPrinStratGen4}
There are no Teichm\"uller discs contained in $\mathcal{H}(1^6) \cap \mathcal{D}_4(1)$.
\end{lemma}

\begin{proof}
The idea of this proof is identical to Lemmas \ref{NoTDInH2211D41} and \ref{NoTDInH21111D41}.  Assume by contradiction that such a surface $(X,\omega)$ exists.  By Lemma \ref{PrinStratTopDichot}, $(X,\omega)$ satisfies topological dichotomy, but \cite{MollerShimuraTeich} implies that $(X,\omega)$ is not uniformly completely periodic.  We show that there is a sequence of surfaces in the Teichm\"uller disc $D$ generated by $(X,\omega)$ converging to a surface in a stratum other than $\mathcal{H}(2^3)$.

As in the proof of the previous lemma, let $z_i$ denote the simple zeros of $\omega$, for $1 \leq i \leq 6$.  By Theorem \ref{CPImpVeech}, there is a sequence of surfaces $\{(X_n, \omega_n)\}_n$ in $D$ converging to $(M_4, \omega_{M_4})$.  For $i = 1,3,5$, we can assume that $z_i$ converges to $z_{i+1}$ in this sequence because there are no Teichm\"uller discs in $\mathcal{D}_4(1)$ in any lower stratum other than $\mathcal{H}(2^3)$ by Lemmas \ref{NoTDInH2211D41} and \ref{NoTDInH21111D41}.  Let $d_{\omega}(\cdot, \cdot)$ denote flat length with respect to $\omega$.  As before, assume that
$$\varepsilon = \max_i\{d_{\omega}(z_i, z_{i+1})\},$$
and $(X,\omega)$ admits a cylinder decomposition into cylinders of unit circumference.  We define a sequence of surfaces $\{(X_n, \omega_n)\}_n$ converging to $(X', \omega') = (M_4, \omega_{M_4})$ by letting $\varepsilon = 1/n$.  For each $(X_n, \omega_n)$ fix a cylinder decomposition $\mathcal{C}_n$ such that the cylinders have unit circumference and pass to a subsequence, such that $\mathcal{C}_n$ has the same number of cylinders as $\mathcal{C}_m$, for all $n,m \geq 0$.

First note that $\mathcal{C}_n$ cannot contain exactly one cylinder, for all $n$, because the sequence converges to a surface that decomposes into two cylinders in every periodic direction.  If we assume that $\mathcal{C}_n$ splits into three cylinders, then there are two possible arrangements of the zeros.  Either one or more of the saddle connections of length at most $\varepsilon$ lies between the top and bottom of a cylinders, or, after renaming the zeros, $\sigma_i$ lies on the top of $C_i$, for $1 \leq i \leq 3$.  If one or more of the saddle connections lies across a cylinder, then we have the same arrangement as in the proof of Lemma \ref{PrinStratTopDichot}: $C_1$ has the saddle connection $\sigma_3$ along its top, and $C_3$ has $z_1$ and $z_3$ on its bottom and $z_2$ and $z_4$ on its top.  In fact, to exclude the possibility of this case from occurring, it suffices to use the ``horocycle trick'' from the previous lemma to get a contradiction.  Therefore, we are left with the other case where $\sigma_i$ lies on the top of $C_i$, for all $i$.  Since we know that the limit $(M_4, \omega_{M_4})$ decomposes into two cylinders, one of the three cylinders, say $C_3$ must have height $h_n$ converging to zero.  Let $\sigma_2$ lie on the bottom of $C_3$, and $\sigma_3$ lie on the top of $C_3$.  Let $(X_n', \omega_n') := H_{s_n} \cdot (X_n, \omega_n)$.  For each $n$, there is a number $s_n$ satisfying $0 \leq s_n \leq 1$ such that the 
$$d_{\omega_n'}(z_3, z_5) = h_n.$$
As $n$ tends to infinity, the limit must lie in the stratum $\mathcal{H}(2,4)$, which does not contain a Teichm\"uller disc in $\mathcal{D}_4(1)$ by Lemma \ref{NoDisksSomeGen4Strata}.  This contradiction implies that the the cylinder decomposition $\mathcal{C}_n$ must consist of exactly two cylinders for all $n$.

Finally, we assume that $\mathcal{C}_n$ contains exactly two cylinders, for all $n$.  Consider sufficiently large $n$ so that $1/n << 1$.  The heights of the cylinders must be bounded away from zero so that the sequence converges to a surface with two cylinders.  This implies that the three saddle connections $\sigma_i$, for $1 \leq i \leq 3$, of length less than $1/n$ lie on the boundaries of the cylinders, $C_1$ and $C_2$, i.e. the short saddle connections are parallel.  Consider a straight trajectory $\gamma$ of length less than two from $z_1$ on the bottom of $C_1$ to itself on the top of $C_2$.  We permit $\gamma$ to pass through another zero.  By Lemma \ref{PrinStratTopDichot}, this saddle connection determines a periodic foliation, thus, a cylinder decomposition $\mathcal{C}_n'$.  We claim that the total height of the cylinders in $\mathcal{C}_n'$ is at most $3/n$.  We consider three cases.  If $\mathcal{C}_n'$ consists of exactly one cylinder, then this is clear because there is a saddle connection of length less than $1/n$ transverse to the foliation.  If $\mathcal{C}_n'$ consists of exactly two cylinders, then there is at least one cylinder of height at most $1/n$.  In fact, both cylinders must have height at most $1/n$ because each cylinder has four zeros on one side and two zeros on the other, which implies that one of the saddle connections of length at most $1/n$ must traverse the heights of both cylinders.  Finally, if $\mathcal{C}_n'$ consists of exactly three cylinders, then each cylinder has two zeros on each side.  Since $\sigma_i$ does not lie in the foliation of $\mathcal{C}_n'$ for all $i$, every cylinder has height at most $1/n$.

If we form a new sequence of surfaces in $D$ by acting on the foliation $\mathcal{C}_n'$ by the Teichm\"uller geodesic flow so that the cylinders have unit circumference, then the curve $\gamma$ must have length at most $6/n$ in this new sequence.  However, as $n$ tends to infinity, this would imply that a curve from $z_1$ to itself contracts to a point.  This forces the surface to degenerate because $z_1$ can no longer be a zero in the limit.  This directly contradicts Lemma \ref{Gen4CannotDeg} and implies that there is no surface generating a Teichm\"uller disc in $\mathcal{D}_4(1)$ in the principal stratum.
\end{proof}

We summarize Lemmas \ref{NoDisksSomeGen4Strata}, \ref{NoTDInH2211D41}, \ref{NoTDInH21111D41}, and \ref{NoTDPrinStratGen4} in the following theorem.

\begin{theorem}
\label{Gen4Class}
The Ornithorynque $(M_4, \omega_{M_4})$ generates the only Teichm\"uller disc in $\mathcal{D}_4(1)$.
\end{theorem}

\bibliography{fullbibliotex}{}

\end{document}